\documentclass[10pt,a4paper]{article}
\usepackage[latin1]{inputenc}
\usepackage{amsmath}
\usepackage{authblk}
\usepackage{amsfonts}
\usepackage{amssymb}
\usepackage{graphicx}
\usepackage{amsthm}
\usepackage{xcolor}
\usepackage{subcaption}
\usepackage{caption}
\usepackage{mathtools}
\usepackage[normalem]{ulem}%to \sout
\usepackage{soul}
\author[1]{Yoshihito Kazashi}
\author[1]{Fabio Nobile}
\author[1]{Eva Vidli\v{c}kov\'a}
\affil[1]{CSQI, \'{E}cole Polytechnique F\'{e}d\'{e}rale de Lausanne, Switzerland}
\title{Stability properties of a projector-splitting scheme for dynamical low rank approximation of random parabolic equations}

\newtheorem{theorem}{Theorem}[section]
\newtheorem{lemma}[theorem]{Lemma}
\newtheorem{remark}[theorem]{Remark}
\newtheorem{proposition}[theorem]{Proposition}

\newtheorem{definition}[theorem]{Definition}
\newtheorem{algorithm}[theorem]{Algorithm}

\DeclareMathOperator{\spn}{span}
\DeclareMathOperator{\kernel}{ker}
\newcommand{\dx}{\mathrm{d}}
\newcommand{\R}{\mathbb{R}}
\newcommand{\N}{\mathbb{N}}

\newcommand{\lr}{L^2_{\rho}}
\newcommand{\lhr}{L^2_{\hat{\rho}}}
\newcommand{\lrz}{L^2_{\rho,0}}
\newcommand{\lhrz}{L^2_{\hat{\rho},0}}
\newcommand{\tr}{_{\mathrm{true}}}
\newcommand{\rdx}{\dx \rho}
\newcommand{\intG}[1]{\int_{\Gamma}{#1} \;\rdx}
\newcommand{\intD}[1]{\int_{D}{#1} \;\dx x}
\newcommand{\eR}[1]{1\leq {#1} \leq R}
\newcommand{\Ldet}{\mathcal{L}_{\mathrm{det}}}
\newcommand{\Cdet}{C_{\mathrm{det}}}
\newcommand{\Cstoch}{C_{\mathrm{stoch}}}
\newcommand{\Lstoch}{\mathcal{L}_{\mathrm{stoch}}}

\newcommand*{\ykcst}[1]{\relax\ifmmode\text{\textcolor{blue}{\sout{\ensuremath{#1}}}}\else\textcolor{blue}{\sout{#1}}\fi}

\newcommand{\Ccoerc}{C_{\mathcal{L}}}
\newcommand{\Cbound}{C_{\mathcal{B}}}
\newcommand{\Cinv}{C_{\mathrm{I}}}
\newcommand{\Cemb}{C_{\mathrm{P}}}
\newcommand{\lrV}{L^2_{\rho}(\Gamma;V)}
\newcommand{\lrVz}{L^2_{\rho,0}(\Gamma;V)}
\newcommand{\lrH}{L^2_{\rho}(\Gamma;H)}
\newcommand{\lrHz}{L^2_{\rho,0}(\Gamma;H)}
\newcommand{\lrVdual}{L^2_{\rho}(\Gamma;V')}
\newcommand{\corr}[1]{{\textcolor{black}{#1}}}

\newcommand{\astoch}{a_{\mathrm{stoch}}}
\newcommand{\hr}{_{\hat{\rho}}}
\newcommand{\Pro}{\mathcal{P}}

\newtheorem{propap}{Proposition}

\begin{document}
	\maketitle
	%\tableofcontents
	\begin{abstract}
	We consider the Dynamical Low Rank (DLR) approximation of random parabolic equations and propose a class of fully discrete numerical schemes.
	Similarly to the continuous DLR approximation, our schemes are shown to satisfy a discrete variational formulation.
	By exploiting this property, we establish stability of our schemes: 
	we show that our explicit and semi-implicit versions are conditionally stable under a ``parabolic'' type CFL condition which does not depend on the smallest singular value of the DLR solution; whereas our implicit scheme is unconditionally stable. 	Moreover, we show that, in certain cases, the semi-implicit scheme can be unconditionally stable if the randomness in the system is sufficiently small.
	Furthermore, we show that these schemes can be interpreted as projector-splitting integrators and are strongly related to the scheme proposed in \cite{Lubich14,Lubich15}, %by Lubich et al. [BIT Num. Math., 54:171-188, 2014; SIAM J. on Num. Anal., 53:917-941, 2015]
	 to which our stability analysis applies as well. The analysis is supported by numerical results showing the sharpness of the obtained stability conditions.
\end{abstract}



	\section{Introduction}
	\label{intro}
	Many physical and engineering applications are modeled by time-dependent partial differential
	equations (PDEs) with input data often subject to uncertainty due to measurement
	errors or insufficient knowledge. These uncertainties can be often described by means of
	probability theory by introducing a set of random variables into the system. In the present work, we consider a random evolutionary PDE
	\begin{equation}\label{eq:intropr}
		\dot{u} + \mathcal{L}(u) = f
	\end{equation}
	with random initial condition, random forcing term and a random linear elliptic operator $\mathcal{L}$. Many of the numerical methods used to approximate such problems, require evaluating the, possibly expensive, model in many random parameters. In this regard, the use of reduced order models (e.g. Proper orthogonal decomposition \cite{Berkooz93,Carlberg11} or generalized Polynomial chaos expansion \cite{Wiener38,Xiu02,Lemaitre10,Cohen11,Nobile09}) is of a high interest. 
	
	When the dependence of the solution on the random parameters significantly changes in time, the use of time-varying bases is very appealing. In the present work, we consider the dynamical low rank (DLR) approximation (see \cite{Sapsis09,Musharbash15,Cheng13,Lubich08,Meyer00,Koch06,Koch07,Feppon18}) which allows both the deterministic and stochastic basis functions to evolve in time while exploiting the structure of the differential equation. An extension to tensor differential equations was proposed in \cite{Koch10,Lubich13}. The DLR approximation of the solution is of the form 
	\begin{equation}\label{eq:introDLRexp}
		u(t) = \bar{u}(t) + \sum_{j=1}^R U_j(t) Y_j(t),\qquad t\in(0,T],
	\end{equation}
	where $R$ is the \emph{rank} of the approximation and is kept fixed in time, $\bar{u}(t) = \mathbb{E}[u(t)]$ is the mean value of the DLR solution, $\{U_j(t)\}_{j=1}^R$ is a time dependent set of deterministic basis functions, $\{Y_j(t) \}_{j=1}^R$ is a time dependent set of zero mean stochastic basis functions.  By suitably projecting the residual of the differential equation one can derive evolution equations for the mean value $\bar{u}$ and the deterministic and stochastic modes $\{U_j\}_{j=1}^R,\,\{Y_j\}_{j=1}^R$ (see \cite{Sapsis09,Koch07}), which in practice need to be solved numerically. An efficient and stable discretization scheme is therefore of a high interest.
	
	In \cite{Sapsis09,Koch07}, Runge-Kutta methods of different orders were applied directly to the system of evolution equations for the deterministic and stochastic basis functions. In the presence of small singular values in the solution, the system of evolution equations becomes stiff as an inversion of a singular or nearly-singular matrix is required to solve it. Applying standard explicit or implicit Runge-Kutta methods leads to instabilities (see \cite{Kieri16}). In this respect, the projector-splitting integrators (proposed in \cite{Lubich14,Lubich15} and applied in e.g. \cite{Einkemmer18,Einkemmer19}) are very appealing. In \cite{Kieri16}, the authors showed that when applying the projector-splitting method for matrix differential equations one can bound the error independently of the size of the singular values, under the assumption that $f - \mathcal{L}$ maps onto the tangent bundle of the manifold of all $R$-rank functions up to a small error of magnitude $\varepsilon$. A limitation of their theoretical result, as the authors point out, is that it requires a Lipschitz condition on $f - \mathcal{L}$ and is applicable to discretized PDEs only under a severe condition $\triangle t L \ll 1$ where $\triangle t$ is the step size and $L$ is the Lipschitz constant, even for implicit schemes. Such condition is, however, not observed in numerical experiments. In \cite{Kieri18}, the authors proposed projected Runge-Kutta methods, where following a Runge-Kutta integration, the solution first leaves the manifold of $R$-rank functions by increasing its rank, and then is retracted back to the manifold. Analogous error bounds as in \cite{Kieri16} are obtained, also for higher order schemes, under the same $\varepsilon$-approximability condition on $f - \mathcal{L}$ and under a restrictive parabolic condition on the time step. 
	
	In this work we propose a class of numerical schemes to approximate the evolution equations for the mean, the deterministic basis and the stochastic basis, which can be of explicit, semi-implicit or implicit type. Although not evident at first sight, we show that the explicit version of our scheme can be reinterpreted as a projector-splitting scheme, whenever the discrete solution is full-rank, and is thus equivalent to the scheme from \cite{Lubich14,Lubich15}. However, our derivation allows for an easy construction of implicit or semi-implicit versions.
	
	The main goal of this work is to prove the stability of the proposed numerical schemes for a parabolic problem \eqref{eq:introDLRexp}. We first show that the continuous DLR solution satisfies analogous stability properties as the weak solution of the parabolic problem \eqref{eq:intropr}. We then analyze the stability of the fully discrete schemes. Quite surprisingly, the stability properties of both the discrete and the continuous DLR solutions do not depend on the size of their singular values, even without any $\varepsilon$-approximability condition on $f - \mathcal{L}$.  The implicit scheme is proven to be unconditionally stable. This improves the stability result which could be drawn from the error estimates derived in \cite{Kieri16}. The explicit scheme remains stable under a standard parabolic stability condition between time and space discretization parameters for an explicit propagation of parabolic equations. The semi-implicit scheme is generally only conditionally stable under again a parabolic stability condition, and becomes unconditionally stable under some restrictions on the size of the randomness of the operator. As an application of the general theory developed in this paper, we consider the case of a  heat equation with a random diffusion coefficient. We dedicate a section to particularize the numerical schemes and the corresponding stability results to this problem. The semi-implicit scheme turns out to be always unconditionally stable if the diffusion coefficient depends affinely on the random variables. The sharpness of the obtained stability conditions on the time step and spatial discretization is supported by the numerical results provided in the last section.
	
	A big part of the paper is dedicated to proving a variational formulation of the discretized DLR problem, analogous to the variational formulation of the continuous DLR problem (see \cite[Prop. 3.4]{Musharbash15}). Such formulation is a key for showing the stability properties and, as we believe, might be useful for some further analysis of the proposed discretization schemes. It as well applies to the projector-splitting integrator from \cite{Lubich14,Lubich15} provided the solution remains full rank at all time steps. However, in the rank-deficient case, our schemes may result in different solutions. We dedicate a subsection to show that a rank-deficient solution obtained by our scheme still satisfies a suitable discrete variational formulation and consequently has the same stability properties as the full-rank case. 
	
	The outline of the paper is the following: in Section~\ref{sec:prstatement} we introduce the problem and basic notation; in Section \ref{sec:DLR} we describe the DLR approximation and recall its geometrical interpretation with variational formulation. In Section \ref{sec:disc} we describe the discretization of the DLR method and propose three  types of time integration schemes. We then derive a variational formulation for the discrete DLR solution and show its reinterpretation as a projector-splitting scheme. Section \ref{sec:stabest} is dedicated to proving the stability properties of both continuous and discrete DLR solution. In Section \ref{sec:heateq}, we analyze the case of a heat equation with random diffusion coefficient and random initial condition. Finally in Section \ref{sec:numres} we present several numerical tests that support the derived theory. Section \ref{sec:conclusions} draws some conclusions.
	
	\section{Problem statement}\label{sec:prstatement}
	
	We start by introducing some notation. Let $(\Gamma, \mathcal{F}, \rho)$ be a probability space. % where $\Gamma \subset \R^M$ for $M\in\mathbb{N}$, $B(\Gamma)$ denotes the Borel $\sigma$-algebra on $\Gamma$ and $\rho$ is a probability measure.  
	Consider the Hilbert space $\lr = \lr(\Gamma)$ of real valued random variables on $\Gamma$ with bounded second moments, with associated scalar product $\langle v,w\rangle_{\lr} = \intG{vw}$ and norm $\|v\|_{\lr} = \sqrt{\langle v,v\rangle_{\lr}}$. Consider as well two separable Hilbert spaces $H$ and $V$ with scalar products $\langle \cdot, \cdot \rangle_H$, $\langle \cdot, \cdot \rangle_V$, respectively. Suppose that $H$ and $V$  form a Gelfand triple $(V,H,V')$, i.e.\ $V$ is a dense subspace of $H$ and the embedding $V\hookrightarrow H$ is continuous with a continuity constant $\Cemb>0$.
	Let $\lrV,\, \lrH$ be the Bochner spaces of square integrable $V$ (resp. $H$) valued functions on $\Gamma$ with scalar products 
	\begin{align*}
		\langle v,w\rangle_{H,\lr} &= \intG{\langle v,w \rangle_H},\qquad v,w\in \lrH\\
		\langle v,w\rangle_{V,\lr} &= \intG{\langle v,w \rangle_V},\qquad v,w\in \lrV,
	\end{align*}
	respectively.
	Then, $(\lrV,\; \lrH,\;\lrVdual)$ is a Gelfand triple as well (see e.g. \cite[Th. 8.17]{Leoni17}), and we have \begin{equation}\label{eq:embed_const}
		\|v\|_{H,\lr}\leq \Cemb \|v\|_{V,\lr}\qquad\forall v\in \lrV.
	\end{equation}
	We define the mean value of a random variable $v$ as $\mathbb{E}[v] = \intG{v}$, where the integral here denotes the Bochner integral in a suitable sense, depending on the co-domain of the random variable considered. In what follows, we will use the notation $\bar{v} = \mathbb{E}[v]$ and $v^* := v - \bar{v}$.  Moreover, we let $(\cdot, \cdot)_{V'V,\lr}$ denote the dual pairing between $\lrVdual$ and $\lrV$: \[(\mathcal{K}, v)_{V'V,\lr} = \int_{\Gamma}\big(\mathcal{K}(\omega), v(\omega)\big)_{V'V}\:\dx \rho(\omega),\quad \mathcal{K}\in \lrVdual,\; v\in \lrV.\]
	With this notation at hand, we now consider a random operator $\mathit{L}$ with values in the space of linear bounded operators from $V$ to $V'$ that is uniformly bounded and  coercive, i.e. a Borel measurable function
	\begin{align*}
		\hspace{3cm}\mathit{L}: \quad&\Gamma & &\to & &\mathfrak{L}(V, V')\hspace{3cm}\\
		& \omega & &\mapsto & &\mathit{L}(\omega)
	\end{align*}
	%\fn{\yk{YK: $\mathit{L}\colon \Gamma\to \mathfrak{L}(V, V') $ needs to be measurable in a suitable sense, so that  $L(\cdot)u(\cdot)\colon\Gamma \to V'$ is ($\mathcal{F}/\text{Borel}(V')$) measurable.}}
	%for which the bilinear form $(\mathit{L}(\omega)\cdot,\cdot)_{V'V}$ is bounded and coercive uniformly over $\Gamma$ with coercivity and continuity constant $\Ccoerc$ and $\Cbound$, respectively, i.e.
	such that there exist $\Ccoerc, \Cbound > 0$ satisfying
	\begin{align}
		\big( \mathit{L}(\omega) v,v\big)_{V' V} &\geq \Ccoerc \|v\|_{V}^2 & &\qquad \forall \omega\in\Gamma,\;\forall v\in V,\label{eq:coerc_const}\\
		\big( \mathit{L}(\omega) v,w\big)_{V' V} &\leq \Cbound \|v\|_V \|w\|_{V} & &\qquad \forall\omega\in\Gamma,\;\forall v,w\in V.\label{eq:bound_const}
	\end{align} 
	Associated to the random operator $\mathit{L}$, we introduce the operator $\mathcal{L}$, defined as
	\begin{align*}
		\mathcal{L}:\quad &\lr(\Gamma; V) & &\to & &\lr(\Gamma;V')\\
		&u& &\mapsto & &\mathcal{L}(u) : \quad \mathcal{L}(u)(\omega) = \mathit{L}(\omega) u(\omega)\in V' \quad\forall \omega\in \Gamma.
	\end{align*}
	Notice that for any strongly measurable $u:\Gamma \to V$ the map $\omega\in\Gamma \mapsto L(\omega)u(\omega)\in V'$ is strongly measurable, $V'$ being separable, see Proposition~\ref{prop:m-bility} in the appendix.
	From the uniform boundedness of $\mathit{L}$ it follows immediately that, if $u$ is square integrable, then $\mathcal{L}(u)$ is square integrable as well and $\|\mathcal{L}(u)\|_{\lr(\Gamma; V')}\leq \Cbound \|u\|_{\lr(\Gamma; V)}$,  $\forall u\in \lr(\Gamma;V)$. The operator $\mathcal{L}$ induces a bilinear form on $\lrV$ defined as $$\langle v,w \rangle_{\mathcal{L},\rho}:= \int_{\Gamma} \big( \mathcal{L}(v)(\omega), w(\omega) \big)_{V' V}\dx \rho(\omega),\qquad v,w\in \lrV,$$ which is coercive and bounded with coercivity and continuity constant $\Ccoerc$ and $\Cbound$, respectively, i.e.
	\begin{align*}
		\langle v,v\rangle_{\mathcal{L},\rho} &\geq \Ccoerc\|v\|_{V,\lr}^2,\\
		\langle u,v\rangle_{\mathcal{L},\rho} &\leq \Cbound\|u\|_{V,\lr}\|v\|_{V,\lr}.
	\end{align*}
	Then, given a final time $T>0$, a random forcing term $f\in L^2(0,T; \lrH)$
	%\fn{\yk{YK: Fabio asked if $\lrH$ in ``$f\in L^2(0,T; \lrH)$'' needs to be separable. I would be cautious for the measurability, since Borel measurability does not imply strong measurability in general for nonseparable-space-valued functions (we need the strong measurability for Bochner integral). But here by definition of $\lrH$ we are forcing $f$ to be strongly measurable, so as far as the measurability of $f$ is concerned it should be fine even if $\lrH$ is not separable. But maybe I am not getting the whole picture of the question.}}
	and a random initial condition $u_0\in \lrV$, we consider now the following parabolic problem: Find a solution $u\tr\in L^2(0,T; \lrV)$ with $\dot{u}\tr\in L^2(0,T; \lrVdual)$ satisfying 
	
	\begin{align}\label{eq:pr1}
		\begin{split}
			\big( \dot{u}\tr, v\big)_{V'V,\lr} + \big(\mathcal{L}(u\tr),v\big)_{V'V,\lr} &= \langle f,v\rangle_{H,\lr}, \\
			&\forall v\in \lrV, \;\text{a.e.\ } t\in (0,T]\\
			u\tr(0) &= u_0.
		\end{split}
	\end{align}
	The general theory of parabolic equations (see e.g.\ \cite{Wloka87}) can be applied to problem \eqref{eq:pr1}, at least in the case of $\lrV, \lrH, \lrVdual$ being separable, e.g.\ when $\Gamma$ is a Polish space and $\mathcal{F}$ is the corresponding Borel $\sigma$-algebra. We conclude then that problem \eqref{eq:pr1} has a unique solution $u\tr$ which depends continuously on $f$ and $u_0$. We note that the theory of parabolic equations would allow for less regular data $f\in L^2(0,T; \lrVdual)$ and $u_0\in \lrH$. However, in this work we restrict our attention to the case $f\in L^2(0,T; \lrH)$, $u_0\in \lrV$.
	
	\section{Dynamical low rank approximation \corr{and its variational formulation}}\label{sec:DLR}
	
	Dynamical low rank (DLR) approximation, or dynamically orthogonal (DO) approximation (see e.g. \cite{Koch07,Sapsis09,Koch07b}) seeks an approximation of the solution $u\tr$ of problem \eqref{eq:pr1} in the form 
	\begin{equation}\label{eq:DLRexp}
		u(t) = \bar{u}(t) + \sum_{j=1}^R U_j(t) Y_j(t),\qquad t\in[0,T]
	\end{equation}
	where $\bar{u}(t)\in V$, $\{U_j(t)\}_{j=1}^R \subset V$ is a time dependent set of linearly independent deterministic basis functions, $\{Y_j(t)\}_{j=1}^R\subset\lr$ is a time dependent set of linearly independent stochastic basis functions. In what follows, we focus on the so called Dual DO formulation (see e.g.\  \cite{Musharbash15}), in which the stochastic basis $\{Y_j(t)\}_{j=1R}^R\subset\lr$ is kept orthonormal in $\langle\cdot,\cdot\rangle_{\lr}$ at all times whereas $\{U_j(t)\}_{j=1}^R$ are only required to be linearly independent at all times. We call $R$ the rank of a function $u$ of the form \eqref{eq:DLRexp}. To ensure the uniqueness of the expansion \eqref{eq:DLRexp} for a given initialization $u(0) = \bar{u}(0) + \sum_{j=1}^R U_j(0)Y_j(0)$, we consider the following conditions:  \begin{equation}\label{eq:ONcond}
		\langle Y_i(t), Y_j(t)\rangle_{\lr} = \delta_{ij},\qquad \mathbb{E}[Y_j(t)] = 0,\quad \forall 1\leq i,j\leq R,\quad \forall t\in [0,T]
	\end{equation}
	and the gauge condition (also called DO condition)
	\begin{equation}\label{eq:DOcond}
		\langle \dot{Y}_i(t),Y_j(t)\rangle_{\lr} = 0,\quad \forall 1\leq i,j\leq R,\quad \forall t\in (0,T)
	\end{equation}
	(see \cite{Kazashi20}).
	
	Plugging the DLR expansion \eqref{eq:DLRexp} into the equation \eqref{eq:pr1} and following analogous steps as proposed in \cite{Sapsis09} leads to the DLR system of equations presented next.
	
	\begin{definition}[DLR solution]\label{def:contDLRsol}
		We define the DLR solution of problem \eqref{eq:pr1} as $$u(t) = \bar{u}(t) + \sum_{i=1}^R U_i(t)Y_i(t)\quad\in \lrV$$ where $\bar{u},\,\{U_i\}_{i=1}^R,\, \{Y_i\}_{i=1}^R$ are solutions of the following system of equations:
		\begin{align} 
			& (\dot{\bar{u}}, v)_{V'V} + (\mathbb{E}[\mathcal{L}(u)], v)_{V'V} = \langle \mathbb{E}[f],v\rangle_{H} & &\forall v\in V \label{eq:DLReq1}\\[6pt]
			& (\dot{U}_j, v)_{V'V} + ( \mathbb{E}[\mathcal{L}(u)Y_j], v)_{V' V} = \langle \mathbb{E}[f Y_j], v\rangle_{H} & &\forall v\in V,\, j=1,\dots,R   \label{eq:DLReq2}\\
			&  \dot{Y}_j + \sum_{i=1}^R (M^{-1})_{j,i}\Pro_{\mathcal{Y}}^{\perp}\Big[( \mathcal{L}^*(u), U_i)_{V' V}  - \langle f^*,U_i\rangle_H \Big] = 0 & &\text{ in } \lr,\, j=1,\dots,R \label{eq:DLReq3}
		\end{align}
		with the initial conditions $\bar{u}(0),\, \{Y_j(0)\}_{j=1}^R,\,\{U_j(0)\}_{j=1}^R$ such that $\bar{u}(0)\in V$, $\{Y_j(0)\}_{j=1}^R$ satisfies the conditions \eqref{eq:ONcond}, $\{U_j(0)\}_{j=1}^R$ are linearly independent in $V$, and $\bar{u}(0)\, +\, \sum_{j=1}^R Y_j(0)U_j(0)$ is a good approximation of $u_0$.
		In \eqref{eq:DLReq3}, the matrix $M\in\mathbb{R}^{R\times R}$ is defined as $M_{ij} := \langle U_i, U_j\rangle_H$, $1\leq i,j\leq R$ and  $\Pro_{\mathcal{Y}}^{\perp}$ denotes the orthogonal projection operator in the space $\lr(\Gamma)$ on the orthogonal complement of the $R$-dimensional subspace $\mathcal{Y} = \spn\{Y_1,\dots,Y_R\}$, i.e.\ 
		\begin{equation}\label{eq:Yproj}
			\Pro_{\mathcal{Y}}^{\perp}[v] = v - \Pro_{\mathcal{Y}}[v] = v - \sum_{j=1}^R \langle v, Y_j\rangle_{\lr} Y_j,\qquad \text{for }v\in\lr.
		\end{equation}
	\end{definition}
	For the initial condition one can use for instance a truncated Karhunen-Lo\`{e}ve expansion  $u(0) = \bar{u}(0) + \sum_{i=1}^R U_i(0) Y_i(0)$ where $ \bar{u}(0) = \mathbb{E}[u_0]$, $\{U_i(0)\}_{i=1}^R$ are the first $R$ (rescaled) eigenfunctions of the covariance operator $\mathcal{C}_{u_0}: H\to H$ defined as \[\langle\mathcal{C}_{u_0} v, w\rangle_H = \mathbb{E}[\langle u_0 - \bar{u}(0), v\rangle_H \langle u_0 - \bar{u}(0),w\rangle_H]\qquad \forall v,w\in H\]
	and $Y_i = \langle u_0 - \bar{u}(0), U_i\rangle_H$ (the eigenfunctions are suitably rescaled so that $\mathbb{E}[Y_i^2] = 1$). We note that for $u_0\in\lrV$, the eigenfunctions $\{U_i(0)\}_{i=1}^R$ are  in $V$.

	In what follows we will use the notation $U = (U_1,\dots,U_R)$ and $Y = (Y_1,\dots,Y_R)$. Then, the approximation \eqref{eq:DLRexp} reads $u = \bar{u} + U Y^{\intercal}$.
	
	\corr{The rest of the section gives a geometrical interpretation of the DLR method and derives a variational formulation, following to a large extent derivations from \cite{Musharbash15}.} Such geometrical interpretation and consequent variational formulation will be key to derive the stability results of the numerical schemes, discussed in Section \ref{sec:stabestDLR}. We first introduce the notion of a manifold of $R$-rank functions, characterize its tangent space in a point as well as the orthogonal projection onto the tangent space.
	
	The vector space consisting of all square integrable random variables with zero mean value will be denoted by $\lrz = \lrz(\Gamma)\subset\lr(\Gamma)$. 
	
	\begin{definition}[Manifold of $R$-rank functions]
		By $\mathcal{M}_R\subset \lrVz$ we denote the manifold consisting of all rank $R$ random functions with zero mean 
		\begin{align}\label{eq:DLRman}
			\begin{split}
				&\mathcal{M}_R = \Big\{v^*\in \lrVz \,|\; v^* = \sum_{i=1}^R U_i Y_i = UY^\intercal,\\
				&\hspace{1.3cm} \langle Y_i, Y_j\rangle_{\lr} = \delta_{ij}, \,\forall \eR{i,j}, \, \{U_i \}_{i=1}^R \text{ linearly independent}\Big\}.
			\end{split}
		\end{align}
	\end{definition}
	It is well known that $\mathcal{M}_R$ admits an infinite dimensional Riemannian manifold structure (\cite{Falco19}).
	
	\begin{proposition}[Tangent space at $UY^\intercal$]
		The tangent space $\mathcal{T}_{UY^\intercal}\mathcal{M}_R$ at a point $U Y^\intercal \in\mathcal{M}_R$ can be characterized as
		\begin{multline}\label{eq:DLRtang}
			\mathcal{T}_{UY^\intercal}\mathcal{M}_R = \Big\{  \delta v\in \lrVz\,|\, \delta v = \sum_{i=1}^R U_i\delta Y_i + \delta U_i Y_i,\\
			\delta U_i\in V,\;\delta Y_i\in\lrz,\; \;\langle \delta Y_i,Y_j\rangle_{\lr} = 0,\,\forall\eR{i,j}\Big\}.
		\end{multline}
	\end{proposition}
	
	\begin{proposition}[Orthogonal projection on $\mathcal{T}_{UY^\intercal}\mathcal{M}_R$]
		The $\lrHz$-orthogonal projection $\Pi_{UY^{\intercal}}[v]$ of a function $v\in\lr(\Gamma, H)$ onto the tangent space $\mathcal{T}_{UY^\intercal}\mathcal{M}_R$ is given by
		\begin{align}\label{eq:DLRproj}
			\begin{split}
				\Pi_{UY^{\intercal}}[v] &= \sum_{i=1}^R \langle v, Y_i\rangle_{\lr} Y_i + \Pro_{\mathcal{Y}}^{\perp}[\sum_{i=1}^R\langle v,U_i\rangle_H (M^{-1}U^\intercal)_i]\\
				&= \Pro_{\mathcal{Y}}[v] + \Pro_{\mathcal{Y}}^{\perp}\big[\Pro_{\mathcal{U}}[v]\big] = \Pro_{\mathcal{Y}}[v] + \Pro_{\mathcal{U}}[v] - \Pro_{\mathcal{Y}}\big[\Pro_{\mathcal{U}}[v]\big]
			\end{split}
		\end{align}
		where $\mathcal{U} = \spn\{U_1,\dots,U_R\}$ and $\Pro_{\mathcal{U}}[\cdot]$ is the $H$-orthogonal projection onto the subspace $\mathcal{U}$.
	\end{proposition}
	For more details, see e.g.\  \cite{Musharbash15}. Note that $\Pi_{UY^{\intercal}}[\cdot]$ can be equivalently written as $\Pi_{UY^{\intercal}}[\cdot] = \Pro_{\mathcal{U}}[\cdot] + \Pro_{\mathcal{U}}^{\perp}\big[\Pro_{\mathcal{Y}}[\cdot]\big]$. In the following we will extend the domain of the projection operator $\Pi_{UY^\intercal}$. Further, we will state two lemmas used to establish Theorem \ref{th:varform}, which presents the variational formulation of the DLR approximation.
	
	The operator $\Pi_{UY^\intercal}$ can be extended to an operator from $\lrVdual$ to $\lrVdual$ as
	\begin{align*}
		\Pi_{UY^\intercal}[\mathcal{K}] := \langle \mathcal{K}, Y\rangle_{\lr} Y^\intercal + \Pro_{\mathcal{Y}}^\perp\big[ ( \mathcal{K}, U)_{V' V}M^{-1}U^\intercal \big]\qquad\forall\mathcal{K}\in \lrVdual.
	\end{align*}
	%	\fn{ If ``$\langle\mathcal{K},Y\rangle_{L_{\rho}^{2}}$'' is in $V'$? For $\mathcal{K}\in L_{\rho}^{2}(\Gamma;V')$ and $Y\in L_{\rho}^{2}(\Gamma)$,
	%			the multiplication is well defined: $Y(\omega)\mathcal{K}(\omega)\in V'$
	%			($\because$ $V'$ is a vector space and $Y(\omega)$ is a scalar).
	%			Moreover, we have $Y\mathcal{K}\in L_{\rho}^{1}(\Gamma;V')$: 
	%			\[
	%			\int_{\Gamma}\|Y(\omega)\mathcal{K}(\omega)\|_{V'}\mathrm{d}\rho(\omega)\leq\|Y\|_{L_{\rho}^{2}(\Gamma)}\|\mathcal{K}\|_{L_{\rho}^{2}(\Gamma;V')}.
	%			\]
	%			Hence, the expression
	%			\[
	%			\langle\mathcal{K},Y\rangle_{L_{\rho}^{2}}:=\int_{\Gamma}Y(\omega)\mathcal{K}(\omega)\mathrm{d}\rho(\omega)
	%			\]
	%			is well defined (but technically this is a new notation...). We want
	%			to show $\langle\mathcal{K},Y\rangle_{L_{\rho}^{2}}\in V'$.
	%			
	%			Given $v\in V$, the linear functional $(\cdot,v)_{V',V}\colon V'\to\mathbb{R}$
	%			is continuous and $(Y(\cdot)\mathcal{K}(\cdot),v)_{V',V}\in L_{\rho}^{1}(\Gamma)$.
	%			Thus, we have
	%			
	%			\begin{align*}
	%			(\langle\mathcal{K},Y\rangle_{L_{\rho}^{2}},v)_{V',V} & =(\int_{\Gamma}Y(\omega)\mathcal{K}(\omega)\mathrm{d}\rho(\omega),v)_{V',V}=\int_{\Gamma}(Y(\omega)\mathcal{K}(\omega),v)_{V',V}\mathrm{d}\rho(\omega)\\
	%			& =\int_{\Gamma}Y(\omega)(\mathcal{K}(\omega),v)_{V',V}\mathrm{d}\rho(\omega)\\
	%			& \leq\|Y\|_{L_{\rho}^{2}(\Gamma)}\|\mathcal{K}\|_{L_{\rho}^{2}(\Gamma;V')}\|v\|_{V}<\infty.
	%			\end{align*}
	%			Therefore, $\langle\mathcal{K},Y\rangle_{L_{\rho}^{2}}\in V'$.
	%	}
	The extended operator satisfies the following.
	
	\begin{lemma}\label{lemma:proj_symm} Let $UY^\intercal\in\mathcal{M}_R$. Then it holds
		\begin{equation}\label{eq:proj_symm}
			(\mathcal{K}, \Pi_{U Y^\intercal}[v] )_{V'V,\lr} = ( \Pi_{U Y^\intercal} [\mathcal{K}], v)_{V'V,\lr},\quad \forall v\in \lrV,\;\mathcal{K}\in \lrVdual.
		\end{equation}
	\end{lemma}
	\begin{proof} First, we show that 
		\[
		(\mathcal{K}, \,\Pro_{\mathcal{Y}}[v])_{V'V,\lr} = (\Pro_{\mathcal{Y}}[\mathcal{K}], \,v)_{V'V,\lr} \quad \forall v\in \lrV,\;\mathcal{K}\in \lrVdual.
		\]
		Indeed,
		\begin{align*}
			&(\mathcal{K}, \,\Pro_{\mathcal{Y}}[v])_{V'V,\lr} = \int_{\Gamma}\Big(\mathcal{K}, \sum_{i=1}^R \langle v, Y_i\rangle_{\lr} Y_i\Big)_{V'V}\dx\rho\\
			&\qquad= \sum_{i=1}^R \int_{\Gamma} \big(\mathcal{K},\langle v, Y_i\rangle_{\lr}Y_i\big)_{V'V}\dx\rho\\
			&\qquad=\sum_{i=1}^R\int_{\Gamma} \big(\mathcal{K}\,Y_i,\langle v,Y_i\rangle_{\lr}\big)_{V'V}\dx\rho = \sum_{i=1}^R \big(\langle\mathcal{K},Y_i\rangle_{\lr},\langle v,Y_i\rangle_{\lr}\big)_{V'V}\\
			&\qquad= \sum_{i=1}^R \int_{\Gamma} \big(\langle\mathcal{K}, Y_i\rangle_{\lr}Y_i, v\big)_{V'V} \dx\rho = (\Pro_{\mathcal{Y}}[\mathcal{K}], \,v)_{V'V,\lr},
		\end{align*}
		where in the forth step we applied Theorem 8.13 from \cite{Leoni17}.
		
		Now we proceed with proving \eqref{eq:proj_symm}
		\begin{align*}
			(\mathcal{K},\, \Pi_{U Y^\intercal}[v] )_{V'V,\lr} &= (\mathcal{K}, \,\Pro_{\mathcal{Y}}[v] + \Pro_{\mathcal{Y}}^{\perp}[\Pro_{\mathcal{U}}[v]])_{V'V,\lr}\\
			&= (\Pro_{\mathcal{Y}}[\mathcal{K}], \,v)_{V'V,\lr} + (\Pro_{\mathcal{Y}}^{\perp}[\mathcal{K}], \,\Pro_{\mathcal{U}}[v])_{V'V,\lr}\\
			&= \big(\Pro_{\mathcal{Y}}[\mathcal{K}], \,v\big)_{V'V,\lr} + \big(\Pro_{\mathcal{Y}}^{\perp}[\mathcal{K}], \,(v,U)_{H}M^{-1}U^\intercal\big)_{V'V,\lr}\\
			&= \big(\Pro_{\mathcal{Y}}[\mathcal{K}], \,v\big)_{V'V,\lr} + \int_{\Gamma}\big(\Pro_{\mathcal{Y}}^{\perp}[\mathcal{K}],UM^{-1} \big)_{V'V}\big(U^\intercal,v\big)_{H}\dx\rho\\
			&= (\Pro_{\mathcal{Y}}[\mathcal{K}],\, v)_{V'V,\lr} + \big((\Pro_{\mathcal{Y}}^{\perp}[\mathcal{K}],U)_{V'V}M^{-1}U^\intercal ,\,v\big)_{V'V,\lr}\\
			&= ( \Pi_{U Y^\intercal} [\mathcal{K}], v)_{V'V,\lr}.
		\end{align*}
		\qed
	\end{proof}
	We are now in the position to state the first variational formulation of the DLR equations.
	
	\begin{lemma}\label{lemma:DLRproj}
		Let $U, Y$ be the solution of the system \eqref{eq:DLReq2}--\eqref{eq:DLReq3}. Then the zero-mean part of the DLR solution $u^* = UY^\intercal$ satisfies
		\begin{equation}\label{eq:DLRvar0}
			( \dot{u}^* + \Pi_{u^*}[\mathcal{L}^*(u) - f^*],\, v)_{V'V,\lr} = 0,\qquad\forall v\in \lrV.
		\end{equation}
	\end{lemma}
	\begin{proof}
		First, we multiply equation \eqref{eq:DLReq2} by $Y_j$ and take its weak formulation in $\lr$. Summing over $j$ results in 
		\begin{equation*}		
			\Big(\dot{U}Y^\intercal + \mathbb{E}\big[\big(\mathcal{L}(u) - f\big)Y\big]Y^\intercal, v\,w\Big)_{V'V,\lr}   = 0 \quad \forall v\in V,\, w\in\lr.
		\end{equation*}
		Notice that $\mathbb{E}\big[\big(\mathcal{L}^*(u) - f^*\big)Y\big] =\mathbb{E}\big[\big(\mathcal{L}(u) - f\big)Y\big]$ since $Y\subset \lrz.$
		Analogously, we multiply \eqref{eq:DLReq3} by $U_j$ and take its weak formulation in $V'$ 
		\begin{multline*} \Big(U_j\dot{Y}_j + \sum_{i=1}^R U_j (M^{-1})_{j,i}\Pro_{\mathcal{Y}}^{\perp}\Big[( \mathcal{L}^*(u)-f^*, U_i)_{V' V} \Big], v\,w\Big)_{V'V,\lr} = 0\\ \forall v\in V,\, w\in\lr.
		\end{multline*}
		Summing over $j$, this leads to \[\Big( U\dot{Y}^\intercal + \Pro_{\mathcal{Y}}^{\perp}\Big[( \mathcal{L}^*(u) - f^*, U)_{V' V}M^{-1}U^\intercal\Big], v\,w\Big)_{V'V,\lr} = 0\quad \forall v\in V,\, w\in\lr.\]
		Summing the derived equations we obtain $$\Big( \frac{\mathrm{d}}{\mathrm{d}t}(UY^{\intercal}) + \Pi_{u^*}[\mathcal{L}^*(u) - f^*],\, z\Big)_{V'V,\lr} = 0\quad \forall z\in \spn\{v\,w: v\in V,\, w\in\lr\}.$$
		In particular, this holds for any $z$ being a Bochner integrable simple function, the collection of which is dense in $\lrV$ (see \cite[Th.  8.15]{Leoni17}).
		\qed
	\end{proof}
	We can finally state the variational formulation corresponding to the DLR equations \eqref{eq:DLReq1}--\eqref{eq:DLReq3}.	
	
	\begin{theorem}[DLR variational formulation]\label{th:varform}
		Let $\bar{u}, U, Y$ be the solution of the system \eqref{eq:DLReq1}--\eqref{eq:DLReq3}. Then the DLR solution $u = \bar{u} + UY^\intercal$ satisfies
		\begin{equation}\label{eq:DLRvar}
			\big( \dot{u} + \mathcal{L}(u),\, v\big)_{V'V,\lr} = \langle f,v\rangle_{H,\lr},\qquad\forall v = \bar{v} + v^*,\,\bar{v}\in V,\, v^*\in\mathcal{T}_{u^*}\mathcal{M}_R.
		\end{equation}
	\end{theorem}
	\begin{proof}
		Based on Lemma~\ref{lemma:DLRproj} and Lemma~\ref{lemma:proj_symm} we can write 
		\begin{multline*} 			
			\big( \dot{u}^*,\,v\big)_{V'V,\lr} + \big(\Pi_{u^*}[\mathcal{L}^*(u) - f^*],\, v\big)_{V'V,\lr}\\
			=\big( \dot{u}^*,\,v\big)_{V'V,\lr} + \big(\mathcal{L}^*(u) - f^*,\, \Pi_{u^*}[v]\big)_{V'V,\lr} = 0,\qquad\forall v\in \lr(\Gamma;V).
		\end{multline*}
		Since $\Pi_{u^*}[v] = v,\;\forall v\in \mathcal{T}_{u^*}\mathcal{M}_R$, this results in $$\big( \dot{u}^* + \mathcal{L}^*(u) - f^*,\, v\big)_{V'V,\lr} = 0,\qquad\forall v\in \mathcal{T}_{u^*}\mathcal{M}_R,$$ which can be equivalently written as 
		\begin{equation}\label{eq:tmpeq1}		
			( \dot{u}^* + \mathcal{L}^*(u) - f^*,\, w + v\big)_{V'V,\lr} = 0,\qquad\forall w\in V,\, \forall v \in\mathcal{T}_{u^*}\mathcal{M}_R,
		\end{equation}
		exploiting the fact that $\big( \dot{u}^* + \mathcal{L}^*(u) - f^*,\, w \big)_{V'V,\lr} = 0,\;\forall w\in V$. Likewise, equation \eqref{eq:DLReq1} can be equivalently written as 
		\begin{equation}\label{eq:tmpeq1.2}
			\big( \dot{\bar{u}} + \mathbb{E}[\mathcal{L}(u) - f],\, w + v\big)_{V'V,\lr} = 0,\qquad \forall w\in V,\, \forall v \in\mathcal{T}_{u^*}\mathcal{M}_R,
		\end{equation}
		exploiting the fact that $\big( \dot{\bar{u}} + \mathbb{E}[\mathcal{L}(u) - f],\, v\big)_{V'V,\lr} = 0$ as $\mathbb{E}[v] = 0 \; \forall v \in\mathcal{T}_{u^*}\mathcal{M}_R$.
		Summing \eqref{eq:tmpeq1} and \eqref{eq:tmpeq1.2} leads to the sought equation \eqref{eq:DLRvar}.
		\qed
	\end{proof}
	
	Recently, the existence and uniqueness of the dynamical low rank approximation for a class of random semi-linear evolutionary equations was established in \cite{Kazashi20} and for linear parabolic equations in two space dimensions with a symmetric operator $\mathcal{L}$ in \cite{Bachmayr20}.
	
	\section{Discretization of DLR equations}\label{sec:disc}
	In this section we describe the discretization of the DLR equations  that we consider in this work. In particular, we focus on the time discretization of \eqref{eq:DLReq1}--\eqref{eq:DLReq3} and propose a staggered time marching scheme that decouples the update of the spatial and stochastic modes. Afterwards, we will show that the proposed scheme can be formulated as a projector-splitting scheme for the Dual DO formulation and comment on its connection to the projector-splitting scheme from \cite{Lubich14}. As a last result we state and prove a variational formulation of the discretized problem.
	
	\paragraph{Stochastic discretization}
	
	We consider a discrete measure given by $\{\omega_k,\lambda_k\}_{k=1}^{\hat{N}}$, i.e.\ a set of sample points $\{\omega_k\}_{k=1}^{\hat{N}}\subset\Gamma$ with $R < \hat{N} < \infty$ and a set of positive weights $\{\lambda_k\}_{k=1}^{\hat{N}},\; \lambda_k>0,$ $\sum_{k=1}^{\hat{N}} \lambda_k=1$, which approximates the probability measure $\rho$ \[\hat{\rho} \coloneqq \sum_{k=1}^{\hat{N}} \lambda_k \delta_{\omega_k}\approx \rho.\]
	The discrete probability space $(\hat{\Gamma} = \{\omega_k\}_{k=1}^{\hat{N}}, 2^{\hat{\Gamma}}, \hat{\rho})$ will replace the original one $(\Gamma,\mathcal{F}, \rho)$ in the discretization of the DLR equations. Notice, in particular, that a random variable $Z:\hat{\Gamma}\mapsto\mathbb{R}$ measurable on $(\hat{\Gamma}, 2^{\hat{\Gamma}}, \hat{\rho})$ can be represented as a vector $z\in\mathbb{R}^{\hat{N}}$ with $z_k = Z(\omega_k), \, k=1,\dots,\hat{N}$. The sample points $\{\omega_k\}_{k=1}^{\hat{N}}$ can be taken as iid samples from $\rho$ (e.g.\ Monte Carlo samples) or chosen deterministically (e.g.\ deterministic quadrature points with positive quadrature weights). The mean value of a random variable  $Z$ with respect to the measure $\hat{\rho}$ is computed as $$\mathbb{E}_{\hat{\rho}}[Z] = \sum_{k=1}^{\hat{N}} Z(\omega_k)\lambda_k.$$
	% and \todo{we consider only random variables with finite second moments} $Z\in\lhr = \lhr(\{\omega_k\})$, i.e.\ $$\langle Z,Z\rangle_{\lhr} = \sum_{k=1}^{\hat{N}} Z(\omega_k)^2\lambda_k<\infty.$$ 
	We introduce also the semi-discrete scalar products $\langle\cdot,\cdot\rangle_{\star,\lhr}$ with $\star = V,H$ and their corresponding induced norms $\|\cdot\|_{\star,\lhr}$. Note that the semi-discrete bilinear form $\langle \cdot, \cdot\rangle_{\mathcal{L},\hat{\rho}}$ defined as \[\langle v, w\rangle_{\mathcal{L},\hat{\rho}} = \sum_{k=1}^{\hat{N}} \mathit{L}(\omega_k)v(\omega_k) w(\omega_k)\lambda_k\] is coercive and bounded, with the same coercivity and continuity constants $\Ccoerc,\,\Cbound$, defined in \eqref{eq:coerc_const}, \eqref{eq:bound_const}, respectively.
	
	\paragraph{Space discretization}
	
	We consider a general finite-dimensional subspace $V_h\subset V$ whose dimension is larger than $R$ and is determined by the discretization parameter $h$. Eventually, we will perform a Galerkin projection of the DLR equations onto the subspace $V_h$. We further assume that an inverse inequality of the type	
	\begin{equation}\label{eq:discDLRinv}
		\| v\|_{V,\lhr}\leq \frac{\Cinv}{h^p}\|v\|_{H,\lhr},\qquad \forall v\in V_h\otimes\lhr
	\end{equation}
	holds for some $p\in\N$ and $\Cinv > 0$.
	
	\paragraph{Time discretization}
	
	For the time discretization we divide the time interval into $N$ equally spaced  subintervals $0=t_0<t_1<\dots<t_N = T$ and denote the time step by  $\triangle t := t_{n+1} - t_n$. Note that the DLR solution $u = \bar{u} + UY^\intercal$ appears in the right hand side of the system of equations \eqref{eq:DLReq1}--\eqref{eq:DLReq3} both in the operator $\mathcal{L}$ and in the projector operator onto the tangent space to the manifold. We will treat these two terms differently. Concerning the projection operator, we adopt a staggered strategy, where, given the approximate solution $u^n = \bar{u}^n + U^n Y^{n^\intercal}$, we first update the mean $\bar{u}^{n+1}$, then we update the deterministic basis $U^{n+1}$ projecting on the subspace $\spn\{Y^n\};$ finally, we update the stochastic basis $Y^{n+1}$ projecting on the orthogonal complement of $\spn\{Y^n\}$ and on the updated subspace $\spn\{U^{n+1}\}$. This staggered strategy resembles the projection splitting operator proposed in \cite{Lubich14}. We will show later in Section~\ref{sec:projsplit} that it does actually coincide with the algorithm in \cite{Lubich14}. Concerning the operator $\mathcal{L}$, we will discuss hereafter different discretization choices leading to explicit, semi-implicit or fully implicit algorithms.
	
	\subsection{Fully discrete problem}\label{sec:fullydiscpr}
	
	We give in the next algorithm the general form of the discretization schemes that we consider in this work.
	
	\begin{algorithm}\label{alg:ourscheme}
		Given the approximated solution $u^n_{h,\hat{\rho}} = \bar{u}^n + \sum_{j=1}^R U_j^n Y_j^n$ at time $t_n$ with 
		\begin{gather*}
			\bar{u}^n, U_j^n \in V_h, \quad Y_i^n\in\lhr,\\ \langle Y_i^n, Y_j^n\rangle_{\lhr} = \delta_{ij}, \quad \mathbb{E}\hr[Y_j^n] = 0,\quad\forall\, \eR{i,j}:
		\end{gather*}
		\begin{enumerate}
			\item Compute the mean value $\bar{u}^{n+1}$ such that 
			\begin{multline} \label{eq:discDLReq1}
				\Big\langle \frac{\bar{u}^{n+1} - \bar{u}^n}{\triangle t},v_h \Big\rangle_{H} + \Big(\mathbb{E}\hr[\mathcal{L}(u^n_{h,\hat{\rho}}, u^{n+1}_{h,\hat{\rho}}) - f^{n,n+1}], v_h \Big)_{V' V} = 0\\ \forall v_h \in V_h.
			\end{multline}
			\item Compute the deterministic basis $\tilde{U}_j^{n+1}$ for $j = 1,\dots,R$
			\begin{multline}\label{eq:discDLReq2}
				\Big\langle \frac{\tilde{U}^{n+1}_j - U^n_j}{\triangle t}, v_h\Big\rangle_H + \Big(\mathbb{E}\hr[(\mathcal{L}(u^n_{h,\hat{\rho}}, u^{n+1}_{h,\hat{\rho}})-f^{n,n+1}) Y^n_j], v_h \Big)_{V' V} = 0\\ \forall v_h \in V_h.
			\end{multline}
			\item Compute the stochastic basis $\{\tilde{Y}_j^{n+1}\}_{j=1}^R$ such that
			\begin{equation}\label{eq:discDLReq3}
				\frac{\tilde{Y}^{n+1} - Y^n}{\triangle t} \tilde{M}^{n+1} + \Pro_{\hat{\rho},\mathcal{Y}^n}^{\perp} \Big[ \big( \mathcal{L}^{*}(u^n_{h,\hat{\rho}}, u^{n+1}_{h,\hat{\rho}}) - f^{{n,n+1}^*},\tilde{U}^{n+1} \big)_{V'V}\Big] = 0.
			\end{equation}
			where $\tilde{M}^{n+1} = \langle\tilde{U}^{{n+1}^\intercal}, \tilde{U}^{n+1}\rangle_H$, $\Pro_{\hat{\rho},\mathcal{Y}^n}^{\perp}[\cdot]$ is the analogue of the projector defined in \eqref{eq:Yproj} but in the discrete space $\lhr$.
			\item Reorthonormalize the stochastic basis: find $(U^{n+1}, Y^{n+1})$ s.t.
			\begin{equation}\label{eq:discDLReq4}
				\sum_{j=1}^R Y^{n+1}_j U^{n+1}_j = \sum_{j=1}^R \tilde{Y}^{n+1}_j \tilde{U}^{n+1}_j, \qquad \langle Y^{{n+1}^\intercal}, Y^{n+1}\rangle_{\lhr} = \mathrm{Id}.
			\end{equation}
			\item Form the approximated solution at time step $t_{n+1}$ as 
			\begin{equation}\label{eq:discDLReq5}
				u^{n+1}_{h,\hat{\rho}} = \bar{u}^{n+1} + \sum_{j=1}^R U_j^{n+1}Y_j^{n+1}.
			\end{equation}
		\end{enumerate}
		The expressions $\mathcal{L}(u^n_{h,\hat{\rho}}, u^{n+1}_{h,\hat{\rho}})$ and $f^{n,n+1}$ stand for an unspecified time integration of the operator $\mathcal{L}(u(t))$ and right hand side $f(t),\; t\in[t_n, t_{n+1}]$ and $v^*$ denotes the 0-mean part of a random variable $v\in\lhr$ with respect to the discrete measure $\hat{\rho}$, i.e.\ $v^* = v - \mathbb{E}_{\hat{\rho}}[v].$
	\end{algorithm}

	The newly computed solution $u^{n+1}_{h,\hat{\rho}}$ belongs to the tensor product space $V_h\otimes\lhr$, since we have $\bar{u}^{n+1}, U^{n+1}_j\in V_h$ and $ Y_j\in\lhr$, $\eR{j}$.
	Note that equation \eqref{eq:discDLReq3} is set in $\lhr$. 
	Since $\lhr$ is a finite dimensional vector space isomorphic to $\R^{\hat{N}}$, equation \eqref{eq:discDLReq3} can be rewritten as a deterministic linear system of $R\times \hat{N}$ equations with $R\times \hat{N}$ unknowns. This system can be decoupled into a linear system of size $R\times R$ for each collocation point. If the deterministic modes $\tilde{U}^{n+1}$ are linearly independent, the system matrix is invertible.
	Otherwise we interpret \eqref{eq:discDLReq3} in a minimal-norm least squares sense, choosing a solution $\tilde{Y}^{n+1}$, if it exists, that minimizes the norm $\|\tilde{Y}^{n+1} - Y^n\|_{\lhr}$. This is discussed in more details in Section \ref{sec:discvar_rankdef}.
	
	The following lemma shows that the scheme \eqref{eq:discDLReq1}--\eqref{eq:discDLReq3} satisfies some important properties that will be essential in the stability analysis presented in Section \ref{sec:stabest}.
	
	\begin{lemma}[Discretization properties]\label{lemma:discprop}
		\sloppy{Assuming that a solution $(\tilde{Y}^{n+1}, \tilde{U}^{n+1},\bar{u}^{n+1})$ exists, the following properties hold for the discretization} \eqref{eq:discDLReq1}--\eqref{eq:discDLReq3}:
		\begin{enumerate}
			\item Discrete DO condition: 
			\begin{equation}\label{eq:discDO1}
				\big\langle\frac{\tilde{Y}_i^{n+1} - Y_i^n}{\triangle t}, Y_j^n\big\rangle_{\lhr} = 0,\quad\forall\eR{i,j} 
			\end{equation} 
			\item $\mathbb{E}\hr[\tilde{Y}^{n+1}] = 0$
			\item $\langle \tilde{Y}^{{n+1}^\intercal}, Y^n\rangle_{\lhr} = \mathrm{Id}$
		\end{enumerate}
	\end{lemma}
	\begin{proof}\mbox{\\}
		\begin{enumerate}
			\item In the following proof we assume that the matrix $\tilde{M}^{n+1} = \langle\tilde{U}^{{n+1}^\intercal}, \tilde{U}^{n+1}\rangle_H$ is full rank. For the rank-deficient case, we refer the reader to the proof of Lemma \ref{lemma:discdo}. Let us multiply equation \eqref{eq:discDLReq3} by $Y^{n^\intercal}$ from the left and take the $\lhr$-scalar product. Since the second term involves $\Pro_{\hat{\rho},\mathcal{Y}^n}^\perp$, the scalar product of $Y^n$ with the second term vanishes which, under the assumption that $\tilde{M}^{n+1}$ is full rank, gives us the discrete DO condition $$\Big\langle  Y^{n^\intercal}, \frac{\tilde{Y}^{n+1} - Y^n}{\triangle t}\Big\rangle_{\lhr} = 0.$$
			\item This is a consequence of the fact that we have  $\mathbb{E}\hr[Y^n] = 0$ \sloppy{and $\mathbb{E}\hr\big[ \big( \mathcal{L}^*(u^n, u^{n+1})-f^{{n,n+1}^*},\tilde{U}^{n+1} \big)_{V'V}\big] = 0.$}
			\item This is immediate from the discrete DO property and $\langle Y^{n^\intercal}, Y^n\rangle_{\lhr} = \text{Id}$.
		\end{enumerate}
		\qed
	\end{proof}
	To complete the discretization scheme \eqref{eq:discDLReq1}--\eqref{eq:discDLReq3} we need to specify the terms $\mathcal{L}(u^n_{h,\hat{\rho}}, u^{n+1}_{h,\hat{\rho}})$ and $f^{n,n+1}$. 
	The DLR system stated in \eqref{eq:DLReq1}--\eqref{eq:DLReq3} is coupled. 
	Therefore, an important feature we would like to attain is to decouple the equations for the mean value, the deterministic and the stochastic modes as much as possible. We describe hereafter 3 strategies for the discretization of the operator evaluation term $\mathcal{L}(u^n_{h,\hat{\rho}}, u^{n+1}_{h,\hat{\rho}})$, and the right hand side $f^{n, n+1}$.
	
	\paragraph{Explicit Euler scheme}
	The explicit Euler scheme performs the discretization $$\mathcal{L}(u^n_{h,\hat{\rho}}, u^{n+1}_{h,\hat{\rho}}) = \mathcal{L}(u^n_{h,\hat{\rho}}),\qquad f^{n,n+1} = f(t_{n}).$$ 
	It decouples the system \eqref{eq:discDLReq1}--\eqref{eq:discDLReq3} since, for the computation of the new modes, we require only the knowledge of the already-computed modes. The equations for the stochastic modes $\{\tilde{Y}^{n+1}_j\}_{j=1}^R$ are coupled together through the matrix $\tilde{M}^{n+1} = \langle \tilde{U}^{{n+1}^\intercal}, \tilde{U}^{n+1}\rangle_{H} \in \R^{R\times R}$ but are otherwise decoupled between collocation points (i.e. $\hat{N}$ linear systems of size $R$ have to be solved). 
	
	\paragraph{Implicit Euler scheme}
	The implicit Euler scheme performs the discretization $$\mathcal{L}(u^n_{h,\hat{\rho}}, u^{n+1}_{h,\hat{\rho}}) = \mathcal{L}(u^{n+1}_{h,\hat{\rho}}),\qquad f^{n,n+1} = f(t_{n+1}).$$ This method couples the system \eqref{eq:discDLReq1}--\eqref{eq:discDLReq3} in a non-trivial way, which is why we do not focus on this method in our numerical results. We mention it in the stability estimates section (Section \ref{sec:stabestDLR}) for its interesting stability properties.
	
	\paragraph{Semi-implicit scheme}
	Assume that our operator $\mathcal{L}$ can be decomposed into two parts 
	\begin{equation*} %\label{eq:op_decomp}
		\mathcal{L}(u) = \Ldet(u) + \Lstoch(u),
	\end{equation*}
	where $\Ldet: V\to V'$ is a linear deterministic operator such that it induces a bounded and coercive bilinear form $\langle \cdot,\cdot\rangle_{\Ldet}$ on $V$ 
	\begin{equation}\label{eq:Ldet}
		\langle u,v\rangle_{\Ldet} := (\Ldet(u), v)_{V'V},\qquad u,v\in V
	\end{equation}
	and that its action on a function $v = v_1 v_2$ with $v_1\in V,\, v_2\in\lr$ is defined as \[\Ldet(v) = \Ldet(v_1) v_2.\] 
	Then, $\Ldet$ is also a linear operator $\Ldet:\lrV\to \lrVdual$ (as well as $\Ldet:\, \lhr(\hat{\Gamma};V)\mapsto \lhr(\hat{\Gamma}; V')$) and induces a bounded coercive bilinear form on $\lrV$ $$\langle u,v\rangle_{\Ldet,\rho} = \intG{(\Ldet(u), v)_{V'V}}.$$
	We propose a semi-implicit time integration of the operator evaluation term
	\begin{equation}\label{eq:discop_decomp}
		\mathcal{L}(u^n_{h,\hat{\rho}}, u^{n+1}_{h,\hat{\rho}}) = \Ldet(u^{n+1}_{h,\hat{\rho}}) + \Lstoch(u^n_{h,\hat{\rho}})
	\end{equation}
	whereas for $f^{n,n+1}$ we can either take $f^{n,n+1} = f(t_{n+1})$ or $f^{n,n+1} = f(t_n)$ or any convex combination of both. The resulting scheme is detailed in the next lemma.
	
	\begin{lemma}\label{lemma:semiimplscheme}
		The semi-implicit integration scheme \eqref{eq:discop_decomp} combined with the general steps \eqref{eq:discDLReq1}--\eqref{eq:discDLReq3} is equivalent to the following set of equations
		\begin{align}
			&\langle \bar{u}^{n+1}, v_h\rangle_{H} + \triangle t\langle \bar{u}^{n+1}, v_h\rangle_{\Ldet} \nonumber\\
			&\hspace{0.6cm}= \langle \bar{u}^{n}, v_h\rangle_{H} - \triangle t ( \mathbb{E}_{\hat{\rho}}[\Lstoch(u^n_{h,\hat{\rho}}) - f^{n,n+1}], v_h)_{V'V}\hspace{1.3cm}\forall v_h\in V_h\label{eq:semieq1}\\[12pt]
			&\langle \tilde{U}^{n+1}_j, v_h\rangle_{H} + \triangle t\langle \tilde{U}^{n+1}_j, v_h\rangle_{\Ldet} \nonumber\\
			&\hspace{0.6cm}=  \langle \tilde{U}^{n}_j, v_h\rangle_{H} - \triangle t ( \mathbb{E}_{\hat{\rho}}[(\Lstoch(u^n_{h,\hat{\rho}})-f^{n,n+1})Y^n_j], v_h)_{V'V}\hspace{0.5cm}\forall v_h\in V_h\label{eq:semieq2}\\[12pt]
			&\Big(\tilde{Y}^{n+1} - Y^n\Big)\big( \tilde{M}^{n+1} + \triangle t\langle\tilde{U}^{{n+1}^\intercal}, \tilde{U}^{n+1}\rangle_{\Ldet} \big) \nonumber\\[3pt]
			&\hspace{0.6cm}= - \triangle t \Pro_{\hat{\rho},\mathcal{Y}^n}^{\perp} [(\Lstoch^{*}(u^n_{h,\hat{\rho}})-f^{{n,n+1}^{*}}, \tilde{U}^{n+1})_{V'V}]\hspace{2.2cm}\text{in }\lhr.\label{eq:semieq3}
		\end{align}
	\end{lemma}
	\begin{proof}
		The equation for the mean \eqref{eq:discDLReq1} using the semi-implicit scheme \eqref{eq:discop_decomp} can be written as 
		\begin{multline*}
			\Big\langle \frac{\bar{u}^{n+1} - \bar{u}^n}{\triangle t},v_h \Big\rangle_{H} + \underbrace{\big(\mathbb{E}_{\hat{\rho}}[\Ldet(\bar{u}^{n+1})], v_h \big)_{V' V}}_{T_1} \\+ \underbrace{\big(\mathbb{E}_{\hat{\rho}}[\Ldet(\tilde{U}^{n+1}Y^{{n+1}^\intercal})], v_h \big)_{V' V}}_{T_2}
			= - \big(\mathbb{E}_{\hat{\rho}}[\Lstoch(u^n_{h,\hat{\rho}}) - f^{n,n+1}], v_h \big)_{V' V} .
		\end{multline*}
		Noticing that
		\begin{align*}
			T_1 &= \big(\Ldet(\bar{u}^{n+1}), v_h\big)_{V'V} = \langle \bar{u}^{n+1}, v_h\rangle_{\Ldet}\\
			T_2 &= \big(\Ldet(\tilde{U}^{n+1})\mathbb{E}_{\hat{\rho}}[Y^{{n+1}^\intercal}], v_h \big)_{V' V} = 0
		\end{align*}
		gives us equation \eqref{eq:semieq1}. Concerning the equation for the deterministic modes we derive 
		\begin{multline*}
			\Big\langle \frac{\tilde{U}^{n+1}_j - U^n_j}{\triangle t}, v_h\Big\rangle_H +
			\underbrace{\big(\mathbb{E}_{\hat{\rho}}[\Ldet(\bar{u}^{n+1}) Y^n_j], v_h \big)_{V' V}}_{T_3} \\
			+\underbrace{\big(\mathbb{E}_{\hat{\rho}}[\Ldet(\tilde{U}^{n+1}\tilde{Y}^{{n+1}^\intercal}) Y^n_j], v_h \big)_{V' V}}_{T_4} \\= 
			-\big(\mathbb{E}_{\hat{\rho}}[(\Lstoch(u^n_{h,\hat{\rho}})-f^{n,n+1}) Y^n_j], v_h \big)_{V' V}.
		\end{multline*}
		The term $T_3$ vanishes since $\mathbb{E}_{\hat{\rho}}[Y^n] = 0$ and the term $T_4$ can be further expressed as 
		\begin{align*}
			T_4 &= \big(\Ldet(\tilde{U}^{n+1})\mathbb{E}_{\hat{\rho}}[\tilde{Y}^{{n+1}^\intercal} Y^n_j], v_h \big)_{V' V} = \big(\Ldet(\tilde{U}^{n+1}_j), v_h \big)_{V' V} \\
			&= \langle \tilde{U}^{n+1}_j, v_h\rangle_{\Ldet},
		\end{align*}
		where we used the discrete DO condition \eqref{eq:discDO1}. Finally, the stochastic equation \eqref{eq:discDLReq3} can be written as
		\begin{multline*}
			\Big(\frac{\tilde{Y}^{n+1} - Y^n}{\triangle t}\Big) (\tilde{M}^{n+1}) +
			\underbrace{\Pro_{\hat{\rho},\mathcal{Y}^n}^{\perp} \Big[ \big( \Ldet^{*}(\tilde{U}^{n+1}\tilde{Y}^{{n+1}^\intercal}),\tilde{U}^{n+1} \big)_{V'V} \Big]}_{T_6} \\
			+ \underbrace{\Pro_{\hat{\rho},\mathcal{Y}^n}^{\perp} \Big[ \big( \Ldet^{*}(\bar{u}^{n+1}),\tilde{U}^{n+1} \big)_{V'V} \Big]}_{T_5}  \\= - \Pro_{\hat{\rho},\mathcal{Y}^n}^{\perp} \Big[ \big( \Lstoch^{*}(u^n_{h,\hat{\rho}}) - f^{{n,n+1}^{{*}}},\tilde{U}^{n+1} \big)_{V'V} \Big].
		\end{multline*} 
		The term $T_5$ vanishes since $\Ldet^*(\bar{u}^{n+1}) = 0$. As for $T_6$, we derive
		\begin{align*}
			%		T_6 =&\, \Pro_{\hat{\rho}\mathcal{Y}^n}^{\perp} \Big[ \big( \Ldet(\tilde{U}^{n+1}\tilde{Y}^{{n+1}^\intercal}),\tilde{U}^{n+1} \big)_{V'V} \Big] \\
			%		&\hspace{4.6cm}-\Pro_{\hat{\rho}\mathcal{Y}^n}^{\perp} \Big[ \big( \mathbb{E}_{\hat{\rho}}[\Ldet(\tilde{U}^{n+1}\tilde{Y}^{{n+1}^\intercal})],\tilde{U}^{n+1} \big)_{V'V} \Big]\\[3pt]
			T_6& =\, \big( \Ldet(\tilde{U}^{n+1})\tilde{Y}^{{n+1}^\intercal},\tilde{U}^{n+1} \big)_{V'V} \\
			&\hspace{1cm}- \big(\Ldet(\tilde{U}^{n+1})\underbrace{\mathbb{E}_{\hat{\rho}}[\tilde{Y}^{{n+1}^\intercal}Y^n]}_{\text{Id}}Y^{n^\intercal},\tilde{U}^{n+1} \big)_{V'V} \\
			&\hspace{1cm}-\Pro_{\hat{\rho},\mathcal{Y}^n}^{\perp} \Big[ \big( \Ldet(\tilde{U}^{n+1})\underbrace{\mathbb{E}_{\hat{\rho}}[\tilde{Y}^{{n+1}^\intercal}]}_{=0},\tilde{U}^{n+1} \big)_{V'V} \Big]\\
			%		=&\, \big( \Ldet(\tilde{U}^{n+1}),\tilde{U}^{n+1} \big)_{V'V}(\tilde{Y}^{{n+1}^\intercal} - Y^{n^\intercal})\\[3pt]
			&=\, \langle\tilde{U}^{n+1}, \tilde{U}^{n+1}\rangle_{\Ldet}(\tilde{Y}^{{n+1}^\intercal} - Y^{n^\intercal})
		\end{align*}
		which leads us to the sought equation \eqref{eq:semieq3}.
		\qed
	\end{proof}
	
	We see from \eqref{eq:semieq1}--\eqref{eq:semieq3} that, similarly to the explicit Euler scheme, the equations for the mean, deterministic modes and stochastic modes are decoupled. 
	If the spatial discretization of the PDEs \eqref{eq:semieq1} and \eqref{eq:semieq2} is performed by the Galerkin approximation, the final linear system involves the inversion of the matrix 
	$$A_{ij} = \langle \varphi_j, \varphi_i\rangle_H + \triangle t \langle\varphi_j, \varphi_i\rangle_{\Ldet},$$ where $\{\varphi_i\}$ is the basis of $V_h$ in which the solution is represented. Both the mass matrix $\langle \varphi_j, \varphi_i\rangle_H$ and the stiffness matrix $\langle\varphi_j, \varphi_i\rangle_{\Ldet}$ are positive definite and do not evolve with time, so that an LU factorization can be computed once and for all at the beginning of the simulation. Concerning the stochastic equation \eqref{eq:semieq3}, we need to solve a linear system with the matrix $\tilde{M}^{n+1} + \triangle t\langle\tilde{U}^{{n+1}^\intercal}, \tilde{U}^{n+1}\rangle_{\Ldet}$ for each collocation point $\omega_k$, unlike the explicit Euler method, where the system involves only the matrix $\tilde{M}^{n+1}$. The matrix $\tilde{M}^{n+1} + \triangle t\langle\tilde{U}^{{n+1}^\intercal}, \tilde{U}^{n+1}\rangle_{\Ldet}$ is symmetric and positive definite with the smallest singular value bigger than that of $\tilde{M}^{n+1}$. Notice, however, that if $\tilde{M}^{n+1}$ is rank deficient, also the matrix $\tilde{M}^{n+1} + \triangle t\langle\tilde{U}^{{n+1}^\intercal}, \tilde{U}^{n+1}\rangle_{\Ldet}$ will be so.
	
	Note that there exists a unique discrete DLR solution for the explicit and semi-implicit version of Algorithm~\ref{alg:ourscheme} also in the rank-deficient case (see Lemma~\ref{lemma:existence} below). The existence of solutions for the implicit version remains still an open question.
	
	\subsection{Discrete variational formulation for the full-rank case}
	
	\corr{This subsection will closely follow the geometrical interpretation introduced in Section \ref{sec:DLR}.} We will introduce analogous geometrical concepts for the discrete setting, i.e.\ manifold of $R$-rank functions, tangent space and orthogonal projection, and will show in Theorem \ref{th:discvarform} that the scheme from Algorithm \ref{alg:ourscheme} can be written in a (discrete) variational formulation, assuming that the matrix $\tilde{M}^{n+1}$ stays full-rank.
	
	\begin{definition}[Discrete manifold of $R$-rank functions]
		By $\mathcal{M}^{h,\hat{\rho}}_R\subset V_h\otimes\lhrz$ we denote the manifold of all rank $R$ functions with zero mean that belong to the (possibly finite dimensional) space $V_h\otimes\lhr$, namely
		\begin{align}\label{eq:discDLRman}
			\begin{split}
				&\mathcal{M}^{h,\hat{\rho}}_R = \Big\{v^*\in V_h\otimes\lhrz \,|\; v^* = \sum_{i=1}^R U_i Y_i,\quad \{Y_i\}_{i=1}^R \subset \lhrz\\
				&\hspace{0.4cm} \langle Y_i, Y_j\rangle_{\lhr} = \delta_{ij}, \;\forall \eR{i,j}, \, \{U_i \}_{i=1}^R\subset V_h \text{ linearly independent}\Big\}.
			\end{split}
		\end{align}	
	\end{definition}
	
	\begin{proposition}[Discrete tangent space at $UY^\intercal$]
		\sloppy{The tangent space $\mathcal{T}_{UY^\intercal}\mathcal{M}^{h,\hat{\rho}}_R$ at a point $U Y^\intercal \in\mathcal{M}^{h,\hat{\rho}}_R$ is formed as}
		\begin{align}\label{eq:discDLRtang}
			\begin{split}
				&\mathcal{T}_{UY^\intercal}\mathcal{M}^{h,\hat{\rho}}_R = \Big\{  \delta v\in V_h\otimes\lhrz\,|\, \delta v = \sum_{i=1}^R U_i\delta Y_i + \delta U_i Y_i,\\
				&\hspace{2.5cm}\delta U_j\in V_h,\;\delta Y_i\in\lhrz,\;\langle \delta Y_i,Y_j\rangle_{\lhr} = 0,\,\forall\eR{i,j}\Big\}.
			\end{split}
		\end{align}
	\end{proposition}

	The projection $\Pi^{h,\hat{\rho}}_{UY^\intercal}$ is defined in the discrete space $V_h\otimes \lhr$ analogously to its continuous version \eqref{eq:DLRproj}. It holds $$\Pi^{h,\hat{\rho}}_{UY^\intercal}:\, V_h\otimes\lhr\to \mathcal{T}_{UY^\intercal}\mathcal{M}^{h,\hat{\rho}}_R\;\subset V_h\otimes \lhr,\qquad\forall\, UY^\intercal\in\mathcal{M}_R^{h,\hat{\rho}}.$$
	A discrete analogue of Lemma~\ref{lemma:proj_symm} holds, i.e.\ 
	\begin{equation}\label{eq:opsymdisc}
		(\mathcal{K}, \Pi^{h,\hat{\rho}}_{U Y^\intercal}[v] )_{V'V,\lhr} = ( \Pi^{h,\hat{\rho}}_{U Y^\intercal} [\mathcal{K}], v)_{V'V,\lhr},\qquad \forall v\in V_h\otimes\lhr,\;\mathcal{K}\in V_h'\otimes\lhr.
	\end{equation}

	The solution of the proposed numerical scheme \eqref{eq:discDLReq1}--\eqref{eq:discDLReq4} satisfies a discrete variational formulation analogous to the variational formulation \eqref{eq:DLRvar}. To show this, we first present a technical lemma which will be important in deriving the variational formulation as well as in the stability analysis presented in Section~\ref{sec:stabest}.
	
	\begin{lemma}\label{lemma:unun+1}
		Let $u^{n}_{h,\hat{\rho}}, u^{n+1}_{h,\hat{\rho}}$ be the discrete DLR solution at $t_n, t_{n+1}$, respectively, from the scheme in Algorithm \ref{alg:ourscheme}. Then the zero-mean parts $u^{n,*}_{h,\hat{\rho}}, u^{n+1,*}_{h,\hat{\rho}}$ satisfy
		\begin{enumerate}
			\item $u^{{n}^*}_{h,\hat{\rho}}\in \mathcal{T}_{\tilde{U}^{n+1}Y^{n^\intercal}} \mathcal{M}_R^{h,\hat{\rho}},$
			\item $u^{n+1,*}_{h,\hat{\rho}}\in \mathcal{T}_{\tilde{U}^{n+1}Y^{n^\intercal}} \mathcal{M}_R^{h,\hat{\rho}}$.
		\end{enumerate}
	\end{lemma}
	\begin{proof}
		\begin{enumerate}
			\item The solution $u^{n,*}_{h,\hat{\rho}}$ can be written as $$u^{n,*}_{h,\hat{\rho}} = \tilde{U}^{n+1} 0^\intercal + U^n Y^{n^\intercal}.$$ Since $\langle 0^\intercal, Y^n\rangle_{\lhr} = 0$, using the definition \eqref{eq:discDLRtang} we have $$u^{n,*}_{h,\hat{\rho}}\in\mathcal{T}_{\tilde{U}^{n+1}Y^{n^\intercal}} \mathcal{M}_R^{h,\hat{\rho}}.$$
			\item The newly computed solution $u^{n+1,*}_{h,\hat{\rho}}$ can be expressed as $$u^{n+1,*}_{h,\hat{\rho}} = \tilde{U}^{n+1}(\tilde{Y}^{{n+1}} - Y^{n})^\intercal + \tilde{U}^{n+1}Y^{n^\intercal}.$$ Based on \eqref{eq:discDO1} in Lemma \ref{lemma:discprop}, we know that $\langle\tilde{Y}^{{n+1}^\intercal} - Y^{n^\intercal}, Y^n\rangle_{\lhr} = 0$, i.e.\ again using the definition \eqref{eq:discDLRtang} we have $u^{n+1,*}_{h,\hat{\rho}}\in\mathcal{T}_{\tilde{U}^{n+1}Y^{n^\intercal}} \mathcal{M}_R^{h,\hat{\rho}}$.
		\end{enumerate}
		\qed
	\end{proof}
	
	\begin{remark}
		Note that for any function of the form $v = \tilde{U}^{n+1} K^\intercal$ or $v = J Y^{n^\intercal}$ with $K\in (\lhr)^R$, $J\in (V_h)^R$, it holds $v\in \mathcal{T}_{\tilde{U}^{n+1}Y^{n^\intercal}} \mathcal{M}_R^{h,\hat{\rho}}$ since we have $$J Y^{n^\intercal} = \tilde{U}^{n+1} 0^\intercal + J Y^{n^\intercal},\quad \mathbb{E}_{\hat{\rho}}[0^\intercal Y^n] = 0$$ $$\tilde{U}^{n+1} K^\intercal = \tilde{U}^{n+1}(\Pro_{\hat{\rho},Y^n}^{\perp}[K])^\intercal + \tilde{U}^{n+1}(\Pro_{\hat{\rho},Y^n}[K])^\intercal,\quad \langle(\Pro_{\hat{\rho},Y^n}^{\perp}[K])^\intercal, Y^n\rangle_{\lhr} = 0.$$
	\end{remark}
	
	Since $\mathcal{T}_{\tilde{U}^{n+1}Y^{n^\intercal}} \mathcal{M}_R^{h,\hat{\rho}}$ is a vector space, it includes any linear combination of $u^{n,*}_{h,\hat{\rho}}$ and $u^{n+1,*}_{h,\hat{\rho}}$. The following lemma is an analogue of Lemma \ref{lemma:DLRproj} and will become useful when we derive the discrete variational formulation.
	\begin{lemma}\label{lemma:DLRprojdisc}
		Let $u^{n}_{h,\hat{\rho}}, u^{n+1}_{h,\hat{\rho}}$ be the discrete DLR solutions at times $t_{n}, t_{n+1}$ as defined in Algorithm~\ref{alg:ourscheme}. Then the zero-mean parts $u^{{n+1}^*}_{h,\hat{\rho}}, u^{n^*}_{h,\hat{\rho}}$ satisfy
		\begin{multline}\label{eq:projeqsum}
			\Big( \frac{(u^{n+1}_{h,\hat{\rho}} - u^n_{h,\hat{\rho}})^*}{\triangle t} + \Pi_{\tilde{U}^{n+1}Y^{n^\intercal}}^{h,\hat{\rho}}[\mathcal{L}^*(u^n_{h,\hat{\rho}}, u^{n+1}_{h,\hat{\rho}}) - f^{{n,n+1}^*}], v_h\Big)_{V'V,\lhr} = 0\\
			\forall v_h\in V_h\otimes\lhr
		\end{multline}
	\end{lemma}
	\begin{proof}
		Multiplying \eqref{eq:discDLReq2} by $Y^n_j$ and summing over $j$, we obtain 
		\begin{multline}\label{eq:proof1}
			\Big\langle \frac{\tilde{U}^{n+1}Y^{n^\intercal} - u^{n,*}_{h,\hat{\rho}}}{\triangle t}, v_h\Big\rangle_H + \Big(\mathbb{E}_{\hat{\rho}}[(\mathcal{L}(u^n_{h,\hat{\rho}}, u^{n+1}_{h,\hat{\rho}}) - f^{n,n+1}) Y^n]Y^{n^\intercal}, v_h \Big)_{V' V} = 0 \\ \quad \forall v_h \in V_h.
		\end{multline}
		Noticing that \begin{align*}\mathbb{E}_{\hat{\rho}}[(\mathcal{L}(u^n_{h,\hat{\rho}}, u^{n+1}_{h,\hat{\rho}})-f^{n,n+1}) Y^n]Y^{n^\intercal} &= \mathbb{E}_{\hat{\rho}}[(\mathcal{L}^*(u^n_{h,\hat{\rho}}, u^{n+1}_{h,\hat{\rho}})-f^{{n,n+1}^*}) Y^n]Y^{n^\intercal}\\ &= \Pro_{\hat{\rho},\mathcal{Y}^n}[\mathcal{L}^*(u^n_{h,\hat{\rho}}, u^{n+1}_{h,\hat{\rho}}) - f^{{n,n+1}^*}],\end{align*}
		and taking the weak formulation of \eqref{eq:proof1} in $\lhr$ results in 
		\begin{multline}\label{eq:tmpprojeq1}
			\langle\tilde{U}^{n+1}Y^{n^\intercal},v_h \rangle_{H,\lhr} = \langle u^{n,*}_{h,\hat{\rho}},v_h\rangle_{H,\lhr} \\+ \triangle t\big( \Pro_{\hat{\rho},\mathcal{Y}^n}[f^{{n,n+1}^*} - \mathcal{L}^*(u^n_{h,\hat{\rho}}, u^{n+1}_{h,\hat{\rho}})],v_h\big)_{V'V,\lhr}
			\quad\forall v_h\in V_h\otimes \lhr.
		\end{multline}
		Similarly, multiplying~\eqref{eq:discDLReq3} by $\tilde{U}^{n+1}$, and further writing \eqref{eq:discDLReq3} in a weak form in $\lhr$, we obtain
		\begin{multline}\label{eq:proof2}
			\Big\langle \frac{u_{h,\hat{\rho}}^{n+1,*} - \tilde{U}^{n+1} Y^{n^\intercal}}{\triangle t}\\
			+ \Pro_{\hat{\rho},\mathcal{Y}^n}^{\perp} \Big[ \big( \mathcal{L}^*(u^n_{h,\hat{\rho}}, u^{n+1}_{h,\hat{\rho}}) - f^{{n,n+1}^*},\tilde{U}^{n+1} \big)_{V'V} (\tilde{M}^{n+1})^{-1}\tilde{U}^{{n+1}^\intercal} \Big],\,w\Big\rangle_{\lhr}= 0,\\
			\quad\forall w\in\lhr.
		\end{multline}
		Since
		\begin{multline*}
			\Pro_{\hat{\rho},\mathcal{Y}^n}^{\perp} \Big[ \big( \mathcal{L}^*(u^n_{h,\hat{\rho}}, u^{n+1}_{h,\hat{\rho}}) - f^{{n,n+1}^*},\tilde{U}^{n+1} \big)_{V'V} (\tilde{M}^{n+1})^{-1}\tilde{U}^{{n+1}^\intercal} \Big]\\ = \Pro_{\hat{\rho},\mathcal{Y}^n}^{\perp}[\Pro_{\mathcal{\tilde{U}}^{n+1}}[\mathcal{L}^*(u^n_{h,\hat{\rho}}, u^{n+1}_{h,\hat{\rho}})-f^{{n,n+1}^*}]],
		\end{multline*}
		taking the weak formulation of \eqref{eq:proof2} in $V_h$ results in 
		\begin{multline}\label{eq:tmpprojeq2}
			\langle u^{n+1,*}_{h,\hat{\rho}},v_h\rangle_{H,\lhr} = \langle\tilde{U}^{n+1}Y^{n^\intercal},v_h\rangle_{H,\lhr} 
			\\+\triangle t\big( \Pro_{\hat{\rho},\mathcal{Y}^n}^\perp\big[\Pro_{\tilde{\mathcal{U}}^{n+1}}[f^{{n,n+1}^*} - \mathcal{L}^*(u^n_{h,\hat{\rho}}, u^{n+1}_{h,\hat{\rho}})]\big],v_h\big)_{V'V,\lhr} \quad\forall v_h\in V_h\otimes \lhr.
		\end{multline}	
		Finally, summing equations \eqref{eq:tmpprojeq1} and \eqref{eq:tmpprojeq2} results in \eqref{eq:projeqsum}.
		\qed
	\end{proof}
	
	We now proceed with the discrete variational formulation.
	
	\begin{theorem}[Discrete variational formulation]\label{th:discvarform}
		Let $u^n_{h,\hat{\rho}}$ and $u^{n+1}_{h,\hat{\rho}}$ be the discrete DLR solution at times $t_n$, $t_{n+1}$, respectively, $n=0,\dots,N-1$, as defined in Algorithm~\ref{alg:ourscheme}. Then it holds
		\begin{align}\label{eq:discDLRvar}
			\begin{split}
				&\hspace{0cm}\big\langle \frac{u^{n+1}_{h,\hat{\rho}} - u^n_{h,\hat{\rho}}}{\triangle t},\, v_h\big\rangle_{H,\lhr} + \big(\mathcal{L}(u^n_{h,\hat{\rho}},u^{n+1}_{h,\hat{\rho}}),\, v_h  \big)_{V'V,\lhr} = \big\langle f^{n,n+1},\,v_h\big\rangle_{H,\lhr},\\
				&\hspace{2.5cm} \forall v_h = \bar{v}_h + v_h^* \text{ with } \bar{v}_h\in V_h\text{ and }v_h^*\in \mathcal{T}_{\tilde{U}^{n+1}Y^{n^\intercal}} \mathcal{M}_R^{h,\hat{\rho}}.
			\end{split}
		\end{align}
	\end{theorem}
	\begin{proof}
		Thanks to Lemma \ref{lemma:unun+1} we have $(u^{n+1}_{h,\hat{\rho}} - u^n_{h,\hat{\rho}})^* \in \mathcal{T}_{\tilde{U}^{n+1}Y^n} \mathcal{M}_R^{h,\hat{\rho}}$, and we can derive 
		\begin{align}\label{eq:proof3}
			\begin{split}
				\Big\langle \frac{(u^{n+1}_{h,\hat{\rho}} - u^n_{h,\hat{\rho}})^*}{\triangle t}, v_h\Big\rangle_{H,\lhr} &= \Big\langle \Pi_{\tilde{U}^{n+1}Y^{n^\intercal}}^{h,\hat{\rho}}\Big[\frac{(u^{n+1}_{h,\hat{\rho}} - u^n_{h,\hat{\rho}})^*}{\triangle t}\Big], v_h\Big\rangle_{H,\lhr} \\
				&= \Big\langle \frac{(u^{n+1}_{h,\hat{\rho}} - u^n_{h,\hat{\rho}})^*}{\triangle t}, \Pi_{\tilde{U}^{n+1}Y^{n^\intercal}}^{h,\hat{\rho}}[v_h]\Big\rangle_{H,\lhr}
			\end{split}
		\end{align}
		and formula \eqref{eq:opsymdisc} gives us 
		\begin{multline}\label{eq:proof4}
			\Big( \Pi_{\tilde{U}^{n+1}Y^{n^\intercal}}^{h,\hat{\rho}}[\mathcal{L}^*(u^n_{h,\hat{\rho}}, u^{n+1}_{h,\hat{\rho}}) - f^{{n,n+1}^*}],\, v_h\Big)_{V'V,\lhr} \\
			= \Big( \mathcal{L}^*(u^n_{h,\hat{\rho}}, u^{n+1}_{h,\hat{\rho}}) - f^{{n,n+1}^*},\, \Pi_{\tilde{U}^{n+1}Y^{n^\intercal}}^{h,\hat{\rho}}[v_h]\Big)_{V'V,\lhr}.
		\end{multline}
		Summing \eqref{eq:proof3}, \eqref{eq:proof4} and applying Lemma \ref{lemma:DLRprojdisc} results in 
		\begin{multline*}
			\big\langle \frac{(u^{n+1}_{h,\hat{\rho}} - u^n_{h,\hat{\rho}})^*}{\triangle t},\, v_h\big\rangle_{H,\lhr} + \big(\mathcal{L}^*(u^n_{h,\hat{\rho}},u^{n+1}_{h,\hat{\rho}}) - f^{{n,n+1}^*},\, v_h  \big)_{V'V,\lhr} = 0 \\
			\forall v_h\in\mathcal{T}_{\tilde{U}^{n+1}Y^{n^\intercal}} \mathcal{M}_R^{h,\hat{\rho}}.
		\end{multline*}
		Now summing this to equation \eqref{eq:discDLReq1} we obtain 
		\begin{multline}\label{eq:tmpeq2}
			\Big\langle \frac{\bar{u}^{n+1}_{h,\hat{\rho}} - \bar{u}^n_{h,\hat{\rho}} + (u^{n+1}_{h,\hat{\rho}} - u^n_{h,\hat{\rho}})^*}{\triangle t},\, w_h + v_h\Big\rangle_{H,\lhr} \\+ \big(\mathbb{E}_{\hat{\rho}}[\mathcal{L}(u^n_{h,\hat{\rho}},u^{n+1}_{h,\hat{\rho}}) - f^{n,n+1}] + \mathcal{L}^*(u^n_{h,\hat{\rho}},u^{n+1}_{h,\hat{\rho}}) - f^{{n,n+1}^*},\, w_h + v_h  \big)_{V'V,\lhr} \\= 0\quad\forall w_h\in V_h,\; \forall v_h \in\mathcal{T}_{\tilde{U}^{n+1}Y^{n^\intercal}} \mathcal{M}_R^{h,\hat{\rho}}
		\end{multline}
		which is equivalent to the final result \eqref{eq:discDLRvar}. In \eqref{eq:tmpeq2} we have employed
		\begin{align*}
			& \Big\langle \frac{(u^{n+1}_{h,\hat{\rho}} - u^n_{h,\hat{\rho}})^*}{\triangle t},\, w_h \Big\rangle_{H,\lhr} + \big(\mathcal{L}^*(u^n_{h,\hat{\rho}},u^{n+1}_{h,\hat{\rho}}) - f^{{n,n+1}^*},\, w_h \big)_{V'V,\lhr} = 0\\
			& \hspace{7cm}\forall w_h\in V_h\\
			& \Big\langle \frac{\bar{u}^{n+1}_{h,\hat{\rho}} - \bar{u}^n_{h,\hat{\rho}}}{\triangle t},\, v_h\Big\rangle_{H,\lhr} + \big(\mathbb{E}_{\hat{\rho}}[\mathcal{L}(u^n_{h,\hat{\rho}},u^{n+1}_{h,\hat{\rho}}) - f^{n,n+1}],\, v_h  \big)_{V'V,\lhr} = 0\\
			& \hspace{7cm}\forall v_h \in\mathcal{T}_{\tilde{U}^{n+1}Y^{n^\intercal}} \mathcal{M}_R^{h,\hat{\rho}},
		\end{align*}
		which holds as $\mathbb{E}[v_h] = 0,\; \forall v_h \in\mathcal{T}_{\tilde{U}^{n+1}Y^{n^\intercal}} \mathcal{M}_R^{h,\hat{\rho}}$.
		\qed
	\end{proof}
	
	\begin{remark}
		The preceding theorem applies to a discretization of any kind of the operator $\mathcal{L}\in \lrVdual$, not necessarily elliptic or linear, as assumed in Section~\ref{sec:prstatement}, as long as Lemma~\ref{lemma:proj_symm} holds.
	\end{remark}
	
	\subsection{Discrete variational formulation for the rank-deficient case}\label{sec:discvar_rankdef}
	%The established discrete variational formulation is only valid in the case of the deterministic basis $\tilde{U}^{n+1}$ being linearly independent
	The discrete variational formulation established in the previous section is valid only in the case of the deterministic basis $\tilde{U}^{n+1}$ being linearly independent, since the proof of Theorem \ref{th:discvarform} implicitly involves the inverse of ${\tilde{M}^{n+1}} = \langle \tilde{U}^{{n+1}^\intercal}, \tilde{U}^{n+1}\rangle_H$.  In this subsection, we show that a discrete variational formulation can be generalized for the rank-deficient case.
	
	When applying the discretization scheme proposed in step 3 of Algorithm~\ref{alg:ourscheme} with a rank-deficient matrix $\tilde{M}^{n+1}$, we recall that the solution $\tilde{Y}^{n+1}$  is defined as the solution of \eqref{eq:discDLReq3} minimizing $\|\tilde{Y}^{n+1} - Y^n\|_{\lhr}$. Note that minimizing $\|\tilde{Y}^{n+1} - Y^n\|_{\lhr}$ is equivalent to minimizing the norm $\|\tilde{Y}^{n+1}(\omega_k) - Y^n(\omega_k)\|_{\R^R}$ for every sample point $\omega_k,\, k=1,\dots,\hat{N}$, where $\|\cdot\|_{\R^R}^2 = \langle \cdot,\cdot\rangle_{\R^R}$ denotes the Euclidean scalar product in $\R^R$.
	
	In what follows we will exploit the fact that the vector space $\lhr$ is isomorphic to $\R^{\hat{N}}$. In particular, it holds that $(\tilde{Y}^{n+1} - Y^n)^\intercal\in\R^{R\times\hat{N}}$, where each column of $(\tilde{Y}^{n+1} - Y^n)^\intercal$ is given by $(\tilde{Y}^{n+1} - Y^n)(\omega_k),\,k=1,\dots,\hat{N}$. With a little abuse of notation, we use $\tilde{U}^{n+1}:\, \R^R\to V_h$ to denote a linear operator which takes real coefficients and returns the corresponding linear combination of the basis functions $\tilde{U}^{n+1}$. By $\tilde{U}^{{n+1}^\intercal}: V_h \to \R^R$ we denote its dual.
	
	\begin{lemma}\label{lemma:kernel}
		For any discrete solution $\tilde{Y}^{n+1}$ of equation \eqref{eq:discDLReq3} that minimizes the norm $\|\tilde{Y}^{n+1} - Y^n\|_{\lhr}$, it holds that every column of the increment $(\tilde{Y}^{n+1} - Y^n)^{\intercal}$ lies in the $\langle \cdot,\cdot\rangle_{\R^R}$-orthogonal complement of the kernel of $\tilde{M}^{n+1}$, i.e.
		\begin{gather*}
			(\tilde{Y}^{n+1} - Y^n)^{\intercal} \;\in\; \big(\kernel(\tilde{M}^{n+1})^\perp\big)^{\hat{N}},
		\end{gather*}
		where $\kernel(\tilde{M}^{n+1}) = \{ v\in \R^R:\, \tilde{M}^{n+1} v = 0 \}.$
	\end{lemma}
	\begin{proof}
		\sloppy{Seeking a contradiction, let us suppose that $(\tilde{Y}^{n+1} - Y^n)^{\intercal} \notin \big(\kernel(\tilde{M}^{n+1})^\perp\big)^{\hat{N}}$.} Let
		\begin{equation}\label{eq:tmpeq12}
			Z^\intercal := \tilde{Y}^{{n+1}^\intercal} - \mathcal{P}_{\kernel(\tilde{M}^{n+1})}[\tilde{Y}^{{n+1}^\intercal} - Y^{n^\intercal}]
			\;\neq\; \tilde{Y}^{{n+1}},
		\end{equation}
		where $\mathcal{P}_{\kernel(\tilde{M}^{n+1})}[v]\in\R^{R\times\hat{N}}$ for  $v\in\R^{R\times\hat{N}}$ denotes the column-wise application of $\langle \cdot,\cdot\rangle_{\R^R}$-orthogonal projection onto the kernel of $\tilde{M}^{n+1}$. 
		Then, such constructed $Z$ satisfies
		\begin{multline*}
			\|(Z - Y^n)(\omega_k)\|_{\R^R} = \|\big(\tilde{Y}^{n+1} - Y^n - \mathcal{P}_{\kernel(\tilde{M}^{n+1})}[\tilde{Y}^{n+1} - Y^n]\big)(\omega_k)\|_{\R^R}  \\
			< \|(\tilde{Y}^{n+1} - Y^n)(\omega_k)\|_{\R^R},
		\end{multline*}
		and solves  \eqref{eq:discDLReq3}:
		\begin{align*}	
			\tilde{M}^{n+1}& (Z - Y^n)^\intercal %= \tilde{M}^{n+1} \big( \tilde{Y}^{n+1} -  \mathcal{P}_{\kernel(\tilde{M}^{n+1})}[\tilde{Y}^{n+1}-Y^n] - Y^n  \big)^\intercal \\
			%&
			=  \tilde{M}^{n+1} (\tilde{Y}^{n+1} - Y^n)^\intercal \\
			&= -\triangle t \Pro_{\hat{\rho},\mathcal{Y}^n}^{\perp} \Big[ \big( \mathcal{L}^{*}(u^n_{h,\hat{\rho}}, \bar{u}^{n+1} + \tilde{U}^{n+1}\tilde{Y}^{{n+1}^\intercal}) - f^{{n,n+1}^*},\tilde{U}^{n+1} \big)_{V'V}\Big]\\
			&= -\triangle t \Pro_{\hat{\rho},\mathcal{Y}^n}^{\perp} \Big[ \big( \mathcal{L}^{*}(u^n_{h,\hat{\rho}}, \bar{u}^{n+1} + \tilde{U}^{n+1}Z^\intercal) - f^{{n,n+1}^*},\tilde{U}^{n+1} \big)_{V'V}\Big],
		\end{align*}
		where in the last step we used that $\kernel(\tilde{M}^{n+1}) = \kernel(\tilde{U}^{n+1})$.
		This leads to a contradiction that $\tilde{Y}^{n+1}$ was the  solution minimizing $\|\tilde{Y}^{n+1} - Y^n\|_{\lhr}$.
		\qed
	\end{proof}
	
	When showing the equivalence between the DLR variational formulation \eqref{eq:DLRvar} and the DLR system of equations \eqref{eq:DLReq1}--\eqref{eq:DLReq3} in the continuous setting, the DO condition \eqref{eq:DOcond} plays an important role. In an analogous way, the discrete DO condition (property 1 from Lemma \ref{lemma:discprop} for the full-rank case) plays an important role when showing the equivalence between the discrete DLR system of equations and the discrete DLR variational formulation.
	\begin{lemma}\label{lemma:discdo}
		Any discrete solution $\tilde{Y}^{n+1}$ of equation \eqref{eq:discDLReq3} which minimizes the norm $\|\tilde{Y}^{n+1} - Y^n\|_{\lhr}$, satisfies the discrete DO condition 
		\begin{equation}\label{eq:discDO}
			\Big\langle\Big(\frac{\tilde{Y}^{n+1} - Y^n}{\triangle t}\Big)^{\intercal}, Y^n\Big\rangle_{\lhr} = 0.
		\end{equation}
	\end{lemma}
	\begin{proof}
		Let $\tilde{Y}^{n+1}$ be a solution of \eqref{eq:discDLReq3} minimizing $\|\tilde{Y}^{n+1} - Y^n\|_{\lhr}$. Thanks to Lemma~\ref{lemma:kernel} we know that \[\big(\tilde{Y}^{n+1} - Y^n\big)^{\intercal} \in \big(\kernel(\tilde{M}^{n+1})^\perp\big)^{\hat{N}}.\]
		Now, let $\tilde{M}^{{n+1}^+}$ denote the pseudoinverse of $\tilde{M}^{n+1}$. Since \[\tilde{M}^{{n+1}^+} \tilde{M}^{n+1} v = v\] for any $v \in \kernel(\tilde{M}^{n+1})^\perp$, the solution $\tilde{Y}^{n+1}$ of equation \eqref{eq:discDLReq3} satisfies
		\begin{equation}\label{eq:tmpeq9}
			\tilde{Y}^{{n+1}^{\intercal}} = Y^{n^\intercal} - \triangle t
			\tilde{M}^{{n+1}^{+}} \Pro_{\hat{\rho},\mathcal{Y}^n}^{\perp} \Big[ \big( \mathcal{L}^{*}(u^n_{h,\hat{\rho}}, u^{n+1}_{h,\hat{\rho}}) - f^{{n,n+1}^*},\tilde{U}^{n+1} \big)^\intercal_{V'V}\Big].
		\end{equation}
		Thus, if we have
		\[
		\mathbb{E}_{\hat{\rho}}\Big[
		Y^{n^\intercal}
		\Big(\Pro_{\hat{\rho},\mathcal{Y}^n}^{\perp} \Big[ \big( \mathcal{L}^{*}(u^n_{h,\hat{\rho}}, u^{n+1}_{h,\hat{\rho}}) - f^{{n,n+1}^*},\tilde{U}^{n+1} \big)_{V'V}\Big] \tilde{M}^{{n+1}^+} \Big)
		\Big]=0,
		\]
		then the statement will follow. 
		\sloppy{But for the column space of $\Pro_{\hat{\rho},\mathcal{Y}^n}^{\perp} \Big[ \big( \mathcal{L}^{*}(u^n_{h,\hat{\rho}}, u^{n+1}_{h,\hat{\rho}}) - f^{{n,n+1}^*},\tilde{U}^{n+1} \big)_{V'V}\Big] \tilde{M}^{{n+1}^+} \in \R^{\hat{N}\times R}$ it holds}
		\begin{multline*}
			\spn\Big\{\Pro_{\hat{\rho},\mathcal{Y}^n}^{\perp} \Big[ \big( \mathcal{L}^{*}(u^n_{h,\hat{\rho}}, u^{n+1}_{h,\hat{\rho}}) - f^{{n,n+1}^*},\tilde{U}^{n+1} \big)_{V'V}\Big] \tilde{M}^{{n+1}^+}\Big\} \\
			\subset \spn \Big\{\Pro_{\hat{\rho},\mathcal{Y}^n}^{\perp} \Big[ \big( \mathcal{L}^{*}(u^n_{h,\hat{\rho}}, u^{n+1}_{h,\hat{\rho}}) - f^{{n,n+1}^*},\tilde{U}^{n+1} \big)_{V'V}\Big]\Big\} \subset \mathcal{Y}^{n^\perp}_{\hat{\rho}}
		\end{multline*}
		with $\mathcal{Y}^{n^\perp}_{\hat{\rho}}\subset\R^{\hat{N}}$ being the orthogonal complement to $\mathcal{Y}^n$ in the scalar product $\langle\cdot,\cdot\rangle_{\lhr}$. Now the proof is complete. %Multiplying \eqref{eq:tmpeq9} by $Y^n$ and taking the expectation (w.r.t. $\hat{\rho}$) leads to \[ \mathbb{E}_{\hat{\rho}}[\tilde{Y}^{{n+1}^\intercal}Y^n] = \mathbb{E}_{\hat{\rho}}[{Y}^{{n}^\intercal}Y^n] = \text{Id}\] which is equivalent to the discrete DO condition \eqref{eq:discDO}.
		\qed
	\end{proof}
	
	In the following lemma we address the question of existence of a unique solution when applying the explicit and semi-implicit scheme.
	\begin{lemma}\label{lemma:existence}
		For the explicit and semi-implicit scheme, as described in Section~\ref{sec:fullydiscpr}, there exists a unique discrete solution $\tilde{Y}^{n+1}$ of equation \eqref{eq:discDLReq3} minimizing the norm $\|\tilde{Y}^{n+1} - Y^n\|_{\lhr}$.
	\end{lemma}
	\begin{proof}
		We will start with the semi-implicit scheme. By virtue of Lemma~\ref{lemma:semiimplscheme}, under the discrete DO condition \eqref{eq:discDO}, applying the semi-implicit scheme to equation \eqref{eq:discDLReq3} is equivalent to solving equation \eqref{eq:semieq3}. We will first focus our attention to  equation \eqref{eq:semieq3} and show that there exists a unique solution minimizing $\|\tilde{Y}^{n+1} - Y^n\|_{\lhr}$. This solution will satisfy the discrete DO and consequently is a unique minimizing solution of \eqref{eq:discDLReq3}.
		Equation \eqref{eq:semieq3} can be rewritten as
		\begin{equation}\label{eq:tmpeq10}
			B \,(\tilde{Y}^{n+1} - Y^n)^\intercal = \text{RHS}\hspace{1cm}\text{in }\lhr,
		\end{equation}
		where
		\begin{align*}
			B &= \tilde{M}^{n+1} + \triangle t\langle\tilde{U}^{{n+1}^\intercal}, \tilde{U}^{n+1}\rangle_{\Ldet},\\
			\text{RHS} &= -\triangle t \big(\tilde{U}^{{n+1}^\intercal}, \Pro_{\hat{\rho},\mathcal{Y}^n}^{\perp} [\Lstoch^{*}(u^n_{h,\hat{\rho}})-f^{{n,n+1}^{*}}]\big)_{V V'}.
		\end{align*}
		Since $\text{RHS}$ above lies in the range of $\tilde{U}^{{n+1}^\intercal}$, which is the same as the range of $B$, a solution of \eqref{eq:tmpeq10} exists. 
		Moreover, since the matrix $B$ is positive definite on the space $\kernel(B)^\perp$, any solution can be expressed as $(\tilde{Y}^{n+1} - Y^n + W)^{\intercal}$ with $W^{\intercal}\in \big(\kernel(B)\big)^{\hat{N}}$ and a unique  $\tilde{Y}^{n+1^{\intercal}}\in \R^{R\times \hat{N}}$ such that $(\tilde{Y}^{n+1} - Y^n)^{\intercal}\in \big(\kernel(B)^\perp\big)^{\hat{N}}$.
		The solution $\tilde{Y}^{n+1}$ minimizes each column $\|(\tilde{Y}^{n+1} - Y^n)(\omega_k)\|_{\R^R}$, $k=1,\dots,\hat{N}$ and thus it is the unique solution of \eqref{eq:tmpeq10} that minimizes
		norm $\|\tilde{Y}^{n+1} - Y^n\|_{\lhr}$.
		We observe that the established solution $\tilde{Y}^{n+1}$ of equation \eqref{eq:tmpeq10} satisfies the discrete DO condition \eqref{eq:discDO}. The argument is analogous to the proof of Lemma~\ref{lemma:discdo}, but instead of $\tilde{M}^{n+1}$ here we take ${B}$. 
		Therefore, the statement for the semi-implicit scheme follows.
		The explicit case can be shown by following analogous steps with 
		\begin{align*}
			B &= \tilde{M}^{n+1},\\
			\text{RHS} &= -\triangle t \big(\tilde{U}^{{n+1}^\intercal}, \Pro_{\hat{\rho},\mathcal{Y}^n}^{\perp} [\mathcal{L}^{*}(u^n_{h,\hat{\rho}})-f^{{n,n+1}^{*}}]\big)_{V V'}.
		\end{align*}
		\qed
	\end{proof}
	Now we can proceed with showing the discrete variational formulation. It is not generally easy to deal with the notion of a tangent space at a certain point on the manifold in the rank-deficient case. In the following theorem we will, however, show that an analogous discrete variational formulation holds. Given $ U\in (V_h)^R$ and $Y\in(\lhrz)^R$, we define the vector space $\mathcal{T}_{UY^\intercal}$ as 
	\begin{multline*}
		\mathcal{T}_{UY^\intercal} = \Big\{  \delta v\in V_h\otimes\lhrz\,| \;\delta v =  \sum_{i=1}^R U_i\delta Y_i + \delta U_i Y_i\\ 
		\delta U_i\in V_h,\;\delta Y_i\in\lhrz,\;\langle \delta Y_i,Y_j\rangle_{\lhr} = 0 \quad\forall i,j = 1,\dots,R \Big\}.
	\end{multline*}
	It is easy to verify that, analogously to Lemma \ref{lemma:unun+1}, the (possibly rank-deficient) discrete DLR solutions $u^n_{h,\hat{\rho}}$ and $u^{n+1}_{h,\hat{\rho}}$ at times $t_n, t_{n+1}$, as defined in Algorithm~\ref{alg:ourscheme} satisfy 
	\begin{equation}\label{rem:unun+1}
		u^n_{h,\hat{\rho}} \in \mathcal{T}_{\tilde{U}^{n+1}Y^{n^\intercal}},\qquad u^{n+1}_{h,\hat{\rho}} \in  \mathcal{T}_{\tilde{U}^{n+1}Y^{n^\intercal}}.
	\end{equation}

	\begin{theorem}\label{th:discvarformrankdef}
		Let $u^n_{h,\hat{\rho}}$ and $u^{n+1}_{h,\hat{\rho}}$ be the (possibly rank-deficient) discrete DLR solution at times $t_n$, $t_{n+1}$, respectively, $n=0,\dots,N-1$, as defined in Algorithm~\ref{alg:ourscheme}. Then the following  variational formulation holds
		\begin{align}\label{eq:discDLRvarrankdef}
			\begin{split}
				&\hspace{0cm}\big\langle \frac{u^{n+1}_{h,\hat{\rho}} - u^n_{h,\hat{\rho}}}{\triangle t},\, v_h\big\rangle_{H,\lhr} + \big(\mathcal{L}(u^n_{h,\hat{\rho}},u^{n+1}_{h,\hat{\rho}}),\, v_h  \big)_{V'V,\lhr} = \big\langle f^{n,n+1},\,v_h\big\rangle_{H,\lhr},\\
				&\hspace{3cm} \forall v_h = \bar{v}_h + v_h^* \text{ with } \bar{v}_h\in V_h \text{ and }\;v_h^*\in \mathcal{T}_{\tilde{U}^{n+1}Y^{n^\intercal}}.
			\end{split}
		\end{align}
	\end{theorem}
	\begin{proof}
		First, consider equation \eqref{eq:discDLReq2} with $v_h = \tilde{U}^{n+1}_j$. Summing over $j$ results in 
		\begin{equation}\label{eq:tmpeq6}
			\big(\mathbb{E}_{\hat{\rho}}[(\mathcal{L}^*(u^n_{h,\hat{\rho}},u^{n+1}_{h,\hat{\rho}})-f^{{n,n+1}^*})Y^n], \tilde{U}^{n+1}\big)_{V'V} = \frac{1}{\triangle t}\Big(\langle U^{n^\intercal},\tilde{U}^{{n+1}}\rangle_H -  \tilde{M}^{n+1} \Big).
		\end{equation}
		Let us proceed with the equation \eqref{eq:discDLReq3}:
		\begin{align*}
			0 &= \frac{\tilde{Y}^{n+1} - Y^n}{\triangle t} \tilde{M}^{n+1} + \Pro_{\hat{\rho},\mathcal{Y}^n}^{\perp} \Big[ \big( \mathcal{L}^{*}(u^n_{h,\hat{\rho}}, u^{n+1}_{h,\hat{\rho}}) - f^{{n,n+1}^*},\tilde{U}^{n+1} \big)_{V'V}\Big]\\
			& =  \frac{\tilde{Y}^{n+1} - Y^n}{\triangle t} \tilde{M}^{n+1} + \big( \mathcal{L}^{*}(u^n_{h,\hat{\rho}}, u^{n+1}_{h,\hat{\rho}}) - f^{{n,n+1}^*},\tilde{U}^{n+1} \big)_{V'V} \\
			&\hspace{3cm}-Y^n \big(\mathbb{E}_{\hat{\rho}}[(\mathcal{L}^*(u^n_{h,\hat{\rho}}u^{n+1}_{h,\hat{\rho}})-f^{{n,n+1}^*})Y^{n^\intercal}] ,\tilde{U}^{n+1}\big)_{V'V}\\
			& = \frac{\tilde{Y}^{n+1}\langle \tilde{U}^{{n+1}^\intercal},\tilde{U}^{{n+1}}\rangle_H - {Y}^{n}\tilde{M}^{n+1} + Y^n \tilde{M}^{n+1} - Y^n \langle U^{n^\intercal},\tilde{U}^{{n+1}}\rangle_H}{\triangle t}\\
			&\hspace{3cm}+\big( \mathcal{L}^{*}(u^n_{h,\hat{\rho}}, u^{n+1}_{h,\hat{\rho}}) - f^{{n,n+1}^*},\tilde{U}^{n+1} \big)_{V'V} \\
			&= \Big\langle \frac{(u^{n+1}_{h,\hat{\rho}} - u^n_{h,\hat{\rho}})^*}{\triangle t},\, \tilde{U}^{n+1} \Big\rangle_H+ \Big( \mathcal{L}^{*}(u^n_{h,\hat{\rho}}, u^{n+1}_{h,\hat{\rho}}) - f^{{n,n+1}^*},\, \tilde{U}^{n+1} \Big)_{V'V}
		\end{align*}
		Taking a weak formulation in $\lhrz$ results in 
		\begin{multline}\label{eq:tmpeq7}
			\Big\langle \frac{(u^{n+1}_{h,\hat{\rho}} - u^n_{h,\hat{\rho}})^*}{\triangle t},\, w_h \Big\rangle_{H,\lhr}+ \Big( \mathcal{L}^{*}(u^n_{h,\hat{\rho}}, u^{n+1}_{h,\hat{\rho}}) - f^{{n,n+1}^*},\, w_h \Big)_{V'V,\lhr}=0\\
			\forall w_h = \tilde{U}^{n+1}\delta Y^\intercal,\; \delta Y\in (\lhrz)^R.
		\end{multline}
		Concerning equation \eqref{eq:discDLReq2}, we proceed as follows: $\forall v_h \in (V_h)^R$
		\begin{align}\label{eq:tmpeq8}
			0 &= \Big\langle \frac{\tilde{U}^{n+1} - U^n}{\triangle t}, v_h\Big\rangle_H + \Big(\mathbb{E}\hr[(\mathcal{L}(u^n_{h,\hat{\rho}}, u^{n+1}_{h,\hat{\rho}})-f^{n,n+1}) Y^n], v_h \Big)_{V' V}\nonumber\\
			&= \Big\langle \frac{\tilde{U}^{n+1}\mathbb{E}_{\hat{\rho}}[\tilde{Y}^{{n+1}^\intercal}Y^n] - U^n\mathbb{E}_{\hat{\rho}}[Y^{{n}^\intercal}Y^n]}{\triangle t}, v_h\Big\rangle_H\nonumber\\
			&\hspace{3cm}+ \Big(\mathbb{E}\hr[(\mathcal{L}(u^n_{h,\hat{\rho}}, u^{n+1}_{h,\hat{\rho}})-f^{n,n+1}) Y^n], v_h \Big)_{V' V}\nonumber\\
			&= \Big\langle \frac{(u^{n+1}_{h,\hat{\rho}} - u^n_{h,\hat{\rho}})^*}{\triangle t},\, v_h Y^{n^\intercal} \Big\rangle_{H,\lhr}\nonumber\\
			&\hspace{1cm} + \Big( \mathcal{L}^{*}(u^n_{h,\hat{\rho}}, u^{n+1}_{h,\hat{\rho}}) - f^{{n,n+1}^*},\, v_h Y^{n^\intercal} \Big)_{V'V,\lhr}\quad \forall v_h\in (V_h)^R,
		\end{align}
		where in the second step we applied $\mathbb{E}_{\hat{\rho}}[\tilde{Y}^{{n+1}^\intercal}Y^n] = \text{Id}$ which holds thanks to the discrete DO condition from Lemma \ref{lemma:discdo}. Summing equation \eqref{eq:tmpeq7} and \eqref{eq:tmpeq8} we obtain
		\begin{multline*}\Big\langle \frac{(u^{n+1}_{h,\hat{\rho}} - u^n_{h,\hat{\rho}})^*}{\triangle t},\, w_h \Big\rangle_{H,\lhr}+ \Big( \mathcal{L}^{*}(u^n_{h,\hat{\rho}}, u^{n+1}_{h,\hat{\rho}}) - f^{{n,n+1}^*},\, w_h \Big)_{V'V,\lhr}=0\\
			\forall w_h \in \mathcal{T}_{\tilde{U}^{n+1}Y^{n^\intercal}}.
		\end{multline*}
		The rest of the proof follows the same steps as in the proof of Theorem~\ref{th:discvarform}, i.e.\ summing the mean value equation \eqref{eq:discDLReq1} and noting that some terms vanish.
		\qed
	\end{proof}
	
	\corr{
		\begin{remark}
			Thanks to the observation that $\kernel(\tilde{M}^{n+1}) = \kernel (\tilde{U}^{n+1})$, we can easily see that any discrete solution $\tilde{Y}^{n+1}$ of equation~\eqref{eq:discDLReq3} leads to the same discrete DLR solution $u^{n+1}_{h,\hat{\rho}} = \bar{u}^{n+1}_{h,\hat{\rho}} + \tilde{U}^{n+1}\tilde{Y}^{{n+1}^\intercal}$. Therefore, the result of the preceding theorem as well as the stability properties shown in Section~\ref{sec:stabest} hold for the discrete DLR solution obtained by any of the solutions of equation \eqref{eq:discDLReq3}.
		\end{remark}
	}
	
	\subsection{Reinterpretation as a projector-splitting scheme}\label{sec:projsplit}
	
	The proposed Algorithm~\ref{alg:ourscheme} was derived from the DLR system of equations \eqref{eq:DLReq1}--\eqref{eq:DLReq3}. This subsection is dedicated to showing that this scheme can in fact be formulated as a projector-splitting scheme for the time discretization of the Dual DO approximation of \eqref{eq:pr1}. Afterwards, we will continue by showing its connection to the projector-splitting scheme of the first order proposed in \cite{Lubich14,Lubich15} and further analyzed in \cite{Kieri16}.
	
	In what follows, we will focus on the evolution of $u^{n,*}_{h,\hat{\rho}}$, i.e.\ the $0$-mean part of the discrete DLR solution $u^{n}_{h,\hat{\rho}}$.
	
	\begin{lemma}\label{lemma:projsplit}
		The discretized system of equations \eqref{eq:discDLReq2}--\eqref{eq:discDLReq3} can be equivalently reformulated as 
		\begin{align}
			\langle\tilde{u}_{h,\hat{\rho}},v_h \rangle_{H,\lhr} &= \langle u^{n,*}_{h,\hat{\rho}},v_h\rangle_{H,\lhr} \nonumber\\
			&\hspace{0.5cm}+ \triangle t\big( \Pro_{\hat{\rho},\mathcal{Y}^n}[f^{{n,n+1}^*} - \mathcal{L}^*(u^n_{h,\hat{\rho}}, u^{n+1}_{h,\hat{\rho}})],v_h\big)_{V'V,\lhr}\label{eq:projeq1},\\
			\langle u^{n+1,*}_{h,\hat{\rho}},v_h\rangle_{H,\lhr} &= \langle\tilde{u}_{h,\hat{\rho}},v_h\rangle_{H,\lhr} \nonumber\\
			&\hspace{0.5cm} +\triangle t\big( \Pro_{\hat{\rho},\mathcal{Y}^n}^\perp\big[\Pro_{\tilde{\mathcal{U}}^{n+1}}[f^{{n,n+1}^*} - \mathcal{L}^*(u^n_{h,\hat{\rho}}, u^{n+1}_{h,\hat{\rho}})]\big],v_h\big)_{V'V,\lhr},\label{eq:projeq2}\\
			&\hspace{7cm}\forall v_h\in V_h\otimes\lhr,
			\nonumber
		\end{align}
		where $\tilde{u}_{h,\hat{\rho}} = \tilde{U}^{n+1}Y^{n^\intercal}$.
	\end{lemma}
	\begin{proof}
		These equations are essentially equations \eqref{eq:tmpprojeq1} and \eqref{eq:tmpprojeq2}, which are shown to hold in the proof of Lemma \ref{lemma:DLRprojdisc}.
		\qed
	\end{proof}
	
	We recall that from Lemma~\ref{lemma:DLRproj}, the zero-mean part of the continuous DLR approximation $u^{*} = U Y^\intercal$ satisfies \begin{multline*}
		( \dot{u}^* + \Pi_{u^*}[\mathcal{L}^*(u) - f^*],\, v)_{V'V,\lr} \\ = ( \dot{u}^* + \Pro_{\mathcal{Y}}[\mathcal{L}^*(u) - f^*] + \Pro_{\mathcal{Y}}^{\perp}[\Pro_{\mathcal{U}}[\mathcal{L}^*(u) - f^*]],\, v)_{V'V,\lr} = 0,\\\forall v\in \lrV.
	\end{multline*}
	Lemma~\ref{lemma:projsplit} therefore shows that the time integration scheme corresponds to a projection splitting scheme in which first the projection $\Pro_{\mathcal{Y}}[\mathcal{L}^*(u) - f^*]$ and then the projection $\Pro_{\mathcal{Y}}^{\perp}[\Pro_{\mathcal{U}}[\mathcal{L}^*(u) - f^*]]$ are applied.

	\subsubsection{Comparison to the projection scheme in \cite{Lubich14}}
	There are several equivalent DLR formulations. The DO formulation, proposed and applied in \cite{Sapsis09,Sapsis12,Ueckermann13}, seeks for an approximation of the form $u_R = UY^\intercal$ with $\{U_j\}_{j=1}^R\subset V_h$ orthonormal in $\langle\cdot,\cdot\rangle_{H}$ and $\{Y_j\}_{j=1}^R\subset\lhr$ linearly independent. The dual DO formulation, on the contrary, keeps the stochastic basis $\{Y_j\}_{j=1}^R$ orthonormal in $\langle\cdot,\cdot\rangle_{\lhr}$ and $\{U_j\}_{j=1}^R$ linearly independent. The double dynamically orthogonal (DDO) or bi-orthogonal formulation  searches for an approximation in the form $u_R = USV^\intercal$ with both $\{U_j\}_{j=1}^R\subset V_h$ and $\{V_j\}_{j=1}^R\subset \lhr$ orthonormal in $\langle\cdot,\cdot\rangle_{H}$, $\langle\cdot,\cdot\rangle_{\lhr}$ , respectively, and $S\in\R^{R\times R}$ a full rank matrix (see e.g. \cite{Cheng13,Cheng13b,Koch07}). In \cite{Choi14,Musharbash15} it was shown that these formulations are equivalent. In our work we consider the dual DO formulation with an isolated mean so that the stochastic basis functions are centered.
	
	A first order projector-splitting scheme introduced in \cite{Lubich14,Lubich15} and further analyzed in \cite{Kieri16} is a time integration scheme successfully used for the integration of dynamical low rank approximation in the DDO formulation. This subsection provides a detailed look into the comparison of the Algorithm \ref{alg:ourscheme} and the discretization scheme from \cite{Lubich14,Lubich15}. We will see that, if the solution is full rank, these schemes are in fact equivalent.
	
	We will adapt the algorithm from \cite{Lubich14} to approximate the DLR solution in the DDO form with an isolated mean, i.e.
	\begin{equation}\label{eq:DDOform}
		u_R(t) = \bar{u}_R(t) + U(t)S(t)V(t)^\intercal \quad\in V_h\otimes\lhr.
	\end{equation}
	
	Having an $R$-rank solution $u^n_{h,\hat{\rho}}$, the basic first-order scheme from \cite{Lubich14} requires the knowledge of the solution $u^{n+1}_{h,\hat{\rho}}$, which is used in evaluating the term $\triangle A = u^{n+1}_{h,\hat{\rho}} - u^n_{h,\hat{\rho}}$. To deal with differential equations where $u^{n+1}_{h,\hat{\rho}}$ is a-priori unknown, we will consider a general scheme where \[\triangle A\approx \triangle t  \big(f^{n,n+1} - \mathcal{L}(u^n_{h,\hat{\rho}}, u^{n+1}_{h,\hat{\rho}})\big)\] where $f^{n,n+1}$ and $\mathcal{L}(u^n_{h,\hat{\rho}}, u^{n+1}_{h,\hat{\rho}})$ can be any of the explicit, implicit or semi-implicit discretizations detailed in Section \ref{sec:fullydiscpr}. Adopting the notation from \cite{Lubich14}, the splitting scheme from \cite{Lubich14,Lubich15} for a DDO approximation of \eqref{eq:pr1} results in the following 6-step algorithm.
	
	\begin{algorithm}\label{alg:projsplit}Let $u^{n}_{h,\hat{\rho}} = {\bar{u}}^n + U_0 S_0 V_0^\intercal$ of the form \eqref{eq:DDOform}.
		\begin{enumerate}
			\item Compute the mean value $\hat{\bar{u}}^{n+1}$ such that 
			\begin{multline*}
				\langle \hat{\bar{u}}^{n+1} ,v_h\rangle_{H} = \langle {\bar{u}}^{n} ,v_h\rangle_{H} + \triangle t \Big(\mathbb{E}\hr[f^{n,n+1} - \mathcal{L}(u^n_{h,\hat{\rho}}, u^{n+1}_{h,\hat{\rho}})], v_h \Big)_{V' V}\\ \forall v_h \in V_h.
			\end{multline*}
			\item Solve for $K_1$ such that 
			\begin{multline*}
				\langle K_1,v_h\rangle_H = \langle U_0 S_0,v_h\rangle_H + \\\triangle t \Big( \mathbb{E}_{\hat{\rho}}\big[\big(f^{{n,n+1}^*} - \mathcal{L}^*(u^n_{h,\hat{\rho}}, u^{n+1}_{h,\hat{\rho}})\big) V_0\big],\,v_h \Big)_{V'V} \quad\forall v_h \in V_h.
			\end{multline*}
			\item Compute $U_1\in V_h,\; \hat{S}_1\in\R^{R\times R}$ such that $$U_1 \hat{S}_1 = K_1\;\text{and}\; U_1 \text{ is orthonormal in }\langle \cdot,\cdot\rangle_{H}.$$
			\item Set $$\tilde{S}_0 = \hat{S}_1 - \triangle t \Big(U_1^\intercal,\, \mathbb{E}_{\hat{\rho}}\big[\big(f^{{n,n+1}^*} - \mathcal{L}^*(u^n_{h,\hat{\rho}}, u^{n+1}_{h,\hat{\rho}})\big)V_0\big] \Big)_{V'V}.$$
			\item Compute $L_1\in \lhr$ such that $$ L_1 =  V_0 \tilde{S}_0^\intercal + \triangle t \Big(f^{{n,n+1}^{*}} - \mathcal{L}^{*}(u^n_{h,\hat{\rho}}, u^{n+1}_{h,\hat{\rho}}),U_1 \Big)_{V'V}.$$
			\item Compute $V_1\in \lhrz,\; S_1\in\R^{R\times R}$ such that $$V_1 S_1^\intercal = L_1 \text{ in }\lhrz \text{ and } V_1 \text{ is orthonormal in }\langle \cdot,\cdot\rangle_{\lhr}.$$
		\end{enumerate}
		
	\end{algorithm}
	
	The new solution $\hat{u}^{n+1}_{h,\hat{\rho}}$ is then defined as 
	\begin{equation*}
		\hat{u}^{n+1,*}_{h,\hat{\rho}} = \hat{\bar{u}}^{n+1} + U_1 S_1 V_1^\intercal.
	\end{equation*}

	Now, let us compare the previous steps to Algorithm \ref{alg:ourscheme}. We can easily observe that $\hat{\bar{u}}^{n+1} = {\bar{u}}^{n+1}$. Since $Y^n = V_0$, we can see that equation \eqref{eq:discDLReq2} is equivalent to step~1 with $U^n = U_0 S_0$, i.e.\ $K_1 = \tilde{U}^{n+1}$. 
	Further, we have \[\tilde{M}^{n+1} = \langle \tilde{U}^{{n+1}^\intercal}, \tilde{U}^{n+1}\rangle_H = \hat{S}_1^\intercal \langle U_1^\intercal, U_1\rangle_H \hat{S}_1 = \hat{S}_1^\intercal \hat{S}_1.\] Equation~\eqref{eq:discDLReq3} can be reformulated as 
	\begin{multline*}
		\tilde{Y}^{{n+1}}\hat{S}_1^\intercal \hat{S}_1 =  Y^{n}\hat{S}_1^\intercal \hat{S}_1 + \triangle t \big( f^{{n,n+1}^{*}} - \mathcal{L}^{*}(u^n_{h,\hat{\rho}}, u^{n+1}_{h,\hat{\rho}}), U_1 \big)_{V'V}\hat{S}_1 \\
		- \triangle t Y^n \big(\mathbb{E}_{\hat{\rho}}\big[ Y^{n^\intercal} (f^{{n,n+1}^{*}} - \mathcal{L}^{*}(u^n_{h,\hat{\rho}}, u^{n+1}_{h,\hat{\rho}}))\big], U_1 \big)_{V'V}\hat{S}_1 
	\end{multline*}
	which, provided $\hat{S}_1$ is invertible, is equivalent to 
	\begin{multline*}
		\tilde{Y}^{{n+1}}\hat{S}_1^\intercal =  Y^{n}\Big(\hat{S}_1^\intercal - \triangle t \big(\mathbb{E}_{\hat{\rho}}\big[ Y^{n^\intercal} (f^{{n,n+1}^{*}} - \mathcal{L}^{*}(u^n_{h,\hat{\rho}}, u^{n+1}_{h,\hat{\rho}}))\big], U_1 \big)_{V'V}\Big) \\
		+ \triangle t \big( f^{{n,n+1}^{*}} - \mathcal{L}^{*}(u^n_{h,\hat{\rho}}, u^{n+1}_{h,\hat{\rho}}), U_1 \big)_{V'V}.
	\end{multline*}
	Note that the expression in brackets in the first term on the right hand side is exactly the transpose of $\tilde{S}_0$ from step~3: \[\hat{S}_1^\intercal - \triangle t \big(\mathbb{E}_{\hat{\rho}}\big[ Y^{n^\intercal} (f^{{n,n+1}^{*}} - \mathcal{L}^{*}(u^n_{h,\hat{\rho}}, u^{n+1}_{h,\hat{\rho}}))\big], U_1 \big)_{V'V} = \tilde{S}_0^\intercal,\]
	from which we deduce
	\[L_1 = \tilde{Y}^{{n+1}}\hat{S}_1^\intercal.\] 
	Finally, we have 
	\[\hat{u}_{h,\hat{\rho}}^{n+1,*} = U_1 S_1 V_1^\intercal = U_1 L_1^\intercal = U_1 \hat{S}_1\tilde{Y}^{{n+1}^\intercal} = \tilde{U}^{n+1}\tilde{Y}^{{n+1}^\intercal} = {u}_{h,\hat{\rho}}^{n+1,*}.\]
	
	We conclude that the scheme in Algorithm~\ref{alg:ourscheme} and the scheme in Algorithm~\ref{alg:projsplit} coincide in exact arithmetic,  provided the matrix $S_1$ is invertible. 
	However, the numerical behavior of the two schemes differs when $S_1$ is singular or close to singular. For $\tilde{M}^{n+1}$ close to singular, solving equation \eqref{eq:discDLReq3} might lead to numerical instabilities. This problem seems to be avoided in the projector-splitting scheme from \cite{Lubich14,Lubich15}, as no matrix inversion is involved. Such ill conditioning is however hidden in performing step~3 of Algorithm \ref{alg:projsplit}, since the QR or SVD decomposition can become unstable for ill-conditioned matrices (see \cite[chap.\ 5]{Golub96}). In the case of a rank deficient basis $\{\tilde{U}^{n+1}\}$, Algorithm \ref{alg:ourscheme} updates the stochastic basis by solving equation \eqref{eq:discDLReq3} in a least square sense while minimizing the norm $\|\tilde{Y}^{n+1} - Y^n\|_{\lhr}$. The previous subsection showed that such solution satisfies the discrete variational formulation which plays a crucial role in stability estimation (see Section \ref{sec:stabestdiscDLR}).  On the other hand, Algorithm \ref{alg:projsplit} relies on the somehow arbitrary completion of the basis $\{U_1\}$ in the step~3. In presence of rank deficiency, the two algorithms can deliver different solutions (see Section~\ref{sec:rankdef} for a numerical comparison).
	
	\begin{remark}
		Note that the ordering of the equations in Algorithm \ref{alg:ourscheme} is crucial. When dealing with the DO formulation, i.e.\ orthonormal deterministic basis and linearly independent stochastic basis, we shall first update the stochastic basis and then evolve the deterministic basis. For a reversed ordering the Theorem \ref{th:discvarform} would not hold.
	\end{remark}
	
	\section{Stability estimates}\label{sec:stabest}
	The stability of the solution of problems similar to \eqref{eq:pr1} are well analyzed (see e.g.\ \cite{Ern04/6}). A natural question is to what extent constraining the dynamics to the low rank manifold influences the stability properties. In Section \ref{sec:stabesttrue}, we will first recall some stability properties of the true solution $u_{\tr}$ of problem~\eqref{eq:pr1}. Then, in Section~\ref{sec:stabestDLR} we will see that these properties hold for the continuous DLR solution as well. It turns out that our discretization schemes satisfy analogous stability properties, as we will see in Section~\ref{sec:stabestdiscDLR}. In particular, we will show that the implicit and semi-implicit version are unconditionally stable under some mild conditions on the size of the randomness in the operator. We will state two types of estimates: the first one holds for an operator $\mathcal{L}$ as described in Section \ref{sec:prstatement} and a second one additionally assuming the operator $\mathcal{L}$ to be symmetric. Note that in the second case the bilinear coercive form $\langle\cdot,\cdot\rangle_{\mathcal{L},\rho}$ is a scalar product on $\lrV$. In the rest of this section we will assume that a solution of problem \eqref{eq:pr1}, a continuous DLR solution and a discrete DLR solution exist.
	\subsection{Stability of the continuous problem}\label{sec:stabesttrue}
	We state here some standard stability estimates concerning the solution $u_{\mathrm{true}}$ of problem~\eqref{eq:pr1}.
	\begin{proposition}\label{prop:stab_true}
		Let $u_{\mathrm{true}}\in L^2(0,T;\lrV)$ be the solution of problem \eqref{eq:pr1}. Then, the following estimates hold:
		\begin{enumerate}
			\item \begin{multline}\label{eq:stabtrue1}
				\|u_{\mathrm{true}}(T)\|_{H,\lr}^2 + \Ccoerc\int_{0}^{T}\|u_{\mathrm{true}}(t)\|_{V,\lr}^2 \,\mathrm{d}t\\
				\leq \|u_{\mathrm{true}}(0)\|_{H,\lr}^2 + \frac{\Cemb^2}{\Ccoerc}\big\|f\big\|_{L^2(0,T;\lrH)}^2;
			\end{multline}
			\item if, in addition, $\mathcal{L}$ is symmetric and $\dot{u}_{\mathrm{true}}\in L^2(0,T;\lrH)$, we have
			\begin{multline}\label{eq:stabtrue2}
				\|u_{\mathrm{true}}(T)\|_{\mathcal{L},\rho}^2 + \int_{0}^{T}\|\dot{u}_{\mathrm{true}}(t)\|_{H,\lr}^2 \,\mathrm{d}t\\
				\leq \|u_{\mathrm{true}}(0)\|_{\mathcal{L},\rho}^2 + \big\|f\big\|_{L^2(0,T;\lrH)}^2,
			\end{multline}
		\end{enumerate}
		where $\Ccoerc>0$ is the coercivity constant defined in \eqref{eq:coerc_const} and $\Cemb$ is the continuous embedding constant defined in \eqref{eq:embed_const}.
		
		For $f = 0$ and $t_1,t_2\in [0,T],\;t_1\leq t_2$, we have:
		\begin{enumerate}
			\item[3.] \begin{equation}\label{eq:stabtrue3}
				\|u_{\mathrm{true}}(t_2)\|_{H,\lr}\leq \|u_{\mathrm{true}}(t_1)\|_{H,\lr},
			\end{equation}
			\item[4.] moreover, if $\mathcal{L}$ is symmetric and $\dot{u}_{\mathrm{true}}\in L^2(0,T;\lrH)$, we have
			\begin{equation}\label{eq:stabtrue4}
				\|u_{\mathrm{true}}(t_2)\|_{\mathcal{L},\rho}\leq \|u_{\mathrm{true}}(t_1)\|_{\mathcal{L},\rho}.
			\end{equation}
		\end{enumerate}
	\end{proposition}
	
	\begin{proof}
		As for part $1$, choose $u_{\mathrm{true}}$ as a test function in the variational formulation~\eqref{eq:pr1}. Using \cite[Prop. 23.23]{Zeidler90} results in 
		\begin{multline*}
			\frac{1}{2}\frac{\dx }{\dx t}\|u_{\mathrm{true}}\|_{H,\lr}^2 + \langle u_{\mathrm{true}}, u\tr\rangle_{\mathcal{L},\rho} = \langle f,u_{\mathrm{true}}\rangle_{H,\lr}\leq \Cemb\|f\|_{H,\lr}\|u_{\mathrm{true}}\|_{V,\lr}\\
			\leq \frac{\Cemb^2}{2\Ccoerc} \|f\|_{H,\lr}^2 + \frac{\Ccoerc}{2}\|u_{\mathrm{true}}\|_{V,\lr}^2\quad\text{for a.e.\ }t\in(0,T].
		\end{multline*}
		Multiplying by $2$ and integrating over $[0,T]$ gives the sought estimate. Part~2 is proved in a similar way by considering $\dot{u}_{\mathrm{true}}$ as a test function. We can derive
		\begin{multline*}
			\|\dot{u}_{\mathrm{true}}\|_{H,\lr}^2 + \frac{1}{2}\frac{\dx }{\dx t}\|u_{\mathrm{true}}\|_{\mathcal{L},\rho}^2= \langle f,\dot{u}_{\mathrm{true}}\rangle_{H,\lr}\leq \|f\|_{H,\lr}\|\dot{u}_{\mathrm{true}}\|_{H,\lr}\\
			\leq \frac{\|f\|_{H,\lr}^2}{2} + \frac{\|\dot{u}_{\mathrm{true}}\|_{H,\lr}^2}{2}
		\end{multline*}
		and obtain the result by multiplying by 2 and integrating over $[0,T]$.
		
		Part~3 and part~4 are consequences of part~1 and 2, where the final integration is realized over $[t_1,t_2]$ instead of $[0,T]$.
		\qed
	\end{proof}
	
	\subsection{Stability of the continuous DLR solution}\label{sec:stabestDLR}
	
	Constraining the dynamics to the $R$-rank manifold does not destroy the stability properties from Proposition~\ref{prop:stab_true}.
	
	\begin{theorem}
		Let $u\in L^2(0,T;\lrV)$ with $\dot{u}\in L^2(0,T;\lrV)$ be the continuous DLR solution defined in Definition~\ref{def:contDLRsol}. Then $u$ satisfies the same inequalities \eqref{eq:stabtrue1}, \eqref{eq:stabtrue2}, \eqref{eq:stabtrue3}, \eqref{eq:stabtrue4} as the true solution $u\tr$.

	\end{theorem}
	\begin{proof}
		Part 1: Let $u = \bar{u} + U Y^\intercal$ with $U Y^\intercal\in\mathcal{M}_R$. Then, we have $u^* = u - \bar{u}\;\in \mathcal{T}_{u^*}\mathcal{M}_R$. Indeed, since 
		$$u^* = \sum_{i=1}^R U_i 0 + U_i Y_i\quad\in \lrVz$$ with $\langle 0, Y_i\rangle_{\lr} = 0$, we can take $u$ as a test function in the variational formulation \eqref{eq:DLRvar}. The rest of the proof follows the same steps as in the proof of Proposition~\ref{prop:stab_true}.
		
		Part 2: we express $$\dot{u}^* = \sum_{j=1}^R \dot{U}_j Y_j + U_j \dot{Y}_j\quad\in \mathcal{T}_{u^*}\mathcal{M}_R$$ since $\langle Y_i, \dot{Y}_j\rangle_{\lr} = \delta_{ij}$ and $\dot{{u}}^*\in \lrV$. As $\dot{\bar{u}}\in V$ we can consider $\dot{u}$ as a test function in the variational formulation \eqref{eq:DLRvar} and arrive at the sought result.
		
		Part 3 and 4 \corr{are} obtained analogously.
		\qed
	\end{proof}

	\subsection{Stabilty of the discrete DLR solution}\label{sec:stabestdiscDLR}
	
	Now we proceed with showing stability properties of the fully discretized DLR system from Algorithm \ref{alg:ourscheme} for the three different operator evaluation terms corresponding to implicit Euler, explicit Euler and semi-implicit scheme. For each of them we will establish boundedness of norms and a decrease of norms for the case of zero forcing term $f$.
	
	The following simple lemma will be repeatedly used throughout.
	\begin{lemma}\label{lemma:scprod}
		Let $\langle \cdot, \cdot\rangle : (V_h\otimes \lhr) \times (V_h\otimes \lhr)\to\mathbb{R}$ be a symmetric bilinear form. Then it holds
		\begin{align*}
			\langle v,w - v\rangle &= \frac{1}{2}\Big( \langle w, w\rangle - \langle v,v\rangle - \langle w-v,w-v\rangle\Big)\\
			\langle w,w - v\rangle &= \frac{1}{2}\Big( \langle w, w\rangle - \langle v,v\rangle + \langle w-v,w-v\rangle\Big)\\
			\langle v,w + v\rangle &= \frac{1}{2}\Big( \langle v, v\rangle - \langle w,w\rangle + \langle w+v,w+v\rangle\Big)
		\end{align*}
		for any $v,w\in V_h\otimes\lhr$.
	\end{lemma}

	\subsubsection{Implicit Euler scheme}
	Applying an implicit operator evaluation, i.e.\ $\mathcal{L}(u^n_{h,\hat{\rho}}, u^{n+1}_{h,\hat{\rho}}) = \mathcal{L}(u^{n+1}_{h,\hat{\rho}})$ results in a discretization scheme with the following stability properties.
	
	\begin{theorem}
		Let $\{u_{h,\hat{\rho}}^n\}_{n=0}^{N}$ be the discrete DLR solution as defined in Algorithm \ref{alg:ourscheme} with $\mathcal{L}(u^n_{h,\hat{\rho}}, u^{n+1}_{h,\hat{\rho}}) = \mathcal{L}(u^{n+1}_{h,\hat{\rho}})$. Then the following estimates hold:
		\begin{enumerate}
			\item
			\begin{multline*}
				\|u^{N}_{h,\hat{\rho}}\|_{H,\lhr}^2  + \triangle t \Ccoerc \sum_{n=0}^{N-1}\|u^{n+1}_{h,\hat{\rho}}\|_{V,\lhr}^2 \\
				\leq  \|u^0_{h,\hat{\rho}}\|_{H,\lhr}^2 + \triangle t\frac{\Cemb^2}{\Ccoerc}\sum_{n=0}^{N-1}\|f(t_{n+1})\|_{H,\lhr}^2,
			\end{multline*}
			\item if $\mathcal{L}$ is a symmetric operator we have 
			\begin{multline*}
				\|u^{N}_{h,\hat{\rho}}\|_{\mathcal{L},\hat{\rho}}^2 + \triangle t\sum_{n=0}^{N-1}\Big\| \frac{u^{n+1}_{h,\hat{\rho}} - u^{n}_{h,\hat{\rho}}}{\triangle t}\Big\|_{H,\lhr}^2\\
				\leq \|u^{0}_{h,\hat{\rho}}\|_{\mathcal{L},\hat{\rho}}^2 + \triangle t \sum_{n=0}^{N-1}\|f(t_{n+1})\|_{H,\lhr}^2,
			\end{multline*}
		\end{enumerate}
		for any time and space discretization parameters $\triangle t,\,h>0$ with $\Ccoerc, \Cemb>0$ the coercivity and continuous embedding constant defined in \eqref{eq:coerc_const}, \eqref{eq:embed_const}, respectively.
		
		In particular, for $f = 0$ and  $n=0,\dots,N-1$ it holds:
		\begin{enumerate}
			\item[3.] $\|u^{n+1}_{h,\hat{\rho}}\|_{H,\lhr}\leq \|u^{n}_{h,\hat{\rho}}\|_{H,\lhr},$
			\item[4.] if $\mathcal{L}$ is a symmetric operator we have $\|u^{n+1}_{h,\hat{\rho}}\|_{\mathcal{L},\hat{\rho}}\leq \|u^{n}_{h,\hat{\rho}}\|_{\mathcal{L},\hat{\rho}}$.
		\end{enumerate}
	\end{theorem}
	\begin{proof}\mbox{\\}
		Thanks to Theorem \ref{th:discvarform}, we know that the discretized DLR system of equations with implicit operator evaluation can be written in a variational formulation as
		\begin{align}\label{eq:implvarform}
			\begin{split}
				&\hspace{0cm}\big\langle \frac{u^{n+1}_{h,\hat{\rho}} - u^n_{h,\hat{\rho}}}{\triangle t},\, v_h\big\rangle_{H,\lhr} + \big\langle u^{n+1}_{h,\hat{\rho}},\, v_h  \big\rangle_{\mathcal{L},\hat{\rho}} = \big\langle f(t_{n+1}),\,v_h\big\rangle_{H,\lhr},\\
				&\hspace{3.1cm} \forall v_h = \bar{v}_h + v_h^* \text{ with } \bar{v}_h\in V_h\text{ and }v_h^*\in \mathcal{T}_{\tilde{U}^{n+1}Y^n} \mathcal{M}_R^{h,\hat{\rho}},
			\end{split}
		\end{align}
		$n = 0,\dots,N-1$.
		\begin{enumerate}
			\item Based on Lemma \ref{lemma:unun+1} we take $v_h = u^{n+1}_{h,\hat{\rho}}$ as a test function in the variational formulation \eqref{eq:implvarform}. Using Lemma \ref{lemma:scprod}  results in 
			\begin{multline*}\|u^{n+1}_{h,\hat{\rho}}\|_{H,\lhr}^2 - \|u^n_{h,\hat{\rho}}\|_{H,\lhr}^2 + \|u^{n+1}_{h,\hat{\rho}} - u^n_{h,\hat{\rho}}\|_{H,\lhr}^2 + 2\triangle t \langle u^{n+1}_{h,\hat{\rho}},u^{n+1}_{h,\hat{\rho}}\rangle_{\mathcal{L},\hat{\rho}} \\
				= 2\triangle t(f(t_{n+1}), u^{n+1}_{h,\hat{\rho}})_{V'V,\lhr}\leq \triangle t\frac{\Cemb^2}{\Ccoerc}\|f(t_{n+1})\|_{H,\lhr}^2 +\triangle t\Ccoerc \| u^{n+1}_{h,\hat{\rho}}\|^2_{V,\lhr}.
			\end{multline*}
			Using the coercivity condition \eqref{eq:coerc_const} and summing over $n=0,\dots,N-1$ gives us the sought result.
			\item Now, consider $v_h = (u^{n+1}_{h,\hat{\rho}} - u^{n+1}_{h,\hat{\rho}})/\triangle t$. Using Lemma \ref{lemma:scprod}, the variational formulation results in
			\begin{multline*}
				\Big\|\frac{u^{n+1}_{h,\hat{\rho}} - u^n_{h,\hat{\rho}}}{\triangle t}\Big\|_{H,\lhr}^2 + \frac{1}{2\triangle t}\big(\|u^{n+1}_{h,\hat{\rho}}\|_{\mathcal{L},\hat{\rho}}^2 - \|u^{n}_{h,\hat{\rho}}\|_{\mathcal{L},\hat{\rho}}^2 + \|u^{n+1}_{h,\hat{\rho}} - u^{n}_{h,\hat{\rho}}\|_{\mathcal{L},\hat{\rho}}^2\big) \\
				= \Big\langle f(t_{n+1}), \frac{u^{n+1}_{h,\hat{\rho}} - u^{n}_{h,\hat{\rho}}}{\triangle t}\Big\rangle_{H,\lhr}\leq \frac{\|f(t_{n+1})\|_{H,\lhr}^2}{2} + \frac{1}{2}\Big\| \frac{u^{n+1}_{h,\hat{\rho}} - u^{n}_{h,\hat{\rho}}}{\triangle t}\Big\|_{H,\lhr}^2.
			\end{multline*}
			Multiplying by $2\triangle t$ and summing over $n=0,\dots,N-1$ leads us to the result.
		\end{enumerate}
		Parts 3 and 4 follow from parts 1 and 2 without summing over $n=0,\dots,N-1$.
		\qed
	\end{proof}
	
	\subsubsection{Explicit Euler scheme}
	Concerning the explicit Euler scheme (see subsection~\ref{sec:fullydiscpr}), which applies the time discretization $\mathcal{L}(u^n_{h,\hat{\rho}}, u^{n+1}_{h,\hat{\rho}}) = \mathcal{L}(u^n_{h,\hat{\rho}})$, the following stability result holds.
	
	\begin{theorem}\label{th:explstab}
		Let $\{u_{h,\hat{\rho}}^n\}_{n=0}^{N}$ be the discrete DLR solution as defined in Algorithm~\ref{alg:ourscheme} with $\mathcal{L}(u^n_{h,\hat{\rho}}, u^{n+1}_{h,\hat{\rho}}) = \mathcal{L}(u^n_{h,\hat{\rho}})$. Then the following estimates hold:
		\begin{enumerate}
			\item \begin{multline*}
				\|u^{N}_{h,\hat{\rho}}\|_{H,\lhr}^2 + \triangle t \Ccoerc (1-\kappa)\sum_{n=0}^{N-1}\|u^{n+1}_{h,\hat{\rho}}\|_{V,\lhr}^2 \leq\|u^0_{h,\hat{\rho}}\|_{H,\lhr}^2 + \\\frac{\triangle t \Cemb^2}{\Ccoerc}\sum_{n=0}^{N-1}\|f(t_n)\|_{H,\lhr}^2
			\end{multline*} for $0< \kappa $ and $\triangle t,h$ satisfying 
			\begin{equation}\label{eq:explstabcond}
				\frac{\triangle t}{h^{2p}} \leq \frac{\kappa\,\Ccoerc}{\Cinv^2\,\Cbound^2}.
			\end{equation}
			%		\item 
			%		$$\|u^{n+1}_{h,\hat{\rho}}\|_{H,\lhr}^2\leq\|u^n_{h,\hat{\rho}}\|_{H,\lhr}^2 + \frac{\triangle t}{2}\Big( \frac{\Cemb^2}{\Ccoerc} + \triangle t\Big)\|f^{n,n+1}\|_{H,\lhr}^2,$$
			\item If $\mathcal{L}$ is a symmetric operator we have $$\|u^{N}_{h,\hat{\rho}}\|_{\mathcal{L},\hat{\rho}}\leq\|u^0_{h,\hat{\rho}}\|_{\mathcal{L},\hat{\rho}} + \frac{\triangle t}{\kappa} \sum_{n=0}^{N-1}\|f(t_n)\|_{H,\lhr}^2,$$
		\end{enumerate}
		for $\triangle t,h$ satisfying 
		\begin{equation}\label{eq:explsymstabcond}
			\frac{\triangle t}{h^{2p}} \leq \frac{2-\kappa}{\Cinv^2 \Cbound}\quad \text{ with } 0<\kappa<2.
		\end{equation}
		Here $\Ccoerc, \Cbound, \Cemb >0$ are the coercivity, continuity and continuous embedding constants defined in \eqref{eq:coerc_const}, \eqref{eq:bound_const}, \eqref{eq:embed_const}, respectively and $\Cinv$ is the inverse inequality constant introduced in \eqref{eq:discDLRinv}.
		
		For $f = 0$ and $n = 0,\dots,N-1$ it holds:
		\begin{enumerate}
			\item[3.] $\|u^{n+1}_{h,\hat{\rho}}\|_{H,\lhr}\leq \|u^{n}_{h,\hat{\rho}}\|_{H,\lhr},$ 
			\begin{equation*}
				\text{under \corr{the} weakened condition }\qquad\frac{\triangle t}{h^{2p}} \leq \frac{2\Ccoerc}{\Cinv^2\Cbound^2 }.\hspace{4cm}
			\end{equation*}
			\item[4.] If $\mathcal{L}$ is a symmetric operator we have 
			\begin{gather*}
				\|u^{n+1}_{h,\hat{\rho}}\|_{\mathcal{L},\hat{\rho}}\leq \|u^{n}_{h,\hat{\rho}}\|_{\mathcal{L},\hat{\rho}},\\
				\|u^{n+1}_{h,\hat{\rho}}\|_{H,\lhr}\leq \|u^{n}_{h,\hat{\rho}}\|_{H,\lhr},
			\end{gather*}
			under \corr{the} weakened condition 
			\begin{equation*}
				\frac{\triangle t}{h^{2p}} \leq \frac{2}{\Cinv^2\Cbound }.
			\end{equation*}
		\end{enumerate}	
	\end{theorem}
	\begin{proof}\mbox{\\}
		Thanks to the Theorem \ref{th:discvarform} we can rewrite the system of equations in the variational formulation
		\begin{align}\label{eq:explvarform}
			\begin{split}
				&\hspace{0cm}\big\langle \frac{u^{n+1}_{h,\hat{\rho}} - u^n_{h,\hat{\rho}}}{\triangle t},\, v_h\big\rangle_{H,\lhr} + \big\langle u^{n}_{h,\hat{\rho}},\, v_h  \big\rangle_{\mathcal{L},\hat{\rho}} = \big\langle f(t_n),\,v_h\big\rangle_{H,\lhr},\\
				&\hspace{2.2cm} \forall v_h = \bar{v}_h + v_h^* \text{ with } \bar{v}_h\in V_h\text{ and }v_h^*\in \mathcal{T}_{\tilde{U}^{n+1}Y^n} \mathcal{M}_R^{h,\hat{\rho}}.
			\end{split}
		\end{align}
		\begin{enumerate}
			\item Based on Lemma \ref{lemma:unun+1} we take $v_h = u^{n+1}_{h,\hat{\rho}}$ as a test function in the variational formulation \eqref{eq:explvarform} and using Lemma \ref{lemma:scprod}  results in 
			\begin{multline*}\|u^{n+1}_{h,\hat{\rho}}\|_{H,\lhr}^2 - \|u^n_{h,\hat{\rho}}\|_{H,\lhr}^2 + \|u^{n+1}_{h,\hat{\rho}} - u^n_{h,\hat{\rho}}\|_{H,\lhr}^2 + 2\triangle t \langle u^{n}_{h,\hat{\rho}},u^{n+1}_{h,\hat{\rho}}\rangle_{\mathcal{L},\hat{\rho}} \\
				= 2\triangle t(f(t_{n}), u^{n+1}_{h,\hat{\rho}})_{V'V,\lhr}\leq \triangle t\frac{\Cemb^2}{\Ccoerc}\|f(t_n)\|_{H,\lhr}^2 +\triangle t\Ccoerc \| u^{n+1}_{h,\hat{\rho}}\|_{V,\lhr}. 
			\end{multline*}
			We further proceed by estimating
			\begin{align}
				&2\triangle t\langle u^{n}_{h,\hat{\rho}},u^{n+1}_{h,\hat{\rho}}\rangle_{\mathcal{L},\hat{\rho}} = 2\triangle t\langle u^{n}_{h,\hat{\rho}} - u^{n+1}_{h,\hat{\rho}},u^{n+1}_{h,\hat{\rho}}\rangle_{\mathcal{L},\hat{\rho}} + 2\triangle t\langle u^{n+1}_{h,\hat{\rho}},u^{n+1}_{h,\hat{\rho}}\rangle_{\mathcal{L},\hat{\rho}}\nonumber\\
				&\geq -2\triangle t \Cbound \|u^{n+1}_{h,\hat{\rho}} - u^{n}_{h,\hat{\rho}}\|_{V,\lhr}\|u^{n+1}_{h,\hat{\rho}}\|_{V,\lhr} + 2\triangle t\Ccoerc \|u^{n+1}_{h,\hat{\rho}}\|_{V,\lhr}^2\nonumber\\
				&\geq -\kappa \triangle t\Ccoerc\|u^{n+1}_{h,\hat{\rho}}\|_{V,\lhr}^2 +2\triangle t\Ccoerc\|u^{n+1}_{h,\hat{\rho}}\|_{V,\lhr}^2  - \triangle t \frac{\Cinv^2 \Cbound^2}{\kappa \,h^{2p} \Ccoerc}\|u^{n+1}_{h,\hat{\rho}} - u^{n}_{h,\hat{\rho}}\|_{H,\lhr}^2\label{eq:tmpeq3}
			\end{align}
			where, in the third step, we used the inequality $$\|u^{n+1}_{h,\hat{\rho}}-u^n_{h,\hat{\rho}}\|_{H,\lhr} \geq \frac{h^p}{\Cinv}\|u^{n+1}_{h,\hat{\rho}}-u^n_{h,\hat{\rho}}\|_{V,\lhr},$$ which holds based on assumption \eqref{eq:discDLRinv}. Combining the terms, using the condition \eqref{eq:explstabcond} and summing over $n = 0,\dots,N-1$ finishes the proof.
			\item Lemma~\ref{lemma:unun+1} enables us to take $u^{n+1}_{h,\hat{\rho}} - u^n_{h,\hat{\rho}}$ as a test function in \eqref{eq:explvarform}. This results in \begin{multline}\label{eq:proof10}
				\frac{1}{\triangle t}\|u^{n+1}_{h,\hat{\rho}} - u^n_{h,\hat{\rho}}\|_{H,\lhr}^2 + \langle u^n_{h,\hat{\rho}}, u^{n+1}_{h,\hat{\rho}} - u^n_{h,\hat{\rho}}\rangle_{\mathcal{L},\hat{\rho}} = \langle f(t_n), u^{n+1}_{h,\hat{\rho}} - u^n_{h,\hat{\rho}}\rangle_{H,\lhr}\\
				\leq \frac{\triangle t\|f(t_n)\|_{H,\lhr}^2}{2\kappa} + \frac{\kappa\,\|u^{n+1}_{h,\hat{\rho}} - u^n_{h,\hat{\rho}}\|_{H,\lhr}^2}{2\triangle t}.
			\end{multline} Using Lemma \ref{lemma:scprod} we obtain 
			\begin{align*}
				&\|u^{n+1}_{h,\hat{\rho}}\|_{\mathcal{L},\hat{\rho}}^2 \leq \|u^n_{h,\hat{\rho}}\|_{\mathcal{L},\hat{\rho}}^2 +\frac{\triangle t}{\kappa}\|f(t_n)\|_{H,\lhr}^2+ \|u^{n+1}_{h,\hat{\rho}} - u^n_{h,\hat{\rho}}\|_{\mathcal{L},\hat{\rho}}^2 \\
				&\hspace{2cm}- \frac{2-\kappa}{\triangle t}\|u^{n+1}_{h,\hat{\rho}}-u^n_{h,\hat{\rho}}\|_{H,\lhr}^2\\
				&\hspace{0.4cm}\leq \|u^n_{h,\hat{\rho}}\|_{\mathcal{L},\hat{\rho}}^2 +\frac{\triangle t}{\kappa}\|f(t_n)\|_{H,\lhr}^2+ \Big( 1-\frac{(2-\kappa)\,h^{2p}}{\Cinv^2\Cbound \triangle t} \Big)\|u^{n+1}_{h,\hat{\rho}} - u^n_{h,\hat{\rho}}\|_{\mathcal{L},\hat{\rho}}^2\\
				&\hspace{0.4cm}\leq \|u^n_{h,\hat{\rho}}\|_{\mathcal{L},\hat{\rho}}^2 +\frac{\triangle t}{\kappa}\|f(t_n)\|_{H,\lhr}^2
			\end{align*}
			where, in the second step, we used the assumption \eqref{eq:discDLRinv}, \eqref{eq:bound_const}
			and the fact that  $1-\frac{(2-\kappa) h^{2p}}{\Cinv^2\Cbound \triangle t} \leq 0$, thanks to the stability condition \eqref{eq:explsymstabcond}.
			\item The proof of part~3 follows the same steps as the proof of part~1. We have \begin{multline*}\|u^{n+1}_{h,\hat{\rho}}\|_{H,\lhr}^2 - \|u^n_{h,\hat{\rho}}\|_{H,\lhr}^2 + \|u^{n+1}_{h,\hat{\rho}} - u^n_{h,\hat{\rho}}\|_{H,\lhr}^2 + 2\triangle t \langle u^{n}_{h,\hat{\rho}},u^{n+1}_{h,\hat{\rho}}\rangle_{\mathcal{L},\hat{\rho}} = 0.
			\end{multline*} In \eqref{eq:tmpeq3} we choose $\kappa = 2$ and conclude the result.
			\item[4.] The proof of the forth property follows the same steps as the proof of part~2. Since there is no need to use the Young's inequality in \eqref{eq:proof10}, the condition on $\triangle t/h^{2p}$ is weakened:
			\begin{align*}
				\|u^{n+1}_{h,\hat{\rho}}\|_{\mathcal{L},\hat{\rho}}^2 &= \|u^n_{h,\hat{\rho}}\|_{\mathcal{L},\hat{\rho}}^2 + \|u^{n+1}_{h,\hat{\rho}} - u^n_{h,\hat{\rho}}\|_{\mathcal{L},\hat{\rho}}^2 - \frac{2}{\triangle t}\|u^{n+1}_{h,\hat{\rho}}-u^n_{h,\hat{\rho}}\|_{H,\lhr}^2\\
				&\leq \|u^n_{h,\hat{\rho}}\|_{\mathcal{L},\hat{\rho}}^2 + \Big( 1-\frac{2h^{2p}}{\Cinv^2 \Cbound \triangle t} \Big)\|u^{n+1}_{h,\hat{\rho}} - u^n_{h,\hat{\rho}}\|_{\mathcal{L},\hat{\rho}}^2.
			\end{align*}
			As for the estimate in the  $\|\cdot\|_{H,\lhr}$-norm we can derive
			\begin{align*}
				\|u^{n+1}_{h,\hat{\rho}}\|_{H,\lhr}^2 &= \|u^n_{h,\hat{\rho}}\|_{H,\lhr}^2 -\frac{\triangle t}{2} \Big( \|u^n_{h,\hat{\rho}}\|^2_{\mathcal{L},\hat{\rho}} - \|u^{n+1}_{h,\hat{\rho}}\|_{\mathcal{L},\hat{\rho}}^2\\
				&\hspace{6cm} + \|u^{n+1}_{h,\hat{\rho}} + u^n_{h,\hat{\rho}}\|_{\mathcal{L},\hat{\rho}}^2 \Big)\\
				&\leq \|u^n_{h,\hat{\rho}}\|_{H,\lhr}^2,
			\end{align*}
			where in the last inequality we applied $\|u^{n+1}_{h,\hat{\rho}}\|_{\mathcal{L},\hat{\rho}}\leq\|u^n_{h,\hat{\rho}}\|_{\mathcal{L},\hat{\rho}} $ for $\frac{\triangle t}{h^{2p}} \leq \frac{2}{\Cinv^2\Cbound}$.
		\end{enumerate}
		\qed
	\end{proof}
	
	\subsubsection{Semi-implicit scheme}
	
	This subsection is dedicated to analyzing the semi-implicit scheme introduced in subsection~\ref{sec:fullydiscpr} which applies the discretization $\mathcal{L}(u^n_{h,\hat{\rho}}, u^{n+1}_{h,\hat{\rho}}) = \Ldet(u^{n+1}_{h,\hat{\rho}}) + \Lstoch(u^n_{h,\hat{\rho}})$.
	
	Apart from the inverse inequality \eqref{eq:discDLRinv} we will be using two additional inequalities. Let us assume there exists a constant $\Cdet>0$ such that 
	\begin{equation}\label{eq:equivineq}
		\langle u, u\rangle_{\Ldet,\hat{\rho}} \geq \Cdet\,\langle u, u\rangle_{\mathcal{L},\hat{\rho}}, \hspace{1cm}\forall u\in V_h\otimes\lhr.
	\end{equation}
	This constant plays an important role in the stability estimation as it quantifies the extent to which the operator is evaluated implicitly. Its significance is summarized in Theorem \ref{th:semiimplstab}. In addition we introduce a constant $\Cstoch$ that bounds the stochasticity of the operator
	\begin{equation}\label{eq:Cstoch}
		|(\Lstoch(u), v)_{V'V,\lhr}|\leq \Cstoch \|u\|_{V,\lhr} \|v\|_{V,\lhr}.
	\end{equation}
	
	\begin{theorem}\label{th:semiimplstab}
		Let $\{u^n_{h,\hat{\rho}}\}_{n=0}^N$ be the discrete DLR solution as defined in Algorithm~\ref{alg:ourscheme} with $\mathcal{L}(u^n_{h,\hat{\rho}}, u^{n+1}_{h,\hat{\rho}}) = \Ldet(u^{n+1}_{h,\hat{\rho}}) + \Lstoch(u^n_{h,\hat{\rho}})$ with $\Ldet$ and $\Lstoch$ satisfying \eqref{eq:equivineq} and \eqref{eq:Cstoch}, respectively. Then it holds
		\begin{enumerate}
			\item \begin{multline*}
				\|u^{N}_{h,\hat{\rho}}\|_{H,\lhr}^2 + \triangle t \Ccoerc (1-\kappa)\sum_{n=0}^{N-1}\|u^{n+1}_{h,\hat{\rho}}\|_{V,\lhr}^2 \leq\|u^0_{h,\hat{\rho}}\|_{H,\lhr}^2 + \\\frac{\triangle t \Cemb^2}{\Ccoerc}\sum_{n=0}^{N-1}\|f^{n,n+1}\|_{H,\lhr}^2
			\end{multline*} for $\kappa > 0$ and $\triangle t,h$ satisfying 
			\begin{equation}\label{eq:semiimplstabcond}
				\frac{\triangle t}{h^{2p}} \leq \frac{\kappa\,\Ccoerc}{\Cinv^2\,\Cstoch^2}.
			\end{equation}
			\item If $\mathcal{L}$ is a symmetric operator we have
		\end{enumerate}
		\begin{equation}\label{eq:est_semiimpl}
			\|u^{N}_{h,\hat{\rho}}\|_{\mathcal{L},\hat{\rho}}\leq \|u^0_{h,\hat{\rho}}\|_{\mathcal{L},\hat{\rho}} + \frac{\triangle t}{\kappa}\sum_{n=0}^{N-1}\|f^{n,n+1}\|_{H,\lhr}^2
		\end{equation}
		for $\triangle t,h$ satisfying 
		\begin{equation*}
			\frac{\triangle t}{h^{2p}} \leq \left\{
			\begin{array}{lr}
				+\infty & \text{if }\; \Cdet \geq \frac{1}{2}\\
				\frac{2-\kappa}{\Cinv^2 \Cbound (1-2\Cdet)} & \text{if } \;\Cdet < \frac{1}{2}
			\end{array}
			\right.
		\end{equation*}	
		Here $\Ccoerc, \Cbound, \Cemb, \Cinv>0$ are the coercivity, continuity, continuous embedding and inverse inequality constants defined in \eqref{eq:coerc_const}, \eqref{eq:bound_const}, \eqref{eq:embed_const}, \eqref{eq:discDLRinv}, respectively. The constants $\Cdet, \Cstoch$ were introduced in \eqref{eq:equivineq}, \eqref{eq:Cstoch}.
		
		For $f = 0$ and $\mathcal{L}$ symmetric we have
		\begin{enumerate}
			\item[3.]\begin{equation}\label{eq:est_semiimpl0f}
				\|u^{n+1}_{h,\hat{\rho}}\|_{\mathcal{L},\hat{\rho}}\leq \|u^{n}_{h,\hat{\rho}}\|_{\mathcal{L},\hat{\rho}},\qquad n=0,\dots,N-1
			\end{equation}
		\end{enumerate}
		with $\triangle t,h$ satisfying a weakened condition
		\begin{equation}\label{eq:semiimplcond}
			\frac{\triangle t}{h^{2p}} \leq \left\{
			\begin{array}{lr}
				+\infty & \text{if }\; \Cdet \geq \frac{1}{2}\\
				\frac{2}{\Cinv^2 \Cbound (1-2\Cdet)} & \text{if } \;\Cdet < \frac{1}{2}
			\end{array}
			\right.
		\end{equation}		
	\end{theorem}
	\begin{proof}
		The variational formulation of the discrete DLR problem from Algorithm~\ref{alg:ourscheme} reads in this case
		\begin{align}\label{eq:semiimplvarform}
			\begin{split}
				&\hspace{0cm}\big\langle \frac{u^{n+1}_{h,\hat{\rho}} - u^n_{h,\hat{\rho}}}{\triangle t},\, v_h\big\rangle_{H,\lhr} + \big\langle u^{n+1}_{h,\hat{\rho}},\, v_h  \big\rangle_{\Ldet,\hat{\rho}} + \big( \Lstoch(u^n_{h,\hat{\rho}}),v_h \big)_{V'V,\lhr} \\
				&\hspace{0.1cm}= \big\langle f^{n,n+1},\,v_h\big\rangle_{H,\lhr}\quad \forall v_h = \bar{v}_h + v_h^* \text{ with } \bar{v}_h\in V_h\text{ and }v_h^*\in \mathcal{T}_{\tilde{U}^{n+1}Y^n} \mathcal{M}_R^{h,\hat{\rho}}.
			\end{split}
		\end{align}
		\begin{enumerate}
			\item We will consider $v_h = u^{n+1}_{h,\hat{\rho}}$ as a test function in \eqref{eq:semiimplvarform} and we derive
			\begin{align*}
				&\|u^{n+1}_{h,\hat{\rho}}\|_{H,\lhr}^2 + 2\triangle t \langle u^{n+1}_{h,\hat{\rho}}, u^{n+1}_{h,\hat{\rho}}\rangle_{\mathcal{L},\hat{\rho}}\\
				&\hspace{0.1cm}= \|u^{n}_{h,\hat{\rho}}\|_{H,\lhr}^2 - \|u^{n+1}_{h,\hat{\rho}} - u^{n}_{h,\hat{\rho}}\|_{H,\lhr}^2 + 2\triangle t\langle f^{n,n+1}, u^{n+1}_{h,\hat{\rho}}\rangle_{H,\lhr}\\
				&\hspace{0.5cm} + 2\triangle t(\Lstoch(u^{n+1}_{h,\hat{\rho}} - u^{n}_{h,\hat{\rho}}), u^{n+1}_{h,\hat{\rho}})_{V'V,\lhr}\\
				&\leq \|u^{n}_{h,\hat{\rho}}\|_{H,\lhr}^2 - \|u^{n+1}_{h,\hat{\rho}} - u^{n}_{h,\hat{\rho}}\|_{H,\lhr}^2 + \triangle t\frac{\Cemb^2}{\Ccoerc}\|f^{n,n+1}\|_{H,\lhr}^2 + \triangle t\Ccoerc \|u^{n+1}_{h,\hat{\rho}}\|_{V,\lhr}^2\\
				& \hspace{0.5cm}+ \kappa \triangle t\Ccoerc \|u^{n+1}_{h,\hat{\rho}}\|_{V,\lhr}^2 + \triangle t \frac{\Cinv^2 \Cstoch^2}{\kappa h^{2p} \Ccoerc}\|u^{n+1}_{h,\hat{\rho}} - u^{n}_{h,\hat{\rho}}\|_{H,\lhr}^2.
			\end{align*}
			Combining the terms and summing over $n = 0,\dots,N-1$ finishes the proof.
			
			\item We will proceed by taking $v_h = u^{n+1}_{h,\hat{\rho}} - u^n_{h,\hat{\rho}}$ in the variational formulation \eqref{eq:semiimplvarform} since $(u^{n+1}_{h,\hat{\rho}} - u^n_{h,\hat{\rho}})^*\in\mathcal{T}_{\tilde{U}^{n+1}Y^n} \mathcal{M}_R^{h,\hat{\rho}}$ (Lemma \ref{lemma:unun+1}). We obtain
			\begin{align}
				&\frac{1}{\triangle t} \|u^{n+1}_{h,\hat{\rho}} - u^n_{h,\hat{\rho}}\|^2_{H,\lhr} + \langle u^{n+1}_{h,\hat{\rho}},u^{n+1}_{h,\hat{\rho}} - u^n_{h,\hat{\rho}}\rangle_{\Ldet,\hat{\rho}} \\
				&\hspace{0.8cm} + \big( \Lstoch(u^n_{h,\hat{\rho}}), u^{n+1}_{h,\hat{\rho}} - u^n_{h,\hat{\rho}}\big)_{V'V,\lhr} \pm \big( \Ldet(u^n_{h,\hat{\rho}}), u^{n+1}_{h,\hat{\rho}} - u^n_{h,\hat{\rho}}\big)_{V'V,\lhr} \nonumber\\
				&\hspace{0.1cm} = \frac{1}{\triangle t} \|u^{n+1}_{h,\hat{\rho}} - u^n_{h,\hat{\rho}}\|^2_{H,\lhr} + \langle u^{n+1}_{h,\hat{\rho}} - u^n_{h,\hat{\rho}}, u^{n+1}_{h,\hat{\rho}} - u^n_{h,\hat{\rho}}\rangle_{\Ldet,\hat{\rho}} \nonumber\\
				&\hspace{0.8cm}+ \langle u^n_{h,\hat{\rho}}, u^{n+1}_{h,\hat{\rho}}-u^n_{h,\hat{\rho}}\rangle_{\mathcal{L},\hat{\rho}}\nonumber\\
				&\hspace{0.1cm}= \langle f^{n,n+1},u^{n+1}_{h,\hat{\rho}} - u^n_{h,\hat{\rho}}\rangle_{H,\lhr}\leq \frac{\triangle t}{2\kappa}\|f^{n,n+1}\|_{H,\lhr}^2 + \frac{\kappa}{2\triangle t}\|u^{n+1}_{h,\hat{\rho}} - u^n_{h,\hat{\rho}}\|_{H,\lhr}^2.\label{eq:tmpeq4}
			\end{align}
			Using Lemma \ref{lemma:scprod} we further derive 
			\begin{align*}
				\|u^{n+1}_{h,\hat{\rho}}\|_{\mathcal{L},\hat{\rho}}^2 &\leq \|u^n_{h,\hat{\rho}}\|_{\mathcal{L},\hat{\rho}}^2 + \frac{\triangle t}{\kappa}\|f^{n,n+1}\|_{H,\lhr}^2 + \|u^{n+1}_{h,\hat{\rho}} - u^n_{h,\hat{\rho}}\|_{\mathcal{L},\hat{\rho}}^2 \\
				&\hspace{0.5cm}- \frac{2-\kappa}{\triangle t} \|u^{n+1}_{h,\hat{\rho}} - u^n_{h,\hat{\rho}}\|_{H,\lhr}^2 \\
				&\hspace{0.5cm}- 2\langle u^{n+1}_{h,\hat{\rho}} - u^n_{h,\hat{\rho}}, u^{n+1}_{h,\hat{\rho}} - u^n_{h,\hat{\rho}}\rangle_{\Ldet,\hat{\rho}}\\
				&\hspace{-1.5cm}\leq \|u^n_{h,\hat{\rho}}\|_{\mathcal{L},\hat{\rho}}^2 + \frac{\triangle t}{\kappa} \|f^{n,n+1}\|_{H,\lhr}^2+ \|u^{n+1}_{h,\hat{\rho}} - u^n_{h,\hat{\rho}}\|_{\mathcal{L},\hat{\rho}}^2 \\
				&\hspace{0.5cm}-\frac{(2-\kappa) h^{2p}}{\Cinv^2\Cbound \triangle t}\|u^{n+1}_{h,\hat{\rho}} - u^n_{h,\hat{\rho}}\|_{\mathcal{L},\hat{\rho}}^2 -2\Cdet\|u^{n+1}_{h,\hat{\rho}} - u^n_{h,\hat{\rho}}\|_{\mathcal{L},\hat{\rho}}^2\\
				&\hspace{-1.5cm}= \|u^n_{h,\hat{\rho}}\|_{\mathcal{L},\hat{\rho}}^2 + \frac{\triangle t}{\kappa}\|f^{n,n+1}\|_{H,\lhr}^2 \\
				&\hspace{0.5cm}+ \big(1 - \frac{(2-\kappa) h^{2p}}{\Cinv^2\Cbound \triangle t} - 2\Cdet\big)\|u^{n+1}_{h,\hat{\rho}} - u^n_{h,\hat{\rho}}\|_{\mathcal{L},\hat{\rho}}^2,
			\end{align*}
			where in the second step we used the inequalities \eqref{eq:discDLRinv}, \eqref{eq:bound_const} and \eqref{eq:equivineq}. From the condition on $\triangle t, h$ after summing over $n=0,\dots,N-1$ the equation \eqref{eq:est_semiimpl} follows.
			\item To treat the case of $f = 0$ we follow analogous steps as in part 2. We consider $\kappa = 0$ as there is no need for the Young inequality in \eqref{eq:tmpeq4}.
		\end{enumerate}
		\qed
	\end{proof}
	
	Theorem~\ref{th:semiimplstab} tells us that when $\mathcal{L}$ is a symmetric operator, using the semi-implicit scheme leads to a conditionally stable solution if $\Cdet \in (0,\frac{1}{2})$ and an unconditionally stable solution, if $\Cdet \geq \frac{1}{2}$ (small randomness).
	
	\begin{remark}
		The discrete variational formulation \eqref{eq:discDLRvar} as well as the stability estimates presented in this section hold for the full-rank solution of the projector-splitting scheme from \cite{Lubich14} with the ordering $K,S,L$, as presented in subsection~\ref{sec:projsplit}. However, these results do not hold with the ordering $K,L,S$, which was discussed in \cite{Lubich14}. This might be another reason why $K,L,S$ performs poorly when compared to $K,S,L$ (see \cite[sec. 5.2]{Lubich14}).
	\end{remark}
	
	\begin{remark}
		All of the derived estimates for the discrete DLR solution obtained by Algorithm~\ref{alg:ourscheme} hold also for the case of $\{u_{h,\hat{\rho}}^n\}_{n=0}^N$ being rank-deficient for some $n=0,\dots,N$ as a consequence of Theorem \ref{th:discvarformrankdef} and the property \eqref{rem:unun+1}. It is not clear whether the Algorithm~\ref{alg:projsplit} satisfies a similar variational formulation. However, the numerical results from subsection \ref{sec:rankdef} seem to exhibit similar stability properties.
	\end{remark}
	
	\section{Example: random heat equation}\label{sec:heateq}
	
	In this section we will specifically address the case of a random heat equation. We will analyze what the underlying assumptions require of this problem, present the explicit and semi-implicit discretization schemes applied to a heat equation and state their stability properties.
	
	Let $D\subset \mathbb{R}^d,\; 1\leq d\leq 3$ be a polygonal domain. Let $V = H^1_0(D)=:H^1_0,\, H=L^2(D)=:L^2,\, V'=H^{-1}(D)=:H^{-1}$ and $\mathcal{L}(x,\xi)(v) = -\nabla\cdot(a(x,\xi)\nabla v)$ with 
	\begin{equation}\label{eq:heatamin}
		0 < a_{\min}\leq a(x,\xi)\leq a_{\max} < \infty,\quad\forall x\in D,\,\forall\xi\in\Gamma.
	\end{equation}
	In this case, the scalar products $\langle v,w\rangle_{H,\lr},\,\langle v,w\rangle_{V,\lr},\, \langle v,w\rangle_{\mathcal{L},\rho}$ are defined as 
	\begin{align*}
		\langle v,w\rangle_{H,\lr} &= \intG{\intD{v\,w}}\\
		\langle v,w\rangle_{V,\lr} &= \intG{\intD{\nabla v\cdot\nabla w}}\\
		\langle v,w\rangle_{\mathcal{L},\rho} &= \intG{\intD{a\nabla v\cdot\nabla w}}.
	\end{align*}
	For the coercivity constant $\Ccoerc$, it holds $\Ccoerc\geq a_{\min}$; for the continuity constant $\Cbound$, we have $\Cbound \leq a_{\max}$; $\Cemb$ is the Poincar\'{e} constant and the problem states: Given $f\in L^2(0,T; \lr(\Gamma;L^2))$ and $u_0\in \lr(\Gamma;L^2)$, find $u\tr\in L^2(0,T; \lr(\Gamma;H^1_0))$ with $\dot{u}\tr\in L^2(0,T; \lr(\Gamma;H^{-1}))$ such that
	\begin{align}\label{eq:heatpr1}
		\begin{split}
			&\intG{\intD{ \dot{u}\tr v}} + \intG{\intD{ a\nabla u\tr\cdot\nabla v}} = \intG{\intD{f v}},\\
			& \hspace{8cm} \forall v\in \lr(\Gamma;H_0^1)\\
			& u\tr = 0 \qquad\qquad \text{a.e. on }(0,T]\times \partial D\times\Gamma\\
			& u\tr(0,\cdot,\cdot) = u_0 \qquad\qquad  \text{a.e. in } D\times\Gamma. 
		\end{split}
	\end{align}
	The discretization is performed as described in Section \ref{sec:disc}. To address the condition \eqref{eq:discDLRinv} we can consider a triangulation $\mathcal{T}_h$ of the domain $D$ specified by the discretization parameter $h$ and a corresponding finite element space $V_h$ of continuous piece-wise polynomials of degree $\leq r$. Under the condition that the family of meshes $\{\mathcal{T}_h\}_h$ is quasi-uniform (see \cite[Def.\ 1.140]{Ern04/1} for definition), we have the inverse inequality (see \cite[Cor.\ 1.141]{Ern04/1})
	\begin{equation*}
		\|\nabla v\|_{H}^2 \leq \frac{\Cinv^2}{h^2} \|v\|_{H}^2,\qquad\forall v\in V_h
	\end{equation*}
	for some $\Cinv > 0$. Integrating over $\Gamma$ results in
	\begin{equation}\label{eq:discDLRinvheat}
		\|v\|_{V,\lhr}^2 \leq \frac{\Cinv^2}{h^2} \|v\|_{H,\lhr}^2,\qquad\forall v\in V_h\otimes\lhr,
	\end{equation}
	i.e.\ we have the condition \eqref{eq:discDLRinv} with $p=1$.
	\subsection{Explicit Euler scheme}
	Applying the explicit Euler scheme in the operator evaluation for a random heat equation, i.e.\ $$\mathcal{L}(u^n_{h,\hat{\rho}}, u^{n+1}_{h,\hat{\rho}}) = -\nabla\cdot(a\nabla u^n_{h,\hat{\rho}}),$$ results in the following system of equations 
	\begin{align*}
		&\langle \bar{u}^{n+1},v_h\rangle_{H} = \langle \bar{u}^{n},v_h\rangle_{H} - \triangle t\,\langle  \mathbb{E}_{\hat{\rho}}[a\nabla u^n_{h,\hat{\rho}}], \nabla v_h\rangle_{H} + \triangle t \langle \mathbb{E}\hr[f(t_n)], v_h\rangle_H, \\
		&\hspace{10.2cm} \forall v_h \in V_h\\[10pt]
		&\langle \tilde{U}^{n+1}_j,v_h \rangle_{H} = \langle {U}^{n}_j,v_h \rangle_{H} - \triangle t\, \langle \mathbb{E}_{\hat{\rho}}[a\nabla u^n_{h,\hat{\rho}} Y^n_j], \nabla v_h\rangle_{H} \\
		&\hspace{5.2cm}+ \triangle t\langle \mathbb{E}\hr[f(t_n) Y^n_j],v_h\rangle_H\qquad \forall j,\:\forall v_h \in V_h\\[5pt]
		&\tilde{M}^{n+1}(\tilde{Y}^{n+1} - Y^n)^\intercal =  - \triangle t \,\Pro_{\hat{\rho},\mathcal{Y}^n}^{\perp} \Big[  \langle a\nabla u^n_{h,\hat{\rho}},\nabla\tilde{U}^{n+1} \rangle_{H}  -  \langle f(t_n),\tilde{U}^{n+1} \rangle_{H} \Big]^\intercal\\
		&\hspace{10.7cm} \text{in }\lhr.
	\end{align*}
	The stability properties stated in Theorem \ref{th:explstab} part 2 and 4 hold under the condition
	\begin{equation*}
		\frac{\triangle t}{h^2} \leq \frac{2-\kappa}{\Cinv^2 \Cbound}.
	\end{equation*}
	
	\subsection{Semi-implicit scheme}
	Let us consider the decomposition 
	\begin{equation}
		a = \bar{a} + \astoch, \quad\text{with } \quad \bar{a} = \mathbb{E}_{\hat{\rho}}[a] \quad\text{and} \quad\mathbb{E}\hr[\astoch] = 0,
	\end{equation}
	i.e.\ $$\mathcal{L}(u) = \underbrace{-\nabla\cdot(\bar{a}\nabla u)}_{\Ldet} \underbrace{- \nabla\cdot(\astoch\nabla u)}_{\Lstoch}.$$
	
	The condition \eqref{eq:Ldet} is satisfied, since $\bar{a}$ is positive everywhere in $D$ as assumed in  \eqref{eq:heatamin}. Hence, $$\langle u,v\rangle_{\Ldet} = \intD{\bar{a}\,\nabla u\cdot\nabla v},\qquad u,v\in V$$ is a scalar product on $V = H^1_0(D)$. The semi-implicit time integration is realized by 
	\begin{equation}
		\mathcal{L}(u^n_{h,\hat{\rho}}, u^{n+1}_{h,\hat{\rho}}) = -\nabla\cdot(\bar{a}\nabla u^{n+1}_{h,\hat{\rho}}) - \nabla\cdot(\astoch\nabla u^n_{h,\hat{\rho}}).
	\end{equation}
	Note that the condition \eqref{eq:equivineq} is automatically satisfied for a random heat equation, since we have 
	\begin{multline*}
		\|u\|_{\Ldet,\rho}^2 = \intG{\intD{\bar{a}\nabla u\cdot\nabla u}} \geq \inf_{x\in D,\xi\in\Gamma}\frac{\bar{a}}{a}\:\intG{\intD{a\nabla u\cdot\nabla u}} \\= \inf_{x\in D,\xi\in\Gamma}\frac{\bar{a}}{a}\:\|u\|_{\mathcal{L},\rho}^2
		\qquad\forall u\in \lrV,
	\end{multline*}
	and $\inf_{x\in D,\xi\in\Gamma}\frac{\bar{a}}{a}\geq \frac{a_{\min}}{a_{\max}}>0$.\\
	
	The system of equations \eqref{eq:semieq1}--\eqref{eq:semieq3} can be rewritten as
	\begin{align*}
		&\langle \bar{u}^{n+1}, v_h\rangle_{H} + \triangle t\langle \bar{a}\nabla \bar{u}^{n+1}, \nabla v_h\rangle_H \\
		&\hspace{1.5cm}= \langle \bar{u}^{n}, v_h\rangle_{H} - \triangle t \langle \mathbb{E}_{\hat{\rho}}[\astoch\nabla u^n_{h,\hat{\rho}}], \nabla v_h\rangle_{H^d} + \triangle t\langle\mathbb{E}\hr[f^{n,n+1}],v_h\rangle_H\\[3pt]
		&\langle \tilde{U}^{n+1}_j, v_h\rangle_{H} + \triangle t\langle \bar{a}\nabla\tilde{U}^{n+1}_j, \nabla v_h\rangle_{H}\\
		&\hspace{0.4cm} =  \langle \tilde{U}^{n}_j, v_h\rangle_{H} - \triangle t \langle \mathbb{E}_{\hat{\rho}}[\astoch \nabla u^n_{h,\hat{\rho}} Y^n_j], \nabla v_h\rangle_{H^d} + \triangle t\langle\mathbb{E}\hr[f^{n,n+1}Y^n_j],v_h\rangle_H\\[3pt]
		&\Big(\tilde{Y}^{n+1} - Y^n\Big)\big( \tilde{M}^{n+1} + \triangle t\langle\bar{a}\nabla\tilde{U}^{{n+1}^\intercal}, \nabla\tilde{U}^{n+1}\rangle_{H} \big)\\
		&\hspace{2.6cm}= - \triangle t \Pro_{\hat{\rho},\mathcal{Y}^n}^{\perp} [\langle\astoch\nabla u^n_{h,\hat{\rho}}, \nabla\tilde{U}^{n+1}\rangle_{H^d} - \langle f^{{n,n+1}^*}, \tilde{U}^{n+1}\rangle_H].
	\end{align*}
	For a further specified diffusion coefficient we can state the following stability properties.

	\begin{proposition}\label{lemma:semiimplheat}
		For the case $$\bar{a}(x) \geq a_{\mathrm{stoch}}(x,\xi),\qquad \forall x\in D,\xi\in\Gamma$$ which is satisfied in particular if
		\begin{equation}\label{eq:heatdiff}
			\begin{gathered}
				a(x,\xi) = \bar{a}(x) + \sum_{j=1}^M a_j(x) \xi_j,\\
				\Gamma\subset\mathbb{R}^M \text{ and } \Gamma \text{ is symmetric, i.e.\ }\xi\in\Gamma\implies -\xi\in\Gamma,
			\end{gathered}
		\end{equation}
		we have the stability properties \eqref{eq:est_semiimpl} and \eqref{eq:est_semiimpl0f} for any $\triangle t, h$.
	\end{proposition}
	
	\begin{proof}
		The condition $\bar{a}(x) \geq a_{\mathrm{stoch}}(x,\xi)$ for every $ x\in D,\xi\in\Gamma$ implies 
		\begin{equation*}
			\frac{\bar{a}(x)}{a(x,\xi)} \geq \frac{1}{2},
		\end{equation*}
		i.e.\ $\Cdet\geq\inf_{x\in D,\xi\in\Gamma} \frac{\bar{a}}{a}\geq \frac{1}{2}.$ Together with Theorem \ref{th:semiimplstab} we conclude the result.
		\qed
	\end{proof}

	Proposition \ref{lemma:semiimplheat} tells us that applying a semi-implicit scheme to solve a heat equation with diffusion coefficient as described in \eqref{eq:heatdiff} results in an unconditionally stable scheme. This result as well as some of the previous estimates will be numerically verified in the following section.
	
	\section{Numerical results}\label{sec:numres}
	
	This section is dedicated to numerically study the stability estimates derived for a discrete DLR approximation in Section~\ref{sec:stabest}. In particular, we will be concerned with a random heat equation, as introduced in \eqref{eq:heatpr1}, with zero forcing term and diffusion coefficient of the form \eqref{eq:heatdiff}. We will look at the behavior of suitable norms of the solutions of the discretization schemes introduced in Section \ref{sec:fullydiscpr}. We will as well look at a discretization scheme in which the projection is performed explicitly to see how important it is to project on the new computed basis $\tilde{U}^{n+1}$ in \eqref{eq:discDLReq3}. As a last result we provide a comparison with the projector-splitting scheme from \cite{Lubich14}.
	
	Let us consider problem \eqref{eq:heatpr1} set in a unit square $D = [0,1]^2$ and sample space $\Gamma = [-1, 1]^M$ with an uncertain diffusion coefficient 
	\begin{equation}
		a(x,\xi) = a_0 + \sum_{m=1}^M \frac{\cos(2\pi m x_1) + \cos(2\pi m x_2)}{m^2 \pi^2}\xi_m,
	\end{equation}
	where $x = (x_1, x_2)\in D,\; \xi = (\xi_1,\dots,\xi_M)\in\Gamma$. 
	%We consider independent and uniformly distributed random variables $\xi_m\sim U([-1,\,1])$  and $a_0 = 0.3$. 
	We let $a_0 = 0.3$, and equip $([-1,1]^M,\mathcal{B}([-1,1]^M))$ with the uniform measure $\rho(\dx \xi) = \bigotimes_{i=1}^{M} \frac{\lambda(\dx \xi_i)}{2}$ with $\lambda$ the Lebesgue measure restricted to the Borel $\sigma$-algebra  $\mathcal{B}([-1,1])$.
	In this case the conditions \eqref{eq:heatamin}, \eqref{eq:Ldet} and \eqref{eq:equivineq} are satisfied with $a_{\min} > 0.04, \Cdet> \frac{1}{2}$. The initial condition is chosen as
	\begin{align}
		\begin{split}\label{eq:init_cond}
			u_0&(x,\xi) = 10\sin(\pi x_1)\sin(\pi x_2) + 2\sin(2\pi x_1)\sin(2\pi x_2)\xi_1 \\
			&\hspace{0.5cm} + 2\sin(4\pi x_1)\sin(4\pi x_2)\xi_2 + 2\sin(6\pi x_1)\sin(6\pi x_2)\xi_1^2.\\
			&=10\sin(\pi x_1)\sin(\pi x_2) + \frac{4}{3}\sin(6\pi x_1)\sin(6\pi x_2)+ 2\sin(2\pi x_1)\sin(2\pi x_2)\xi_1\\
			&\hspace{0.5cm} + 2\sin(4\pi x_1)\sin(4\pi x_2)\xi_2 + 2\sin(6\pi x_1)\sin(6\pi x_2)\big(\xi_1^2 - \mathbb{E}[\xi_1^2]\big).
		\end{split}
	\end{align}	
	
	The spatial discretization is performed by the finite element (FE) method with $P1$ finite elements over a uniform mesh. The dimension of the corresponding FE space is determined by $h$---the element size. For this type of spatial discretization we have the inverse inequality \eqref{eq:discDLRinvheat}: $$\|v\|_{V,\hat{\rho}}^2 \leq \frac{\Cinv^2}{h^2} \|v\|_{H,\lhr}^2,\qquad\forall v\in V_h\otimes\lhr.$$ Concerning the stochastic discretization we will consider a tensor grid quadrature with Gauss-Legendre points for the case of a low-dimensional stochastic space $M = 2$ and a Monte-Carlo quadrature for the case $M = 10$. The time integration implements the explicit scheme and the semi-implicit scheme described in subsection~\ref{sec:fullydiscpr}. We will consider the forcing term $f = 0$, i.e.\ a dissipative problem and time $T$ such that the energy norm ($\|\cdot\|_{\mathcal{L},\hat{\rho}}$) of the solution attains a value smaller than $10^{-10}$. Our simulations were performed using the Fenics library \cite{Fenics}.
	\subsection{Explicit scheme}
	Since $f = 0$, the result in Theorem \ref{th:explstab} predicts a decay of the norm of the solution 
	\begin{equation*}
		\|u^{n+1}_{h,\hat{\rho}}\|_{\mathcal{L},\hat{\rho}}\leq\|u^n_{h,\hat{\rho}}\|_{\mathcal{L},\hat{\rho}},\qquad\|u^{n+1}_{h,\hat{\rho}}\|_{H,\lhr}\leq\|u^n_{h,\hat{\rho}}\|_{H,\lhr}\qquad\forall n=0,\dots,N-1
	\end{equation*}
	under the stability condition 
	\begin{equation}\label{eq:numexplcond}
		\frac{\triangle t}{h^2}\leq \frac{2}{\Cinv^2\Cbound} = : K.
	\end{equation}
	We aim at verifying such result numerically. We set a rank $R = 3$ and consider a sample space $[-1,1]^M$ of dimension $M = 2$ or $M = 10$ with either Gauss-Legendre or Monte-Carlo (MC) stochastic discretization.
	\subsubsection{$M = 2$}
	First we consider the sample space $[-1,1]^M$ of dimension $M=2$ and Gauss-Legendre quadrature with $9\times 9 = 81$ collocation points. 
	From what we observed in our simulations, for this test case we have $K\approx 0.085$. Figure \ref{fig:expl3scen} shows the behavior of the energy norm ($\|\cdot\|_{\mathcal{L},\hat{\rho}}$) and the $L^2$ norm ($\|\cdot\|_{H,\lhr}$) in 3 different scenarios: in the first scenario we set $h_1 = 0.142, \triangle t_1 = 0.0018$, i.e.\ the condition $\triangle t_1/h_1^2 \leq K$ is satisfied and observe that both the energy norm and the $L^2$ norm of the solution decrease in time (see Figure \ref{fig:expl3scen}(a)); in the second scenario, we halved the element size $h_2 = h_1/2$ and divided by 4 the time step $\triangle t_2 = \triangle t_1/4$ so that the condition \eqref{eq:numexplcond} is still satisfied. The norms again decreased in time (Figure \ref{fig:expl3scen}(b)); in the third scenario we violated the condition \eqref{eq:numexplcond} by setting $h_3 = h_1/2$ and $\triangle t_3 = \triangle t_1/3$. After a certain time the norms exploded (Figure \ref{fig:expl3scen}(c)).
	
	\begin{figure} % "[t!]" placement specifier just for this example
		\captionsetup[subfigure]{justification=centering}
		\centering
		\begin{minipage}{0.32\textwidth}
			\center\includegraphics[width=\linewidth]{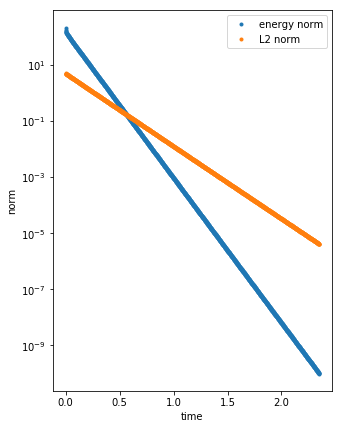}
			\subcaption{\footnotesize$h_1~=~0.142,$\\$\triangle t_1~=~0.0017$}
		\end{minipage}
		\begin{minipage}{0.32\textwidth}
			\center\includegraphics[width=\linewidth]{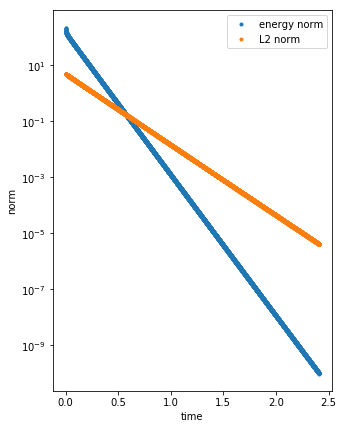}
			\subcaption{\footnotesize$h_1~=~0.142/2,$\\$\triangle t_1~=~0.0017/4$}
		\end{minipage}
		\begin{minipage}{0.32\textwidth}
			\center\includegraphics[width=\linewidth]{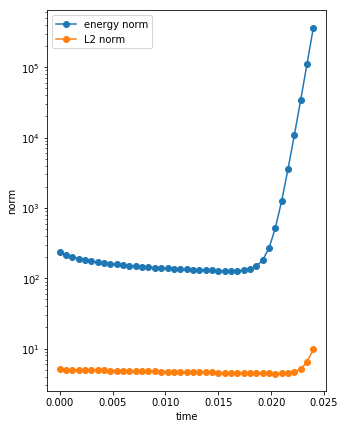}
			\subcaption{\footnotesize$h_1~=~0.142/2,$\\$\triangle t_1~=~0.0017/3$}
		\end{minipage}
		\caption{Behavior of the energy norm ($\|\cdot\|_{\mathcal{L},\rho}$---blue) and the $L^2$ norm ($\|\cdot\|_{H,\lr}$--- orange) when applying the explicit time integration scheme with $M=2$ and 81 Gauss-Legendre collocation points for three different pairs of the discretization parameters $h,\triangle t$. %We see that satisfying the condition \eqref{eq:numexplcond} ((a) and (b)) results in stable behavior while when violating the condition (c) the solution blows up.
			When the condition \eqref{eq:numexplcond} is satisfied the solution is stable [(a)--(b)], whereas violating the condition results in instability [(c)]
		} \label{fig:expl3scen}
	\end{figure}

	To numerically demonstrate the sharpness of the condition \eqref{eq:numexplcond}, we ran the simulation with 72 different pairs of discretization parameters $h,\triangle t$. The results are shown in Figure \ref{fig:expl}, where we depict whether the energy norm at time $T$ is bellow $10^{-10}$, in which case the norm was consistently decreasing; or more than $10^4$, in which case the solution blew up. We observe that a stable $\triangle t$ has to be chosen to satisfy $\triangle t \leq K h^2$, which confirms the sharpness of our theoretical derivations.
	
	\begin{figure}
		\center\includegraphics[width=0.75\linewidth]{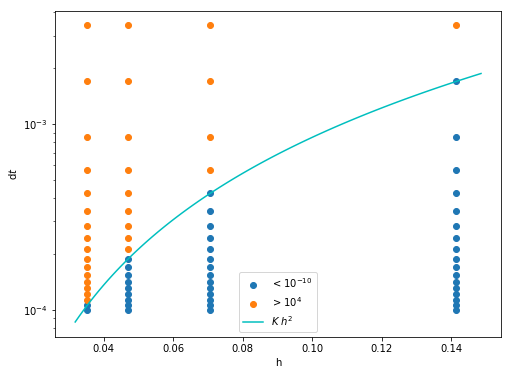}
		\caption{This figure shows whether the energy norm $\|\cdot\|_{\mathcal{L},\rho}$ of the solution was monotonously decreasing till $10^{-10}$ (blue) or has blown up (orange) for different choices of time step $\triangle t$ and discretization parameter $h$ when applying the explicit scheme for the operator evaluation. We observe a clear quadratic dependence of $\triangle t$ on $h$. $K$ was set to $0.085$}\label{fig:expl}
	\end{figure}
	\subsubsection{$M = 10$}
	In our second example we will consider a higher-dimensional problem: $M = 10$ for which we use a standard Monte-Carlo technique with $50$ points. We observe a very similar behavior as in the small dimensional case. Figure \ref{fig:expl3scenMC} shows that satisfying the condition $\triangle t_1/h_1^2 \leq K$ with $K = 0.085$ results in a stable scheme while violating it makes the solution blow up.
	
	\begin{figure} % "[t!]" placement specifier just for this example
		\captionsetup[subfigure]{justification=centering}
		\centering
		\begin{minipage}{0.32\textwidth}
			\center\includegraphics[width=\linewidth]{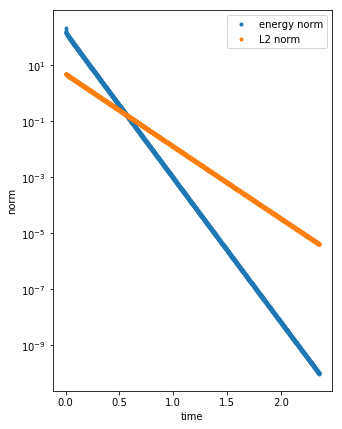}
			\subcaption{\footnotesize$h_1=0.142,$\\$\triangle t_1 = 0.0017$}
		\end{minipage}
		\begin{minipage}{0.32\textwidth}
			\center\includegraphics[width=\linewidth]{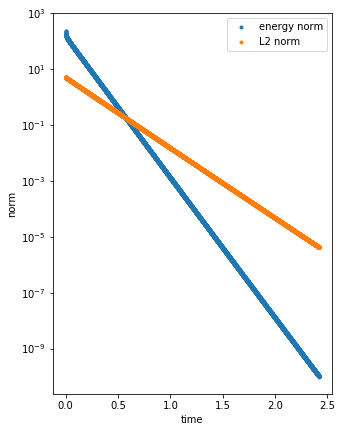}
			\subcaption{\footnotesize$h_1 = 0.142/2,$\\$\triangle t_1 = 0.0017/4$}
		\end{minipage}
		\begin{minipage}{0.32\textwidth}
			\center\includegraphics[width=\linewidth]{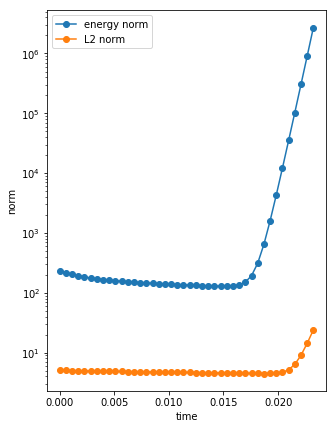}
			\subcaption{\footnotesize$h_1 = 0.142/2,$\\$\triangle t_1 = 0.0017/3$}
		\end{minipage}
		\caption{Behavior of the energy norm ($\|\cdot\|_{\mathcal{L},\rho}$---blue) and the $L^2$ norm ($\|\cdot\|_{H,\lr}$---orange) when applying the explicit time integration scheme with $M=10$ and 50 Monte Carlo points for three different pairs of the discretization parameters $h,\triangle t$. We see, again, that satisfying the condition \eqref{eq:numexplcond} ((a) and (b)) results in stable behavior while when violating the condition (c) the solution blows up} \label{fig:expl3scenMC}
	\end{figure}
	
	\subsection{Semi-implicit scheme}
	
	We proceed with the same test-case with $M = 10$, same spatial and stochastic discretization, i.e.\ Monte-Carlo method with 50 samples and employ a semi-implicit scheme in the operator evaluation.  
	Since the diffusion coefficient considered is of the form \eqref{eq:heatdiff} and $f = 0$, Theorem \ref{th:semiimplstab} predicts $$\|u^{n+1}_{h,\hat{\rho}}\|_{\mathcal{L},\hat{\rho}}\leq\|u^n_{h,\hat{\rho}}\|_{\mathcal{L},\hat{\rho}}\qquad \forall h,\,\triangle t,\;\forall n=0, \dots, N-1.$$ We set the spatial discretization $h = 0.142$ and vary the time step $\triangle t$. We observe a stable behavior no matter what $\triangle t$ is used, which confirms the theoretical result (see Figure \ref{fig:impl2scen}).
	
	\begin{figure} % "[t!]" placement specifier just for this example
		\begin{minipage}{0.45\textwidth}
			\center\includegraphics[width=0.9\linewidth]{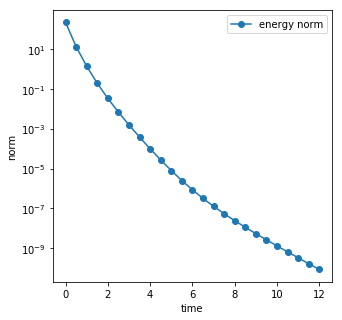}
			\subcaption{\footnotesize$h=0.142,\;\triangle t_1 = 0.5$}
		\end{minipage}
		\begin{minipage}{0.45\textwidth}
			\center\includegraphics[width=0.9\linewidth]{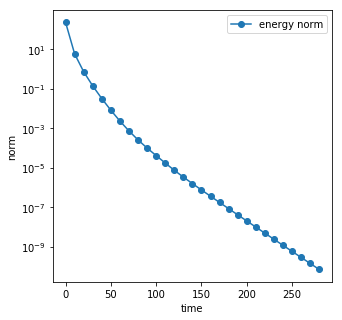}
			\subcaption{\footnotesize$h = 0.142,\;\triangle t_2 = 10$}
		\end{minipage}
		\caption{Behavior of the energy norm ($\|\cdot\|_{\mathcal{L},\rho}$) for two different time steps when applying the semi-implicit time integration scheme. We observe a decrease of norms for arbitrarily large time step} \label{fig:impl2scen}
	\end{figure}
	
	We report that the results for $M=2$ with 81 Gauss-Legendre collocation points exhibited a similar unconditionally-stable behavior.
	
	\subsubsection{Explicit projection}
	
	The following results give an insight into the importance of performing the projection in a `Gauss-Seidel' way, i.e.\ projection on the stochastic basis is done explicitly, $Y^n$ kept from the previous time step, while the projection on the deterministic basis is done implicitly, i.e.\ we use the new computed $\tilde{U}^{n+1}$ (see Algorithm~\ref{alg:ourscheme} for more details). For comparison we consider a fully explicit projection, i.e.\ $Y^n$ as the stochastic basis and $U^n$ as the deterministic basis. We use a semi-implicit scheme to treat the operator evaluation term as described in subsection~\ref{sec:fullydiscpr}. As shown in Figure \ref{fig:proj_expl}, in all 3 cases the solution reaches the zero steady state, however, not in a monotonous way. 
	\begin{figure} % "[t!]" placement specifier just for this example
		\begin{minipage}{0.32\textwidth}
			\center\includegraphics[width=\linewidth]{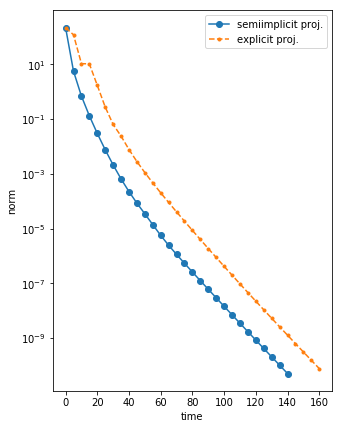}
			\subcaption{\footnotesize$h=0.142,\;\triangle t_1 = 5$}
		\end{minipage}
		\begin{minipage}{0.32\textwidth}
			\center\includegraphics[width=\linewidth]{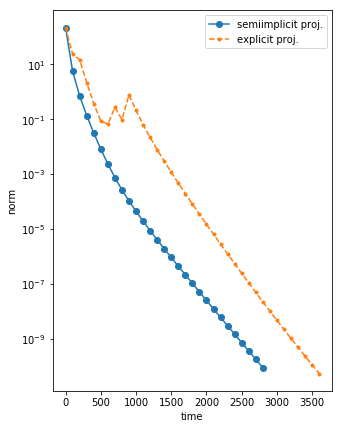}
			\subcaption{\footnotesize$h = 0.142,\;\triangle t_2 = 100$}
		\end{minipage}
		\begin{minipage}{0.32\textwidth}
			\center\includegraphics[width=\linewidth]{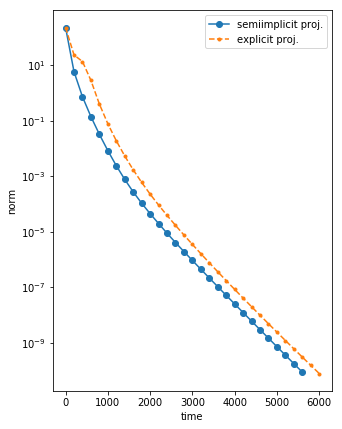}
			\subcaption{\footnotesize$h = 0.142,\;\triangle t_3 = 200$}
		\end{minipage}
		\caption{Behavior of the energy norm ($\|\cdot\|_{\mathcal{L},\rho}$) for 3 different time steps when treating the projection in an explicit way (orange) and in a semi-implicit way (blue). We used the semi-implicit scheme for the operator evaluation term. We see that, as opposed to a semi-implicit projection, with an explicit projection we do not obtain an unconditional norm decrease} \label{fig:proj_expl}
	\end{figure}
	
	\subsection{Comparison with the DDO projector-splitting scheme}\label{sec:rankdef}
	
	We now compare the performance of the discretization scheme from Algorithm~\ref{alg:ourscheme} with the projector-splitting scheme from Algorithm~\ref{alg:projsplit}.
	
	We proceed with setting $h = 0.142, M = 10, \triangle t = 100$, stochastic discretization is performed again by Monte-Carlo method with 50 points and we implemented the semi-implicit scheme in the operator evaluation for both the Algorithm~\ref{alg:ourscheme} and the projector-splitting Algorithm~\ref{alg:projsplit}. We expect that the energy norm decreases on every step independently of the time step size.
	
	We fix $R = 3$. Throughout the whole simulation, the computed solution stays full rank, in which case the two schemes have been shown to be equivalent (see subsection~\ref{sec:projsplit}). In Figure \ref{fig:splitR3}(a) this can be well observed. Steps~2 and 5 from Algorithm~\ref{alg:projsplit} are performed by a QR decomposition, whereas the linear system in \eqref{eq:discDLReq3} is solved by the Cholesky factorization (with a help of the SciPy library \cite{Scipy}, version 0.19.1).

	\begin{figure} % "[t!]" placement specifier just for this example
		\begin{minipage}{0.45\textwidth}
			\center\includegraphics[width=0.9\linewidth]{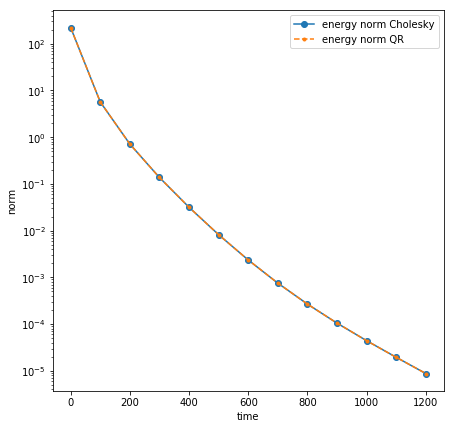}
			\subcaption{\footnotesize$R = 3,\; \triangle t_1 = 100$}
		\end{minipage}
		\begin{minipage}{0.45\textwidth}
			\center\includegraphics[width=0.9\linewidth]{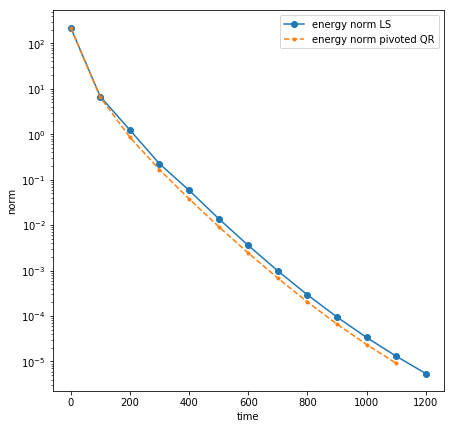}
			\subcaption{\footnotesize$R = 20, \;\triangle t_2 = 100$}
		\end{minipage}
		\caption{Energy norm ($\|\cdot\|_{\mathcal{L},\rho}$) for 2 different ranks $R = 3,20$ and 2 different time discretization schemes: Algorithm~\ref{alg:projsplit} with (pivoted) QR decomposition (orange) and Algorithm~\ref{alg:ourscheme} with Cholesky factorization or least squares. Both methods in both cases exhibit a monotonous decrease of the energy norm} \label{fig:splitR3}
	\end{figure}
	
	We now investigate the behavior of the two algorithms in presence of a rank deficient solution. We fix $R = 20$. The initial condition \eqref{eq:init_cond} is of rank $3$. For the first couple of steps the DLR solution therefore stays of rank lower than $R = 20$. The matrix $\tilde{M}^{n+1}$ from \eqref{eq:discDLReq3} is singular and the solution of the system \eqref{eq:discDLReq3} is obtained as a least squares solution implemented via an SVD decomposition. The threshold to detect the effective rank of $\tilde{M}^{n+1}$ is set to $\varepsilon\, \sigma_1 R$ where $\varepsilon$ is the machine precision and $\sigma_1$ is the largest singular value of $\tilde{M}^{n+1}$. Steps~2 and 5 from Algorithm~\ref{alg:projsplit} are performed by a pivoted QR decomposition. The solution obtained by Algorithm~\ref{alg:ourscheme} is proved to be stable in this scenario. The two proposed schemes exhibit minor differences, however both of them are stable (see Figure~\ref{fig:splitR3}(b)).
	
	\section{Conclusions}\label{sec:conclusions}
	In this work we proposed and analyzed three types of discretization schemes, namely explicit, implicit and semi-implicit, to obtain a numerical solution of the DLR system of evolution equations for the deterministic and stochastic modes. Such discrete DLR solution was obtained by projecting the discretized dynamics on the tangent space of the low-rank manifold at an intermediate point. This point was built using the new-computed deterministic modes and old stochastic modes. We found this projection property to be useful when investigating stability of the DLR solution. The solution obtained by the implicit scheme remains unconditionally bounded by the data in suitable norms. Concerning the explicit and semi-implicit schemes, we derived stability conditions on the time step, independent of the smallest singular value, under which the solution remains bounded. Remarkably, applying the proposed semi-implicit scheme to a random heat equation with diffusion coefficient affine with respect to random variables results in a scheme unconditionally stable, with the same computational complexity as the explicit scheme. Our theoretical derivations are supported by numerical tests applied to a random heat equation with zero forcing term. In the semi-implicit case, we observed that the norm of the solution consistently decreases for every time-step considered. In the explicit case, our numerical results suggest that our theoretical stability condition on the time step is in fact sharp. Our future work includes investigating if the proposed approach can be extended to higher-order projector-splitting integrators, or used to show stability properties for other types of equations.
	
	\section*{Acknowledgments}
This work has been supported by the Swiss National Science Foundation under the Project n. 172678 ``Uncertainty Quantification Techniques for PDE constrained optimization and random evolution equations''.

\appendix
\section*{Appendix}
Let $(\Omega,\mathcal{F}, \rho)$ be a measure space. Let $V$ be a separable
Banach space, and $V'$ be its topological dual space. Let $\mathfrak{L}(V,V')$
be the space of bounded linear operators equipped with the operator
norm. Moreover, let $\mathcal{\mathcal{B}}(V)$, $\mathcal{\mathcal{B}}(V')$,
$\mathcal{B}(\mathfrak{L}(V,V'))$ be the corresponding Borel $\sigma$-algebras.
\begin{propap}\label{prop:m-bility}
	Suppose that $L\colon\Omega\to\mathfrak{L}(V,V')$ is $\mathcal{F}/\mathcal{B}(\mathfrak{L}(V,V'))$-measurable.
	Let a measurable mapping $(\Omega,\mathcal{F})\ni\omega\mapsto v(\omega)\in(V,\mathcal{B}(V))$
	be given. Then, the mapping
	\[
	(\Omega,\mathcal{F})\ni\omega\mapsto L(\omega)v(\omega)\in(V',\mathcal{B}(V'))
	\]
	is measurable. In particular, if $V'$ is separable,  $(\Omega,\mathcal{F})\ni\omega\mapsto L(\omega)v(\omega)\in V'$
	is strongly measurable.
\end{propap}
\begin{proof}
	We will show that the mapping $\omega\mapsto L(\omega)v(\omega)$
	is the composition of measurable mappings
	\[
	(\Omega,\mathcal{F})\xrightarrow{\omega\mapsto(\omega,v(\omega))}(\Omega\times V,\mathcal{\mathcal{F}}\otimes\mathcal{B}(V))\xrightarrow{(\omega,\phi)\mapsto L(\omega)\phi}(V',\mathcal{B}(V')).
	\]
	The first mapping is measurable, since for every product set $A\times B\in\mathcal{\mathcal{F}}\times\mathcal{B}(V)$
	its pre-image is in $\mathcal{F}$. We show that the second mapping
	is measurable. First, notice that for each $\phi\in V$
	\[
	L(\cdot)\phi\colon\Omega\to V'
	\]
	is $\mathcal{F}/\mathcal{B}(V')$-measurable. Indeed, from the assumption,
	$\omega\mapsto L(\omega)\in\mathfrak{L}(V,V')$ is $\mathcal{F}/\mathcal{B}(\mathfrak{L}(V,V'))$
	measurable, and the mapping $\mathfrak{L}(V,V')\ni\bar{L}\mapsto\bar{L}\phi\in V'$ is 
	continuous. Thus, the $\mathcal{F}/\mathcal{B}(V')$-measurability
	follows. Therefore, since $L(\omega)\colon V\to V'$ is continuous
	for each $\omega\in\Omega$, the mapping
	
	\[
	(\Omega\times V,\mathcal{\mathcal{F}}\otimes\mathcal{B}(V))\ni(\omega,\phi)\mapsto L(\omega)\phi\in(V',\mathcal{B}(V'))
	\]
	is a Carath\'{e}odory function. 
	Hence, from the separability of $V$, the
	measurability of the second mapping follows, see \cite[Lemma 4.51]{Aliprantis.C.D_Border_2006_book}.
	Now the proof is complete.
\end{proof}

\clearpage
\bibliography{document.bib}

\begin{thebibliography}{10}

\bibitem{Aliprantis.C.D_Border_2006_book}
C.~D. Aliprantis and K.~C. Border.
\newblock {\em Infinite Dimensional Analysis}.
\newblock {Springer, Berlin}, 3 edition, 2006.

\bibitem{Fenics}
M.~S. Aln{\ae}s, J.~Blechta, J.~Hake, A.~Johansson, B.~Kehlet, A.~Logg,
  C.~Richardson, J.~Ring, M.~E. Rognes, and G.~N. Wells.
\newblock The fenics project version 1.5.
\newblock {\em Arch. of Numer. Softw.}, 3(100), 2015.

\bibitem{Bachmayr20}
M.~{Bachmayr}, H.~{Eisenmann}, E.~{Kieri}, and A.~{Uschmajew}.
\newblock Existence of dynamical low-rank approximations to parabolic problems.
\newblock {\em Math. Comp.}, 90:1799--1830, 2021.

\bibitem{Meyer00}
M.~H. Beck, A.~J{\"a}ckle, G.~Worth, and H.-D. Meyer.
\newblock The multiconfiguration time-dependent {H}artree ({MCTDH}) method: a
  highly efficient algorithm for propagating wavepackets.
\newblock {\em Phys. Rep.}, 324(1):1--105, 2000.

\bibitem{Berkooz93}
G.~Berkooz, P.~Holmes, and J.~L. Lumley.
\newblock The proper orthogonal decomposition in the analysis of turbulent
  flows.
\newblock {\em Annu. Rev. Fluid Mech.}, 25(1):539--575, 1993.

\bibitem{Carlberg11}
K.~Carlberg and C.~Farhat.
\newblock A low-cost, goal-oriented compact proper orthogonal decomposition
  basis for model reduction of static systems.
\newblock {\em Int. J. Num. Methods Eng.}, 86(3):381--402, 2011.

\bibitem{Cheng13}
M.~Cheng, T.~Y. Hou, and Z.~Zhang.
\newblock A dynamically bi-orthogonal method for time-dependent stochastic
  partial differential equations i: Derivation and algorithms.
\newblock {\em J. Comput. Phys.}, 242:843--868, 2013.

\bibitem{Cheng13b}
M.~Cheng, T.~Y. Hou, and Z.~Zhang.
\newblock A dynamically bi-orthogonal method for time-dependent stochastic
  partial differential equations ii: Adaptivity and generalizations.
\newblock {\em J. Comput. Phys.}, 242:753--776, 2013.

\bibitem{Choi14}
M.~Choi, T.~P. Sapsis, and G.~E. Karniadakis.
\newblock On the equivalence of dynamically orthogonal and bi-orthogonal
  methods: Theory and numerical simulations.
\newblock {\em J. Comput. Phys.}, 270:1--20, 2014.

\bibitem{Cohen11}
A.~Cohen, R.~Devore, and C.~Schwab.
\newblock Analytic regularity and polynomial approximation of parametric and
  stochastic elliptic {PDE}s.
\newblock {\em Anal. Appl.}, 09(1):11--47, 2011.

\bibitem{Einkemmer19}
L.~Einkemmer.
\newblock A low-rank algorithm for weakly compressible flow.
\newblock {\em SIAM J. Sci. Comput.}, 41(5):A2795--A2814, 2019.

\bibitem{Einkemmer18}
L.~Einkemmer and C.~Lubich.
\newblock A low-rank projector-splitting integrator for the vlasov--poisson
  equation.
\newblock {\em SIAM J. Sci. Comput.}, 40(5):B1330--B1360, 2018.

\bibitem{Ern04/1}
A.~Ern and J.-L. Guermond.
\newblock {\em Finite Element Interpolation}, pages 3--80.
\newblock Springer New York, New York, NY, 2004.

\bibitem{Ern04/6}
A.~Ern and J.-L. Guermond.
\newblock {\em Time-Dependent Problems}, pages 279--334.
\newblock Springer New York, New York, NY, 2004.

\bibitem{Falco19}
A.~Falc{\'o}, W.~Hackbusch, and A.~Nouy.
\newblock On the {D}irac--{F}renkel variational principle on tensor banach
  spaces.
\newblock {\em Found. Comput. Math.}, 19(1):159--204, Feb 2019.

\bibitem{Feppon18}
F.~Feppon and P.~F.~J. Lermusiaux.
\newblock A geometric approach to dynamical model order reduction.
\newblock {\em SIAM J. Matrix Anal. Appl.}, 39(1):510--538, 2018.

\bibitem{Golub96}
G.~H. Golub and C.~F. Van~Loan.
\newblock {\em Matrix Comput.}
\newblock Johns Hopkins University Press, Baltimore, MD, USA, 3 edition, 1996.

\bibitem{Scipy}
E.~Jones, T.~Oliphant, P.~Peterson, et~al.
\newblock {SciPy}: Open source scientific tools for {Python}, 2001.

\bibitem{Kazashi20}
Y.~{Kazashi} and F.~{Nobile}.
\newblock Existence of dynamical low rank approximations for random semi-linear
  evolutionary equations on the maximal interval.
\newblock {\em Stoch PDE: Anal Comp}, 9:603--629, 2021.

\bibitem{Kieri16}
E.~Kieri, C.~Lubich, and H.~Walach.
\newblock Discretized dynamical low rank approximation in the presence of small
  singular values.
\newblock {\em SIAM J. on Numer. Anal.}, 54(2):1020--1038, 2016.

\bibitem{Kieri18}
E.~Kieri and B.~Vandereycken.
\newblock Projection methods for dynamical low-rank approximation of
  high-dimensional problems.
\newblock {\em Comput. Methods Appl. Math.}, 19(1):73--92, 2018.

\bibitem{Koch06}
O.~Koch, W.~Kreuzer, and A.~Scrinzi.
\newblock Approximation of the time-dependent electronic {S}chr\"odinger
  equation by {MCTDHF}.
\newblock {\em Appl. Math. Comput.}, 173(2):960--976, Feb. 2006.

\bibitem{Koch07}
O.~Koch and C.~Lubich.
\newblock Dynamical low rank approximation.
\newblock {\em SIAM J. Matrix Anal. Appl.}, 29(2):434--454, 2007.

\bibitem{Koch07b}
O.~Koch and C.~Lubich.
\newblock Regularity of the multi-configuration time-dependent {H}artree
  approximation in quantum molecular dynamics.
\newblock {\em ESAIM: M2AN}, 41(2):315--331, 2007.

\bibitem{Koch10}
O.~Koch and C.~Lubich.
\newblock Dynamical tensor approximation.
\newblock {\em SIAM J. Matrix Anal. Appl.}, 31(5):2360--2375, 2010.

\bibitem{Lemaitre10}
O.~Le~Maitre and O.~Knio.
\newblock {\em Spectr. Methods Uncertain. Quantif.}
\newblock Springer, 2010.

\bibitem{Leoni17}
G.~Leoni.
\newblock {\em A first course in Sobolev spaces (2nd Ed.)}.
\newblock American Mathematical Society, Providence, Rhode Island, 2017.

\bibitem{Lubich08}
C.~Lubich.
\newblock {\em From quantum to classical molecular dynamics: reduced models and
  numerical analysis}.
\newblock European Mathematical Society, 2008.

\bibitem{Lubich14}
C.~Lubich and I.~V. Oseledets.
\newblock A projector-splitting integrator for dynamical low-rank
  approximation.
\newblock {\em BIT Numer. Math.}, 54(1):171--188, Mar 2014.

\bibitem{Lubich15}
C.~Lubich, I.~V. Oseledets, and B.~Vandereycken.
\newblock Time integration of tensor trains.
\newblock {\em SIAM J. Numer. Anal.}, 53(2):917--941, 2015.

\bibitem{Lubich13}
C.~Lubich, T.~Rohwedder, R.~Schneider, and B.~Vandereycken.
\newblock Dynamical approximation of hierarchical {T}ucker and tensor-train
  tensors.
\newblock {\em SIAM J. Matrix Anal. Appl.}, 34(2):470--494, 2013.

\bibitem{Musharbash15}
E.~Musharbash, F.~Nobile, and T.~Zhou.
\newblock Error analysis of the dynamically orthogonal approximation of time
  dependent random {PDE}s.
\newblock {\em SIAM J. Sci. Comput.}, 37(2):A776--A810, 2015.

\bibitem{Nobile09}
F.~Nobile and R.~Tempone.
\newblock Analysis and implementation issues for the numerical approximation of
  parabolic equations with random coefficients.
\newblock {\em Int. J. Num. Methods Eng.}, 80(6-7):979--1006, 2009.

\bibitem{Sapsis09}
T.~P. Sapsis and P.~F.~J. Lermusiaux.
\newblock Dynamically orthogonal field equations for continuous stochastic
  dynamical systems.
\newblock {\em Phys. D: Nonlinear Phenom.}, 238(23):2347--2360, 2009.

\bibitem{Sapsis12}
T.~P. Sapsis and P.~F.~J. Lermusiaux.
\newblock Dynamical criteria for the evolution of the stochastic dimensionality
  in flows with uncertainty.
\newblock {\em Phys. D: Nonlinear Phenom.}, 241(1):60--76, 2012.

\bibitem{Ueckermann13}
M.~P. Ueckermann, P.~F. Lermusiaux, and T.~P. Sapsis.
\newblock Numerical schemes for dynamically orthogonal equations of stochastic
  fluid and ocean flows.
\newblock {\em J. Comput. Phys.}, 233:272--294, 2013.

\bibitem{Wiener38}
N.~Wiener.
\newblock The homogeneous chaos.
\newblock {\em Am. J. Math.}, 60(4):897--936, 1938.

\bibitem{Wloka87}
J.~Wloka.
\newblock {\em Partial Differential Equations}.
\newblock Cambridge University Press, 1987.

\bibitem{Xiu02}
D.~Xiu and G.~E. Karniadakis.
\newblock The wiener--askey polynomial chaos for stochastic differential
  equations.
\newblock {\em SIAM J. Sci. Comput.}, 24(2):619--644, 2002.

\bibitem{Zeidler90}
E.~Zeidler.
\newblock {\em Linear monotone operators: Hilbert Space Methods and Linear
  Parabolic Differential Equations}.
\newblock Springer-Verlag, New York, 1990.

\end{thebibliography}
\bibliographystyle{abbrv}

\end{document}